\documentclass[a4paper,11pt]{article}

\usepackage{amssymb, amsmath,amsthm}
\usepackage{amsbsy}
\usepackage{mathtools}
\usepackage{cases}

\usepackage[english]{babel}
\usepackage{accents}
\usepackage{dsfont}
\usepackage{bm}
\usepackage{booktabs}
\usepackage{varwidth}
\usepackage{hyperref}
\hypersetup{
	colorlinks,
	linkcolor={red!50!black},
	citecolor={blue!50!black},
	urlcolor={blue!80!black}
}
\usepackage{etoolbox}
\apptocmd{\sloppy}{\hbadness 10000\relax}{}{}

\usepackage{enumerate}
\usepackage[shortlabels]{enumitem}

\usepackage{diagbox,multirow,arydshln}
\usepackage{array}
\usepackage{hhline}
\usepackage{makecell}
\usepackage{float}
\usepackage{framed}
\usepackage{varwidth}
\usepackage[export]{adjustbox}
\usepackage{colortbl}
\usepackage{sidecap}
\usepackage[section]{placeins}
\usepackage{framed}
\usepackage{caption}
\usepackage{placeins}

\usepackage{imakeidx}

\usepackage{algorithm}
\usepackage{algpseudocode}

\usepackage{graphics}
\usepackage{epsfig}
\usepackage{graphicx}
\usepackage{epstopdf}
\usepackage{color}
\usepackage{subfigure}

\usepackage[dvipsnames,x11names]{xcolor}
\usepackage{xparse}
\usepackage {tikz}
\usetikzlibrary{matrix,chains,positioning,decorations.pathreplacing,arrows}
\usetikzlibrary{cd}
\usetikzlibrary{decorations.pathmorphing,backgrounds,fit}
\usetikzlibrary{decorations.markings}
\usepackage{tikz-network}
\usetikzlibrary{calc}

\colorlet{red2}{red!65!black}
\definecolor{palecerulean}{RGB}{137, 193, 206}
\definecolor{alizarin}{rgb}{0.82, 0.1, 0.26}
\definecolor{amber}{rgb}{1.0, 0.49, 0.0}
\definecolor{harvardcrimson}{rgb}{0.79, 0.0, 0.09}
\definecolor{lava}{rgb}{0.81, 0.06, 0.13}
\definecolor{lightgreen}{rgb}{0.56, 0.93, 0.56}
\definecolor{azure}{rgb}{0.0, 0.5, 1.0}
\definecolor{darkpastelgreen}{rgb}{0.01, 0.75, 0.24}
\definecolor{dartmouthgreen}{rgb}{0.05, 0.5, 0.06}
\definecolor{amethyst}{rgb}{0.6, 0.4, 0.8}
\definecolor{arrow}{RGB}{69, 114, 195}
\definecolor{last-arrow}{RGB}{190,120, 68}
\newcommand{\correct}[1]{{\color{black}{#1}}}

\usepackage[auth-sc,affil-sl]{authblk}

\usepackage{todonotes}

\usepackage[style=numeric,backend=bibtex,giveninits=true,maxbibnames=99,sortcites=true]{biblatex} 
\usepackage{ifthen}
\usepackage{float}
\floatstyle{ruled}
\newfloat{algorithm}{htb}{alg}
\floatname{algorithm}{Algorithm}

\newtheorem{theorem}{Theorem}[section]
\newtheorem{definition}[theorem]{Definition}
\newtheorem{proposition}[theorem]{Proposition}

\newtheorem{remark}[theorem]{Remark}

\numberwithin{equation}{section}

\newcommand{\R}{\mathbb{R}}

\newcommand{\A}{\mathcal{A}}
\newcommand{\Breg}{\mathcal{B}^{\mathcal{R}}}
\newcommand{\F}{\mathcal{F}}

\newcommand{\Reg}{\mathcal{R}}

\newcommand{\argmin}{\operatorname{argmin}}

\newcommand{\dom}{\operatorname{dom}}

\newcommand{\bb}{{\boldsymbol{b}}}

\newcommand{\bu}{{\boldsymbol{u}}}
\newcommand{\bv}{{\boldsymbol{v}}}
\newcommand{\bw}{{\boldsymbol{\omega}}}
\newcommand{\bx}{{\boldsymbol{x}}}
\newcommand{\by}{{\boldsymbol{y}}}
\newcommand{\bz}{{\boldsymbol{z}}}
\newcommand{\bxi}{{\boldsymbol{\xi}}}

\title{Uniformly convex neural networks and non-stationary iterated network Tikhonov (iNETT) method}

\author[]{Davide Bianchi}
\author[]{Guanghao Lai}
\author[]{Wenbin Li \thanks{\bf Corresponding Author: Wenbin Li}}

\affil[]{School of Science, Harbin Institute of Technology, Shenzhen, Shenzhen 518055, China. \textnormal{Email:} {\bf  bianchi@hit.edu.cn,  21s058002@stu.hit.edu.cn, liwenbin@hit.edu.cn}}

\date{}

\begin{document}
\maketitle
\begin{abstract}
We propose a non-stationary iterated network Tikhonov (iNETT) method for the solution of ill-posed inverse problems. The iNETT employs deep neural networks to build a data-driven regularizer, and it avoids the difficult task of estimating the optimal regularization parameter. To achieve the theoretical convergence of iNETT, we introduce uniformly convex neural networks to build the data-driven regularizer. Rigorous theories and detailed algorithms are proposed for the construction of convex and uniformly convex neural networks. In particular, given a general neural network architecture, we prescribe sufficient conditions to achieve a trained neural network which is component-wise convex or uniformly convex; moreover, we provide concrete examples of  realizing convexity and uniform convexity in the modern U-net architecture. With the tools of convex and uniformly convex neural networks, the iNETT algorithm is developed and a rigorous convergence analysis is provided. Lastly, we  show applications of the iNETT algorithm in 2D computerized tomography, where numerical examples illustrate the efficacy of the proposed algorithm.
\end{abstract}

\vspace{0.1cm}

\vspace{0.1cm}
\noindent\small\correct{MSC2020: 47A52; 65F22; 68T07.}
\section{Introduction}\label{sec:introduction}
Consider the discretized form of an ill-posed linear problem,
\begin{equation}\label{model_eq1}
F \bx  = \by,
\end{equation}	
where $X$ and $Y$ denote finite dimensional normed spaces, i.e. $X= \left(\R^N,\|\cdot\|_X\right)$, $Y= \left(\R^M,\|\cdot\|_Y\right)$, and $F : X \to  Y$ is the  discretization of an ill-posed linear operator $\F$. The inverse problem aims to recover $\bx$ from the observed data $\by^\delta$ contaminated by unknown error with bounded norm,
\[
\by^\delta = \by+\boldsymbol{\eta}, \quad \mathrm{where}\ \|\boldsymbol{\eta}\|_Y\leq \delta\,.
\]
As $\by^\delta$ is not necessarily in the range of $F$, i.e. $\by^\delta \notin \operatorname{Rg}(F)$, we consider a variational approach to solve the inverse problem,
\begin{equation}\label{model_eq2}
\bx^\delta\coloneqq \underset{\bx \in X}{\argmin} \|F\bx - \by^\delta  \|^2_Y\,.
\end{equation}

For the ill-posed inverse problem, regularization techniques should be introduced when solving equation (\ref{model_eq2}). The regularization aims to provide prior knowledge and improve stability of the solution. For example, a typical choice of regularization in imaging is the $\ell^p$-norm of $\bx$, with $p\geq 1$, which can be weighted by the Laplacian operator with Dirichlet or Neumann boundary condition \cite{andrews1977digital, engl1996regularization}. Recently, deep learning approaches are introduced to develop data-driven regularization terms in the solutions of inverse problems. In \cite{li2020nett}, the authors propose a network Tikhonov (NETT) approach, which combines deep neural networks with a Tikhonov regularization strategy. The general form of NETT can be summarized as follows,
\begin{equation}\tag*{NETT}\label{NETT}
 x^\delta_\alpha \coloneqq \underset{x \in\dom(\F)\cap \dom\left(\Phi_\Theta\right)}{\argmin} \A(\F x,y^\delta) + \alpha\psi\left(\Phi_{\Theta} (x)\right),
\end{equation}
where $\A(\F x,y^\delta)\geq 0$ is the data-fidelity term which measures misfits between the approximated and measurement data, $\alpha>0$ is a regularization parameter, and $\psi\left(\Phi_\Theta (x)\right)$ is the regularization (or \emph{penalty}) term including a neural network architecture $\Phi_{\Theta}$. In particular, $\psi$ is a nonnegative functional, and the neural network $\Phi_{\Theta}$ is trained to penalize artifacts in the recovered solution. By training $\Phi_{\Theta}$ in an appropriate way, the neural-network based regularization term is able to capture the feature of solution errors due to data noises and the inexact iterative scheme, so that it can provide penalization on the artifacts of solutions in an adaptive manner. This data-driven regularization strategy shows many advantages in solving inverse problems, and related studies can be found in \cite{antholzer2021discretization,antholzer2019nett,obmann2021augmented} and \cite{aggarwal2018modl} as well.

Motivated by \ref{NETT}, we propose an iterated network Tikhonov (iNETT) method which combines the data-driven regularization strategy with an iterated Tikhonov method. In a Tikhonov-like method as \ref{NETT}, the regularization parameter $\alpha$ plays an important role since it controls the trade-off between the data-fidelity term and the regularization term. The value of $\alpha$ relies on the noise level $\delta$, and it will affect the proximity of the recovered solution to the minimizer of the data-fidelity term. Poor choices of $\alpha$ can lead to very poor solutions, and it is well known that an accurate estimate of the optimal $\alpha$ is difficult to achieve and it typically relies on heuristic assumptions (e.g., \cite{hanke1996limitations,hansen1992analysis,reichel2013old}). As a result, a natural strategy is to consider an iterated Tikhonov method with non-stationary values of $\alpha$ in the iteration. The non-stationary iterated Tikhonov method is able to avoid exhaustive tuning of the regularization parameter, and it achieves better convergence rates in many applications. For example, we refer the readers to \cite{engl1987choice, hanke1998nonstationary, donatelli2012nondecreasing, bianchi2015iterated, cai2016regularization, bianchi2017generalized, buccini2017iterated} for the applications of iterated Tikhonov in Hilbert spaces, and \cite{osher2005iterative,bachmayr2009iterative,jin2012nonstationary,jin2014nonstationary} in Banach spaces.

Combining the strategy of neural-network based regularizer with the non-stationary iterated Tikhonov method, the  iNETT method has the following general form,
\begin{equation}\tag*{iNETT}\label{iterNETT}
	\begin{cases}
		\bx^\delta_n:= \underset{\bx \in X}{\argmin} \frac{1}{r}\|F \bx - \by^\delta\|_Y^r + \alpha_n\Breg_{\boldsymbol{\xi}^\delta_{n-1}}(\bx,\bx^\delta_{n-1}),\\
		\boldsymbol{\xi}^\delta_{n}:= \boldsymbol{\xi}^\delta_{n-1} - \frac{1}{\alpha_{n}} F^TJ_r\left(F\bx_n^\delta - \by^\delta\right),\\
		\bx_0 \in X, \, \boldsymbol{\xi}_0 \in \partial \Reg(\bx_0),
	\end{cases} 
\end{equation}
where $\Reg\coloneqq \Phi_{\Theta}^{uc} \colon X\to (\R, |\cdot|)$ is a uniformly convex neural network, $\Breg_{\bxi^\delta_{n-1}}(\cdot,\cdot)$ is the Bregman distance induced by $\Reg$ in the direction $\bxi^\delta_{n-1}\in \partial \Reg(\bx^\delta_{n-1})$, $J_r$ denotes the duality map for $r \in (1,\infty)$, and $\{\alpha_n\}_n$ is a sequence of positive real numbers. The value of $\alpha_n$ controls the amount of regularization, and it plays the role of regularization parameter. By taking a decreasing sequence of $\{\alpha_n\}_n$ and considering the standard discrepancy principle as the stopping rule, the \ref{iterNETT} algorithm can automatically determines the amount of regularization. We will provide the details of \ref{iterNETT} in Section \ref{sec:iNETT}, including a rigorous convergence analysis and many implementation details.

In the formula of \ref{iterNETT}, the neural network $ \Phi_{\Theta}^{uc}$ is employed to build the regularization term, and it is required to be uniformly convex. The property of uniform convexity is demanded in the convergence analysis of the iterated Tikhonov method \cite{jin2014nonstationary}. As a result, another important aspect of the paper is the modeling of convex and uniformly convex neural networks. In Section \ref{sec:notation}, we provide an exact mathematical modeling for the general architecture of neural networks. Our modeling can express the modern convolutional neural networks, where the operations like skip connection and concatenation are included. In Section \ref{sec:ucNN}, we propose rigorous theories for the convex and uniformly convex neural networks. Given a general neural network $\Phi_\Theta : X \to Z$, we prescribe sufficient conditions to obtain a related neural network which is component-wise convex or uniformly convex. The main idea comes from some recent works on convex neural networks, e.g. \cite{amos2017input, sivaprasad2021curious,tan2022data,mukherjee2020learned}, but we largely extend them to build modern architectures which can embrace state-of-the-art neural networks. In Section \ref{sec:example}, we provide particular examples of convex and uniformly convex U-net architectures. The U-net is a convolutional neural network widely used in image processing and related imaging science \cite{ronneberger2015u}. We give rigorous formulas for the U-net architecture, and explain the approaches to obtain convex and uniformly convex U-net architectures according to the general theories proposed in section \ref{sec:ucNN}. In Section \ref{sec:iNETT} and Section \ref{sec:numerical_examples}, we provide implementation details as we employ the convex U-net to build a uniformly convex  regularizer for the \ref{iterNETT} algorithm. The proposed method is successfully applied to computerized tomography in Section \ref{sec:numerical_examples}. The tool of convex and uniformly convex neural networks is actually a by-product when designing the \ref{iterNETT} algorithm, but it seems more interesting than the algorithm itself. The tool of convex neural networks shall have many interesting applications in the future study.

\section{Notation and setting}\label{sec:notation}
We collect here most of the notations and definitions we will use through this work. As main references, the reader can look at \cite{scherzer2009variational,zalinescu2002convex,rockafellar2009variational}. First of all, let us fix $X:= \left(\R^N, \|\cdot \|_X\right)$ and $Y:= \left(\R^M, \|\cdot \|_Y\right)$,  where $N,M \in \mathbb{N}$, and $\|\cdot\|_X$ and $\|\cdot\|_Y$ are some norms on $\R^N$ and $\R^M$, respectively. In the case of standard $\ell^p$ spaces, with $p\geq 1$, then we will indicate the corresponding norm with the usual notation $\|\cdot\|_p$. 

 We will indicate in bold any finite dimensional (column) vector, e.g. $\bx := (x_1,\ldots, x_N)^T \in \R^N$, where $T$ denotes the transpose operation, and we will use the notation $\bx\leq \hat{\bx}$ meaning that $x_i\leq \hat{x}_i$ for every $i=1,\ldots, N$.  With abuse of language, given a real-valued function $\sigma \colon \R^N \to \R$, we will say that $\sigma$ is \emph{monotone nondecreasing} if $\sigma(\bx)\leq \sigma(\hat{\bx})$ for every $\bx\leq \hat{\bx}$. In case of a function with multivariate output, $\sigma \colon \R^N \to \R^D$, we will indicate with $\sigma_d \colon \R^N \to \R$ its components, for $d=1,\ldots, D$.

 For a fixed $\bz \in Z\coloneqq(\R^D,\|\cdot\|_Z)$, we indicate with  $C(\bz,\cdot) : X \to Z\times X$ the ``concatenation'' operator, that is,
\begin{equation}\label{concatenation}
	C(\bz,\bx) \coloneqq (z_1,\ldots, z_D,x_1,\ldots,x_N)^T.
\end{equation}

Fix now a matrix $F : X \to Y$, which is the discretization of an ill-posed linear operator between Banach spaces. We will assume that, given the unperturbed and observed data $\by \in Y$ and  $\by^\delta\in Y$, respectively, then
\begin{equation}\tag{H0}\label{hypothesis0}
\by \in \operatorname{Rg}(F), \quad \mbox{that is}, \quad 	F\bx = \by \quad \mbox{is solvable},
\end{equation}
and
$$\by^\delta=\by+\boldsymbol{\eta}  \quad \mbox{where} \quad \|\boldsymbol{\eta}\|_Y\leq \delta.$$ 

We recall that  a Banach space $Y$ is uniformly smooth if its modulus of smoothness
\begin{equation*}
	\rho(\tau) \coloneqq \sup\left\{\frac{\|\by + \tau\hat{\by}\|+\|\by - \tau\hat{\by}\|}{2}-1  \mid \|\by\|=\|\hat{\by}\|=1\right\}, \quad \tau>0
\end{equation*} 
satisfies $\lim_{\tau\to 0^+}\rho(\tau)/\tau=0$.  Examples of uniformly smooth spaces are all the $\ell^p$-spaces for $p\in(1,\infty)$.

Given an extended real-valued function, $\Reg  \colon \dom(\Reg)\subseteq X \to (-\infty,+\infty]$, then $\Reg$ is \textit{uniformly convex} if there exists a nonnegative map $h\colon [0,+\infty)\to [0,+\infty]$ such that $h(s)=0$ if and only if $s=0$ and 
\begin{equation*}\label{eq:uni_convex}
	\Reg(t\bx + (1-t)\hat{\bx}) + t(1-t)h(\|\bx-\hat{\bx}\|_X) \leq t\Reg(\bx)+ (1-t)\Reg(\hat{\bx}), \qquad \forall t\in[0,1]\mbox{ and } \forall\bx,\hat{\bx} \in \dom(\Reg).
\end{equation*}

Finally, we recall that $\Reg$ is \emph{coercive} if it is bounded below on bounded sets and 
\begin{equation*}
	\liminf_{\|\bx\|_X\to \infty}\frac{\Reg(\bx)}{\|\bx\|_X}=\infty.
\end{equation*}

\subsection{Bregman distance} Given a convex function $$\Reg : \dom\left(\Reg\right)\subseteq X \to (-\infty,+\infty],$$ $\Reg$ is called \emph{proper} if $\dom\left(\Reg\right):= \left\{\bx \in X \, : \, \Reg(\bx)<+\infty\right\}\neq \emptyset$. For every $\hat{\bx} \in \dom\left(\Reg\right)$, a \emph{subgradient} of $\Reg$ at $\hat{\bx}$ is an element $\bxi$ of the dual space $X^*$ such that
\begin{equation*}
	  \Reg(\bx) - \Reg(\hat{\bx}) - \langle \bxi, \bx-\hat{\bx}\rangle \geq 0 \quad \forall \bx \in X,
\end{equation*}
where the bracket is the evaluation of $\bxi$ at $\bx-\hat{\bx}$. Clearly, since $X$ is finite dimensional, then $X$ is reflexive and $\langle \cdot, \cdot \rangle$ is the standard inner product, that is,
\begin{equation*}
	\langle \bxi, \bx-\hat{\bx}\rangle = \sum_{i=1}^N \xi_i(x_i-\hat{x}_i).
\end{equation*}
The collection of all subgradients of $\Reg$ at $\hat{\bx}$ is denoted by $\partial \Reg (\hat{\bx})$. The \emph{subdifferential} of $\Reg$ is the multi-valued map $\partial \Reg : \dom\left(\partial \Reg\right)\subseteq X \to 2^{X^*}$ such that
\begin{align*}
	& \dom\left(\partial \Reg\right):= \left\{\hat{\bx} \in \dom(\Reg) \, : \, \partial \Reg (\hat{\bx})\neq \emptyset \right\},\\
	& \hat{\bx} \mapsto \partial \Reg(\hat{\bx}).
\end{align*}
Let us recall that if $\dom(\Reg)=X$, then $\dom\left(\partial \Reg\right)=X$ (e.g. \cite[Lemma 3.16]{scherzer2009variational}).

Finally, for every $\hat{\bx} \in\dom\left(\partial \Reg\right)$ and $\bxi \in \partial \Reg (\hat{\bx})$, the \emph{Bregman distance} $\Breg_\bxi(\cdot,\hat{\bx}): X \to [0,+\infty)$ induced by $\Reg$ at $\hat{\bx}$ in the direction $\bxi$ is defined by
$$
\Breg_\bxi(\bx,\hat{\bx}):= \Reg(\bx) - \Reg(\hat{\bx}) - \langle \bxi, \bx-\hat{\bx}\rangle.
$$
\begin{remark}\label{rem:breg_coercive}
It is straightforward to check that if $\Reg$ is uniformly convex then $\Breg_\bxi(\cdot,\hat{\bx})$ is uniformly convex too, for any fixed $\bxi$ and $\hat{\bx}$. Moreover, since $X$ is reflexive, then $\Breg_\bxi(\cdot,\hat{\bx})$ is coercive. See for example \cite[Corollary 2.4]{zalinescu1983uniformly}
\end{remark}
We can now introduce the definition of solution of the model problem \eqref{model_eq1}, with respect to the Bregman distance from a reference initial guess.
\setcounter{theorem}{0}
\begin{definition}\label{def:R-minimizing_sol}
	Fix $\bx_0 \in \dom\left(\partial \Reg\right)$, $\bxi_0 \in\partial \Reg (\bx_0)$. 	An element $\bx^\dagger \in  \dom(\Reg)$ is called a $\Breg_{\bxi_0}$-\emph{minimizing} solution of \eqref{model_eq1} if $F\bx^\dagger=\by$ and 
	\begin{equation*}
		\Breg_{\bxi_0}(\bx^\dagger,\bx_0) = \min \left\{  \Breg_{\bxi_0}(\bx,\bx_0) \, : \, \bx \in \dom(\Reg),\, F\bx=\by \right\}.
	\end{equation*}
\end{definition}

 As a last piece of notation, we introduce the \emph{duality map}. For every fixed $r\in (1,\infty)$, the duality map $J_r \colon X \to 2^{X^*}$ is given by
 \begin{equation*}
 	J_r(\bx) \coloneqq \{ \bxi \in X^* \mid \|\bxi\|=\|\bx\|_X^{r-1} \mbox{ and } \langle \bxi, \bx\rangle =\|\bx\|_X^r \}.
 \end{equation*}
 In particular, $J_r$ is the subdifferential of the map $\bx\mapsto \frac{\|\bx\|_X^r}{r}$. 
 
 \begin{remark}\label{rem:J_r}
  If $X$ is an $\ell^p$ space, then $J_r$ is single-valued and for $r=2$ it holds 
 \begin{equation*}
 	J_2(\bx) = \operatorname{sgn}(\bx)\frac{|\bx|^{p-1}}{\|\bx\|_2^{p-2}}.
 \end{equation*} 
 More generally, if a Banach space is uniformly smooth, then $J_r$ is single-valued.
 \end{remark}
 In view of the above remark and for the well-posedness of the \ref{iterNETT} method (see Section \ref{sec:iNETT}), we will assume that
 \begin{equation}\tag{H1}\label{hypothesis1}
 	Y \mbox{ is uniformly smooth}.
 \end{equation}

\subsection{Neural networks}\label{sec:DNN}

A neural network is a chain of compositions of  affine operators and  nonlinear operators. For an introduction to neural networks from an applied mathematical point of view, we refer to \cite{higham2019deep}, whereas we refer to \cite{arridge2019solving} for a focus on deep learning techniques for inverse problems.  We present here the basic architecture upon which we will devise the neural networks to be implemented in \ref{iterNETT}. There are several many choices for the linear and the nonlinear operators, and each of them generate a different neural network. We do not focus now on those choices, which will be made only later (see  Sections \ref{sec:example}, \ref{subsec:CovRegul} and \ref{sec:numerical_examples}), to keep here a more general setting.

First, fix a \emph{set of parameters} $\Theta:=\{\bb^k;A^{k,j_k};W^k\}_{k=1}^{L},$ where $\bb^k$ are vectors commonly called \textit{bias} terms, and $A^{k,j_k}$ and $W^k$ are matrices. Second, fix a collection $\{\sigma^k\}_{k=1}^{L}$ of possibly nonlinear operators.

Finally, define  $\Phi_\Theta: X \to Z$ such that 
\begin{equation}\tag{NN}\label{DNN}
	\begin{cases}
		\Phi_\Theta(\bx)\coloneqq \bz \in Z &\mbox{where } \bz= \bz^{L+1},\\
		\bz^{k+1}= \sigma^{k}\left(\hat{\bb}^{k} +  W^{k}\hat{\bz}^{k}\right) \in Z^{k+1} &\mbox{for }k=1,\ldots,L,\\
		 \hat{\bz}^{k}= C^k(\bz^k)\coloneqq 
		\begin{cases}
			\bz^k \mbox{ or} \\
			C(\bz^{i_k},\bz^k)
		\end{cases} & \mbox{for } i_k\in\{1,\ldots,k\},\\
	\hat{\bb}^{k}	= \bb^{k} + A^{k,j_{k}}\bz^{j_{k}} & \mbox{for } j_{k}\in \{1,\ldots, k\},\\
		\bz^1\coloneqq \bx \in X, 
	\end{cases}
\end{equation} 
where $Z^{L+1}=Z\coloneqq(\R^D,\|\cdot\|_Z)$ and $Z^k\coloneqq (\R^{D_k},\|\cdot\|_{Z_k})$ are  finite dimensional normed vector spaces, and $C(\cdot,\cdot)$ is the concatenation operator \eqref{concatenation}. 
The operators $C(\bz^{i_{k}},\cdot)$ and $A^{k,j_k}$  represent the \emph{skip connections}, that is, some of the data in the previous iterations are used in future iterations, skipping intermediate steps. 
See Figure \ref{fig:DNN} for a visual representation. The integer $L$ is referred to as the \emph{depth} of the neural network, and $\sigma^{k}\left(\hat{\bb}^{k} +  W^{k}\hat{\bz}^{k}\right)$ as the $k$-th \emph{layer}. When $L>2$, then the neural network is commonly called \emph{deep neural network}. 

The set $\Theta$ is made by the disjoint union of two subsets, the set of \emph{free parameters} $\Theta_{free}$ and the set of \emph{frozen parameters} $\Theta_{frozen}$, that is
$$
\Theta= \Theta_{free}\sqcup \Theta_{frozen}.
$$
The set of free parameters $\Theta_{free}$ is typically initialized to a starting set of values and then it is \textit{trained}  by minimizing a loss function over a training sample.  Vice-versa, the set of frozen parameters $\Theta_{frozen}$ is fixed and unaffected by the training process. For example, some of the matrices $W^k$ can be fixed to be the identity matrix $I$ or the bias terms $\bb^k$ to be the zero vector $\boldsymbol{0}$. $\Theta_{frozen}$ can be empty, that is, all the parameters are trainable. About the specific training strategy we will employ, see Subsection \ref{subsec:CovRegul}.

In the case that $A^{k,j_k}$ is fixed to be the zero matrix and $\hat{\bz}^{k}= \bz^{k}$, for every $k$, then we have a \emph{feedforward} neural network. For examples of simple architectures of feedforward neural networks of convolutional type, see \cite{chen2017low,morotti2021green}.  For examples of more involved neural networks described by \eqref{DNN}, see ResNet \cite{he2016deep}, DenseNet \cite{huang2017densely} and U-Net \cite{ronneberger2015u}.

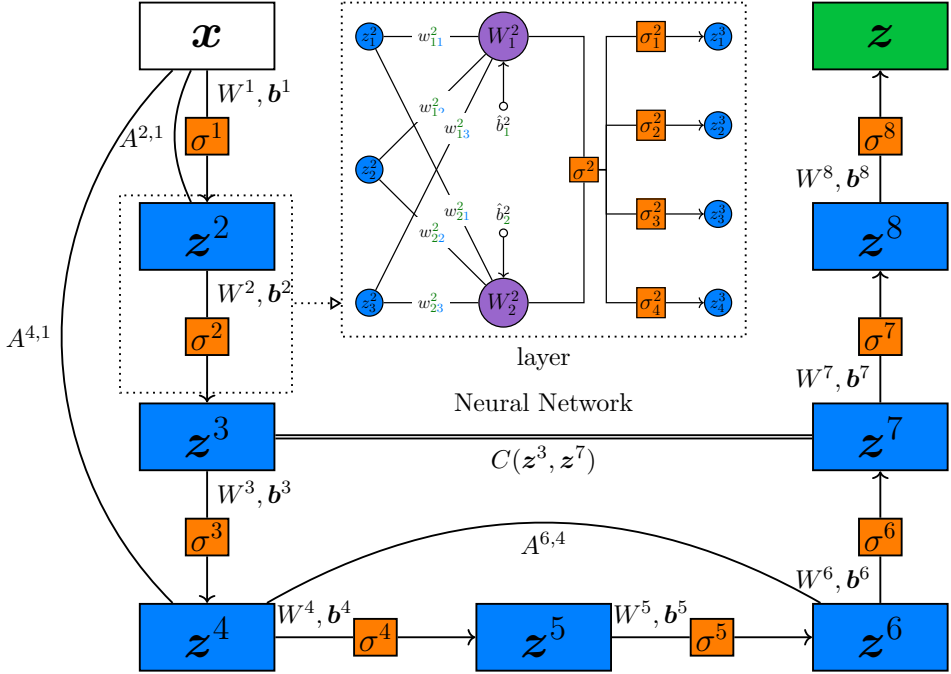
\begin{figure}[!tbh]
		\begin{adjustbox}{max totalsize={.9\textwidth}{.9\textheight},center}%
		\begin{tikzpicture}[auto,node distance =3 cm ,on grid ,
			thick ,
			square1/.style={
				draw,
				inner sep=2pt,
				font=\Large,
				top color = amber, 
				bottom color = amber
			},
			square2/.style={
				draw,
				minimum width=2cm,
				minimum height=1cm,
				font=\huge,
				top color = azure, 
				bottom color = azure
			},
			input/.style ={draw, circle, top color = darkpastelgreen, bottom color =darkpastelgreen, black, text=black,inner sep=0.5pt},
			ghost/.style ={ circle ,top color =white , bottom color = white ,
				draw,white , text=white , minimum width =1 cm},
			bias/.style ={ circle ,top color =white , bottom color = white ,
				draw, white , text=black , inner sep=0.5pt},
			output/.style ={draw, circle, top color = amethyst, bottom color =amethyst, black, text=black,inner sep=0.5pt},
			state/.style={draw,
				circle,
				inner sep=2pt,
				font=\huge,
				top color = azure, 
				bottom color = azure, 		
				draw,
				black,
				text=black}]

			\node[square2, top color=white, bottom color=white] (0) [] {$\bx$};
			\node[square1] (1/2) [below =1.5cm of 0] {$\sigma^1$};
			\node[square2] (1) [below =1.5cm of 1/2] {$\bz^2$};
			\node[square1] (3/2) [below =1.5cm of 1] {$\sigma^2$};
			\node[square2] (2) [below =1.5cm of 3/2] {$\bz^3$};
			\node[square1] (5/2) [below =1.5cm of 2] {$\sigma^3$};
			\node[square2] (3) [below =1.5cm of 5/2] {$\bz^4$};
			\node[square1] (7/2) [right =2.5cm of 3] {$\sigma^4$};
			\node[square2] (4) [right =2.5cm of 7/2] {$\bz^5$};
			\node[square1] (9/2) [right =2.5cm of 4] {$\sigma^5$};
			\node[square2] (5) [right =2.5cm of 9/2] {$\bz^6$};
			\node[square1] (11/2) [above =1.5cm of 5] {$\sigma^6$};
			\node[square2] (6) [above =1.5cm of 11/2] {$\bz^7$};
			\node[square1] (13/2) [above =1.5cm of 6] {$\sigma^7$};
			\node[square2] (7) [above =1.5cm of 13/2] {$\bz^8$};
			\node[square1] (15/2) [above =1.5cm of 7] {$\sigma^8$};
			\node[square2, top color=darkpastelgreen, bottom color=darkpastelgreen] (8) [above =1.5cm of 15/2] {$\bz$};
			
			\path (0) edge [black,-] node[right] {$W^1, \bb^1$} (1/2);
			\path (1/2) edge [black,->] node[right] {} (1);
			\path (1) edge [black,-] node[right] {$W^2, \bb^2$} (3/2);
			\path (3/2) edge [black,->] node[right] {} (2);
			\path (2) edge [black,-] node[right] {$W^3, \bb^3$} (5/2);
			\path (5/2) edge [black,->] node[right] {} (3);
			
			\path (0) edge [black,-,bend right=25] node[left] {$A^{2,1}$} (1);
			\path (0) edge [black,-,bend right=45] node[left] {$A^{4,1}$} (3);

			\path (3) edge [black,-] node[above] {$W^4, \bb^4$} (7/2);
			\path (7/2) edge [black,->] node[below] {} (4);
			\path (4) edge [black,-] node[above] {$W^5,\bb^5$} (9/2);
			\path (9/2) edge [black,->] node[below] {} (5);
			
			\path (5) edge [black,-] node[left] {$W^6, \bb^6$} (11/2);
			\path (11/2) edge [black,->] node[left] {} (6);
			\path (6) edge [black,-] node[left] {$W^7, \bb^7$} (13/2);
			\path (13/2) edge [black,->] node[left] {} (7);
			\path (7) edge [black,-] node[left] {$W^8, \bb^8$} (15/2);
			\path (15/2) edge [black,->] node[left] {} (8);
			
			\path (3) edge [black,-, bend left= 30] node[below] {$A^{6,4}$} (5);
			
			\path (2) edge [black,double,-] node[below] {$C(\bz^3,\bz^7)$} (6);

			\node at (5,-2) [] (layer) {
				\begin{adjustbox}{max totalsize={.35\textwidth}{.35\textheight},center}
					\begin{tikzpicture}[auto,node distance =3 cm ,on grid ,
						thick ,
						square/.style={
							draw,
							inner sep=2pt,
							font=\Large,
							top color = amber, 
							bottom color = amber
						},
						input/.style ={draw, circle, top color = azure, bottom color =azure, black, text=black,inner sep=0.5pt},
						ghost/.style ={ circle ,top color =white , bottom color = white ,
							draw,white , text=white , minimum width =1 cm},
						bias/.style ={circle,top color =white , bottom color = white ,
							draw, white , text=black , inner sep=0.5pt},
						output/.style ={draw, circle, top color = azure, bottom color =azure, black, text=black,inner sep=0.5pt},
						state/.style={draw,
							circle,
							inner sep=2pt,
							font=\Large,
							top color = amethyst, 
							bottom color = amethyst, 		
							draw,
							black,
							text=black}]
						
						\node[state] (E) [] {$W^2_1$};
						\node[ghost] (F) [below =of E] {};
						\node[state] (G) [below =of F] {$W^2_2$};
						
						\node[input] (H) [left =of F] {$z^2_2$};
						\node[input] (I) [above =of H] {$z^2_1$};
						\node[input] (L) [below =of H] {$z^2_3$};
						
						\node[ghost] (A) [right =1.5cm of E] {};
						\node[ghost] (B) [right =1.5cm of G] {};

						\node[ghost] (C) [right =0.3cm of A] {};
						\node[square] (D) [below = of C] {$\sigma^2$};
						
						\node[square] (S1) [right =1.5cm of C] {$\sigma^2_1$};
						\node[square] (S2) [below =2cm of S1] {$\sigma^2_2$};
						\node[square] (S3) [below =2cm of S2] {$\sigma^2_3$};
						\node[square] (S4) [below =2cm of S3] {$\sigma^2_4$};

						\node[output] (Y1) [right =1.5cm of S1] {$z^3_1$};
						\node[output] (Y2) [below =2cm of Y1] {$z^3_2$};
						\node[output] (Y3) [below =2cm of Y2] {$z^3_3$};
						\node[output] (Y4) [below =2cm of Y3] {$z^3_4$};
						
						\node[bias] (Eb) [below  =2cm of E] {$\hat{b}^2_{\textcolor{dartmouthgreen}{1}}$};
						\node[bias] (Gb) [above  =2cm of G] {$\hat{b}^2_{\textcolor{dartmouthgreen}{2}}$};
						
						\path (H) edge [black,-] node[fill=white, anchor=center, pos=0.5,font=\bfseries] {$w^{\textcolor{dartmouthgreen}{2}}_{\textcolor{dartmouthgreen}{1}\textcolor{azure}{2}}$} (E);
						\path (H) edge [black,-] node[fill=white, anchor=center, pos=0.5,font=\bfseries] {$w^{\textcolor{dartmouthgreen}{2}}_{\textcolor{dartmouthgreen}{2}\textcolor{azure}{2}}$} (G);
						
						\path (I) edge [black,-] node[fill=white, anchor=center, pos=0.5,font=\bfseries] {$w^{\textcolor{dartmouthgreen}{2}}_{\textcolor{dartmouthgreen}{1}\textcolor{azure}{1}}$} (E);
						\path (I) edge [black,-] node[fill=white, anchor=center, pos=0.70,font=\bfseries] {$w^{\textcolor{dartmouthgreen}{2}}_{\textcolor{dartmouthgreen}{2}\textcolor{azure}{1}}$} (G);
						
						\path (L) edge [black,-] node[fill=white, anchor=center, pos=0.70,font=\bfseries] {$w^{\textcolor{dartmouthgreen}{2}}_{\textcolor{dartmouthgreen}{1}\textcolor{azure}{3}}$} (E);
						\path (L) edge [black,-] node[fill=white, anchor=center, pos=0.5,font=\bfseries] {$w^{\textcolor{dartmouthgreen}{2}}_{\textcolor{dartmouthgreen}{2}\textcolor{azure}{3}}$} (G);

						\draw[black,-] (E) -| (D);
						\draw[black,-] (G) -| (D);

						\draw[black,-] (D) -- ($(D)+(0.5,0)$) |-  (S1);
						\draw[black,-] (D) -- ($(D)+(0.5,0)$) |-  (S2);
						\draw[black,-] (D) -- ($(D)+(0.5,0)$) |-  (S3);
						\draw[black,-] (D) -- ($(D)+(0.5,0)$) |-  (S4);
						
						\path (S1) edge [black,->] node[] {} (Y1);
						\path (S2) edge [black,->] node[] {} (Y2);
						\path (S3) edge [black,->] node[] {} (Y3);
						\path (S4) edge [black,->] node[] {} (Y4);

						\path (Eb) edge [black,o->] node[] {} (E);
						\path (Gb) edge [black,o->] node[] {} (G);
					\end{tikzpicture}
				\end{adjustbox}
			};
			
			\draw[thick,dotted]     ($(1.north west)+(-0.28,0.10)$) rectangle ($(3/2.south east)+(0.93,-0.55)$);
			
			\draw[thick,dotted]     ($(5,-2)+(-3,2.5)$) rectangle ($(5,-2)+(3,-2.5)$);
			\node[below=0.5pt] (S) at (layer.south) {layer};
			\draw[thick,dotted,-{Triangle[open]}] ($(1)+(1.3,-1)$) -- ($(1)+(+2,-1)$);

			\node[below=0.7cm] (S) at (layer.south) {Neural Network};	
			
		\end{tikzpicture}
	\end{adjustbox}
	\caption{Example of a simple neural network architecture of depth $L=8$, with input $\bx$ and output $\bz$ described by the white and green rectangles, respectively. As it can be visually checked, the operators $A^{2,1}$, $A^{4,1}$ $A^{6,4}$, and $C(\bz^3,\bz^7)$ bring over data of previous iterations skipping some connections. \correct{To differentiate the concatenation $C(\cdot,\cdot)$ from the other matrix-vector operators, we use a double edge to denote it.} Inside the dotted square, it is depicted a blow-up of the second layer: The rows of the matrix $W^2=(w^2_{ij})\in \R^{2\times 3}$ are represented by the violet circles. Note that in that layer the bias term $\hat{\bb}^2$ is given by $\hat{\bb}^2=\bb^2+ A^{2,1}\bx$.}\label{fig:DNN}
\end{figure}

\section{Convex and uniformly convex neural networks}\label{sec:ucNN}

We provide some sufficient conditions to guarantee that a neural network, with the architecture described by \eqref{DNN} in Subsection \ref{sec:DNN}, is component-wise convex or component-wise uniformly convex. In particular, given any neural network of depth $L$, it is possible to generate a component-wise uniformly convex neural network of depth $L+1$ which embeds the original neural network.

Fix a neural network of depth $L$, $\Phi_\Theta: X \to Z$,  and consider the following constraints:
\begin{align}\tag{C1}\label{hypothesis2}
 	&	\sigma_d^k \mbox{ is continuous, convex and monotone nondecreasing} \quad  \forall d=1,\ldots,D_k,\, \forall k=1,\dots,L;\\
 	\tag{C2}\label{hypothesis3}
 	&\begin{cases}
 		W^k \mbox{ has nonnegative entries } \forall k=2,\dots,L,\\
 		A^{k,j_k} \mbox{ has nonnegative entries } \forall j_k\geq 2.
 	\end{cases}
 \end{align}
 In the next proposition we show that $\Phi_{\Theta}$ is component-wise convex. This result is a generalization of the input convex neural networks presented in \cite{amos2017input}, but we largely extend them to build modern architectures which can embrace state-of-the-art neural networks.

\begin{theorem}\label{prop:DNN_convex}
Under the constraints \eqref{hypothesis2} and	\eqref{hypothesis3}, any neural network described by the set of equations \eqref{DNN} is continuous and component-wise convex.  
\end{theorem}
\begin{proof}
$\Phi_{\Theta}$ is trivially continuous because composition of continuous  functions. We want to show that given $\Phi_\Theta=(\Phi_{\Theta,1},\ldots, \Phi_{\Theta,D})^T$, then
\begin{equation*}
	\Phi_{\Theta,d}(t\bx +(1-t)\bw)\leq t\Phi_{\Theta,d}(\bx) + (1-t)\Phi_{\Theta,d}(\bw) \quad \mbox{for every }d=1,\ldots,D \mbox{ and }\bx,\bw\in X.
\end{equation*}
We proceed by induction. Fix $L=1$ and $\bz^1\coloneqq  t\bx +(1-t)\bw \in X$: Then, by \eqref{DNN},  it holds that
\begin{equation*}
	\Phi_{\Theta,d}(\bz^1)=\sigma^1_d\left(\hat{\bb}^{1} +  W^{1}\hat{\bz}^1\right)= \sigma^1_d\left(\bb^{1} + A^{1,1}\bz^1 +  W^{1}C^1(\bz^1)\right). 
\end{equation*}  
Let us remember that $C^1$ is such that
$$
\begin{cases}
	C^1 : X \to X,\\
	C^1(\bx)=\bx,
\end{cases}\quad \mbox{or} \quad \begin{cases}
C^1 : X \to X\times X,\\
C^1(\bx)=C(\bx,\bx).
\end{cases}
$$
In the latter case, 
\begin{align*}
	C^1(\bz^0)&= (t\bx +(1-t)\bw, t\bx +(1-t)\bw)^T\\
	&= t(\bx,\bx)^T+ (1-t)(\bw,\bw)^T\\
	&=tC(\bx,\bx)+(1-t)C(\bw,\bw)\\
	&=tC^1(\bx) + (1-t)C^1(\bw).
\end{align*}
Then, by the linearity of $A^{1,1}$ and $W^1$, and independently of the choice of $C^1$, it is clear that
\begin{align*}
	\bb^{1} + A^{1,1}\bz^1 +  W^{1}C^1(\bz^1)&= t(\bb^{1} + A^{1,1}\bx +  W^1C^1(\bx)) + (1-t)(\bb^{1} + A^{1,1}\bw + W^1C^1(\bw))\\
	&= t\bu + (1-t)\bv, 
\end{align*}
where
$$
\bu\coloneqq\bb^{1} + A^{1,1}\bx +  W^1C^1(\bx), \qquad \bv\coloneqq\bb^{1} + A^{1,1}\bw + W^1C^1(\bw).
$$
Using now the convexity of $\sigma^1_d$, we obtain 
\begin{align*}
	\Phi_{\Theta,d}(t\bx +(1-t)\bw) &=\sigma^1_d\left(\bb^{1} + A^{1,1}\bz^1 +  W^{1}C^1(\bz^1)\right)\\ &= \sigma^1_d\left( t\bu + (1-t)\bv\right)\\
	&\leq t\sigma^1_d(\bu) + (1-t)\sigma^0_d(\bv)\\
	&= t\Phi_{\Theta,d}(\bx) + (1-t)\Phi_{\Theta,d}(\bw),
\end{align*}
namely, $\Phi_\Theta$ is component-wise convex. Assume now that every neural network $\Phi_{\Theta}$ described by \eqref{DNN} is component-wise convex when the depth $L$ is $L=1,\ldots, n-1$, for any fixed $n\geq2$. We need to prove that $\Phi_{\Theta}$ is component-wise convex when $L=n$. Fix again $\bz^1\coloneqq t\bx +(1-t)\bw$. Observe that, at each step $k=1,\ldots,L-1$, every $\bz^{k+1}$ in \eqref{DNN} is a neural network $\Phi_{\Theta}^{k} : X \to Z^{k+1}$ of depth $L_{k}\leq L-1$ and with the same initial input $\bz^1$, that is $\bz^{k+1}=\Phi_{\Theta}^{k}(\bz^1)$. Then, by the induction hypothesis,
\begin{equation}\label{eq:induction}
\Phi_{\Theta}^{k}(\bz^1)=\Phi_{\Theta}^{k}(t\bx +(1-t)\bw)\leq t\Phi_{\Theta}^{k}(\bx) +(1-t)\Phi_{\Theta}^{k}(\bw) \qquad \forall\, k=1,\ldots,L-1.
\end{equation}
So,
\begin{align*}
	\Phi_{\Theta,d}(\bz^1)=\bz^{L+1}&= \sigma^{L}_d\left(\hat{\bb}^{L} +  W^{L}\hat{\bz}^{L}\right)\\
	&= \sigma^{L}_d\left(\bb^{L} + A^{L,j_{L}}\bz^{j_{L}} +  W^{L}C^{L}(\bz^{L})\right),\\
	&= \sigma^{L}_d\left(\bb^{L} + A^{L,j_{L}}\Phi_{\Theta}^{j_{L}-1}(\bz^1) +  W^{L}C^{L}(\Phi_{\Theta}^{L-1}(\bz^1))\right).
\end{align*} 
We need to take care of the concatenation operator. Again, there are two cases: $$
\begin{cases}
	C^{L}\colon  Z^{L} \to Z^{L},\\
C^{L}(\Phi_{\Theta}^{L-1}(\bz^1)) = \Phi_{\Theta}^{L-1}(\bz^1),	
\end{cases}\quad \mbox{or} \quad 
\begin{cases}
C^{L}\colon  Z^{L} \to Z^{i_{L}}\times Z^{L},\\
C^{L}(\Phi_{\Theta}^{L-1}(\bz^1))=C(\Phi_{\Theta}^{i_{L}-1}(\bz^1),\Phi_{\Theta}^{L-1}(\bz^1)),
\end{cases}
$$
where $i_{L} \in \{1,\ldots,L\}$, and with the convention that $\Phi_{\Theta}^0(\bz^1)=\bz^1$. Clearly, 
$$
C(\bu_1,\bu_2)\leq C(\bv_1,\bv_2) \quad \mbox{for every}\quad \bu_1\leq \bv_1,\;  \bu_2\leq \bv_2.
$$
Therefore, if we are in the latter case, by  \eqref{eq:induction}
\begin{align*}
	C(\Phi_{\Theta}^{i_{L-1}}(\bz^1),\Phi_{\Theta}^{L-1}(\bz^1)) &= C(\Phi_{\Theta}^{i_{L-1}}(t\bx +(1-t)\bw),\Phi_{\Theta}^{L-1}(t\bx +(1-t)\bw))\nonumber\\
	&\leq C(t\Phi_{\Theta}^{i_{L-1}}(\bx) + (1-t)\Phi_{\Theta}^{i_{L-1}}(\bw),t\Phi_{\Theta}^{L-1}(\bx) + (1-t)\Phi_{\Theta}^{L-1}(\bw))\nonumber\\
	&=tC(\Phi_{\Theta}^{i_{L-1}}(\bx),\Phi_{\Theta}^{L-1}(\bx)) + (1-t)C(\Phi_{\Theta}^{i_{L-1}}(\bw),\Phi_{\Theta}^{L-1}(\bw))\\
	&= tC^{L}(\Phi_{\Theta}^{L-1}(\bx)) +(1-t)C^{L}(\Phi_{\Theta}^{L-1}(\bw)).
\end{align*}
As a consequence, independently from the case we are in,
\begin{equation*}
C^{L}(\Phi_{\Theta}^{L-1}(\bz^1)) \leq tC^{L}(\Phi_{\Theta}^{L-1}(\bx)) +(1-t)C^{L}(\Phi_{\Theta}^{L-1}(\bw)),	
\end{equation*} 
and then, by \eqref{hypothesis3},
\begin{equation}\label{prop1:eq1}
	W^{L}C^{L}(\Phi_{\Theta}^{L-1}(\bz^1))\leq tW^{L}C^{L}(\Phi_{\Theta}^{L-1}(\bx)) + (1-t)W^{L}C^{L}(\Phi_{\Theta}^{L-1}(\bw)).
\end{equation}
Now, instead, if $j_{L}=1$, then $\Phi_{\Theta}^{j_{L}-1}(\bz^1)=\bz_1$ and by linearity
\begin{equation}\label{prop1:eq2}
 A^{L,j_{L}}\Phi_{\Theta}^{j_{L}-1}(\bz^1)=tA^{L,j_{L}}\bx + (1-t)A^{L,j_{L}}\bw.
\end{equation}
If $j_{L}\geq 2$, by \eqref{eq:induction} and \eqref{hypothesis3}, 
\begin{equation}\label{prop1:eq3}
 A^{L,j_{L}}\Phi_{\Theta}^{j_{L}-1}(\bz^1)\leq tA^{L,j_{L}}\Phi_{\Theta}^{j_{L}-1}(\bx) + (1-t)A^{L,j_{L}}\Phi_{\Theta}^{j_{L}-1}(\bw).
\end{equation}
Let us write
\begin{align*}
&\bu\coloneqq t(\bb^{L} + A^{L,j_{L}}\Phi_{\Theta}^{j_{L}-1}(\bx) +  W^{L}C^{L}(\Phi_{\Theta}^{L-1}(\bx))), \\
&\bv\coloneqq (1-t)(\bb^{L} + A^{L,j_{L}}\Phi_{\Theta}^{j_{L}-1}(\bw) +  W^{L}C^{L}(\Phi_{\Theta}^{L-1}(\bw))).
\end{align*}
Combining \eqref{prop1:eq1}, \eqref{prop1:eq2} and \eqref{prop1:eq3} with \eqref{hypothesis2}, we get 

\scalebox{0.9}{\parbox{.5\linewidth}{%
	\begin{align*}\Phi_{\Theta,d}(\bz^1)=\sigma^{L}_d\left(\bb^{L} + A^{L,j_{L}}\Phi_{\Theta}^{j_{L}-1}(\bz^1) +  W^{L}C^{L}(\Phi_{\Theta}^{L-1}(\bz^1))\right)&\leq \sigma^{L}_d(t\bu+(1-t)\bv)\\
	&\leq t\sigma^{L}_d(\bu) + (1-t)\sigma^{L}_d(\bv)\\
	&=t\Phi_{\Theta,d}(\bx) + (1-t)\Phi_{\Theta,d}(\bw).\end{align*}%
}
}

\noindent This concludes the proof.
\end{proof}

\begin{figure}[!tbh]
		\begin{adjustbox}{max totalsize={.8\textwidth}{.8\textheight},center}%
		\begin{tikzpicture}[auto,node distance =3 cm ,on grid ,
			thick ,
			square1/.style={
				isosceles triangle,
				isosceles triangle apex angle=60,
				draw,
				inner sep=1pt,
				top color = amber, 
				bottom color = amber
			},
			square2/.style={
				draw,
				minimum width=2cm,
				minimum height=1cm,
				font=\huge,
				top color = azure, 
				bottom color = azure
			},
			input/.style ={draw, circle, top color = darkpastelgreen, bottom color =darkpastelgreen, black, text=black,inner sep=0.5pt},
			ghost/.style ={ circle ,top color =white , bottom color = white ,
				draw,white , text=white , minimum width =1 cm},
			bias/.style ={ circle ,top color =white , bottom color = white ,
				draw, white , text=black , inner sep=0.5pt},
			output/.style ={draw, circle, top color = amethyst, bottom color =amethyst, black, text=black,inner sep=0.5pt},
			state/.style={draw,
				circle,
				inner sep=2pt,
				font=\huge,
				top color = azure, 
				bottom color = azure, 		
				draw,
				black,
				text=black}]

			\node[square2, top color=white, bottom color=white] (0) [] {$\boldsymbol{x}$};
			\node[square1, shape border rotate=270] (1/2) [below =1.5cm of 0] {$\sigma^1$};
			\node[square2] (1) [below =1.5cm of 1/2] {$\boldsymbol{z}^2$};
			\node[square1, shape border rotate=270] (3/2) [below =1.5cm of 1] {$\sigma^2$};
			\node[square2] (2) [below =1.5cm of 3/2] {$\boldsymbol{z}^3$};
			\node[square1, shape border rotate=270] (5/2) [below =1.5cm of 2] {$\sigma^3$};
			\node[square2] (3) [below =1.5cm of 5/2] {$\boldsymbol{z}^4$};
			\node[square1] (7/2) [right =2.5cm of 3] {$\sigma^4$};
			\node[square2] (4) [right =2.5cm of 7/2] {$\boldsymbol{z}^5$};
			\node[square1] (9/2) [right =2.5cm of 4] {$\sigma^5$};
			\node[square2] (5) [right =2.5cm of 9/2] {$\boldsymbol{z}^6$};
			\node[square1, shape border rotate=90] (11/2) [above =1.5cm of 5] {$\sigma^6$};
			\node[square2] (6) [above =1.5cm of 11/2] {$\boldsymbol{z}^7$};
			\node[square1, shape border rotate=90] (13/2) [above =1.5cm of 6] {$\sigma^7$};
			\node[square2] (7) [above =1.5cm of 13/2] {$\boldsymbol{z}^8$};
			\node[square1, shape border rotate=90] (15/2) [above =1.5cm of 7] {$\sigma^8$};
			\node[square2, top color=darkpastelgreen, bottom color=darkpastelgreen] (8) [above =1.5cm of 15/2] {$\boldsymbol{z}$};

			\path (0) edge [black,-,bend right=25] node[left] {$A^{2,1}$} (1);
			\path (0) edge [black,-,bend right=45] node[left] {$A^{4,1}$} (3);

			\path (0) edge [black,-] node[right] {$W^1, \bb^1$} (1/2);
			\path (1/2) edge [black,->] node[right] {} (1);
			\path (1) edge [black,-] node[right] {$W^2_{\geq0}, \bb^2$} (3/2);
			\path (3/2) edge [black,->] node[right] {} (2);
			\path (2) edge [black,-] node[right] {$W^3_{\geq0}, \bb^3$} (5/2);
			\path (5/2) edge [black,->] node[right] {} (3);

			\path (3) edge [black,-] node[above] {$W^4_{\geq0}, \bb^4$} (7/2);
			\path (7/2) edge [black,->] node[below] {} (4);
			\path (4) edge [black,-] node[above] {$W^5_{\geq0}, \bb^5$} (9/2);
			\path (9/2) edge [black,->] node[below] {} (5);
			
			\path (5) edge [black,-] node[left] {$W^6_{\geq0}, \bb^6$} (11/2);
			\path (11/2) edge [black,->] node[left] {} (6);
			\path (6) edge [black,-] node[left] {$W^7_{\geq0}, \bb^7$} (13/2);
			\path (13/2) edge [black,->] node[left] {} (7);
			\path (7) edge [black,-] node[left] {$W^8_{\geq0}, \bb^8$} (15/2);
			\path (15/2) edge [black,->] node[left] {} (8);

			\path (3) edge [black,-, bend left= 30] node[below] {$A^{6,4}_{\geq0}$} (5);

			\path (2) edge [black,double,-] node[below] {$C(\boldsymbol{z}^3,\bz^7)$} (6);
			
			\node at (5,-1) [] (layer) {};

			\node[below=4cm] (S) at (layer.south) {Convex Neural Network};
		\end{tikzpicture}
		\end{adjustbox}%
		\caption{Continuation of the example in Figure \ref{fig:DNN}. All the functions $\sigma^k$ are now assumed to be component-wise convex and monotone nondecreasing, and this is represented by the triangular-shaped nodes. Observe that the matrices $W^k$, for $k=2,\ldots,8$, and $A^{6,4}$ have nonnegative entries, while $W^1$, $A^{2,1}$, $A^{4,1}$ and the bias terms $\bb^k$ do not have any constraints.}\label{fig:DNN_connected}
\end{figure}
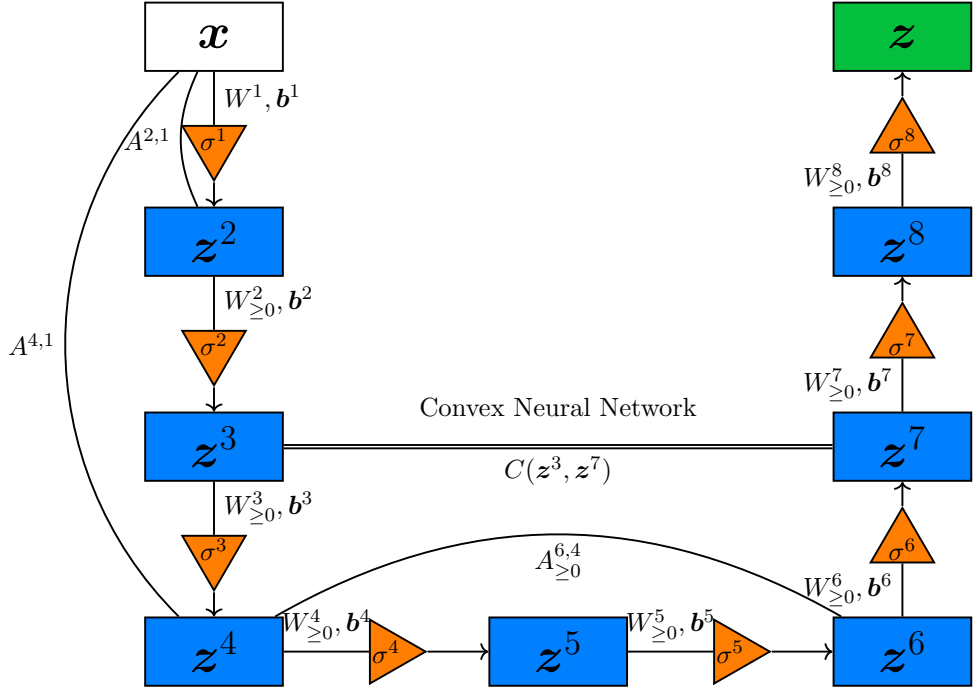

See Figure \ref{fig:DNN_connected} for a visual representation of a component-wise  convex neural network. Now, given a component-wise convex neural network $\Phi^c_{\Theta}\colon X \to Z$ of depth $L$,  define
\begin{equation}\tag{ucNN}\label{eq:ucNN}
\begin{cases}
\Phi^{uc}_\Theta: X \to Z',\\
\Phi^{uc}_\Theta (\bx)\coloneqq \Sigma_{\boldsymbol{f},\boldsymbol{g} }\left(\bx,\Phi^c_{\Theta}(\bx)\right),
\end{cases}	
\end{equation}
where $Z'\coloneqq (\R^{D'},\|\cdot\|_{Z'})$ and $\Sigma_{\boldsymbol{f},\boldsymbol{g}} \colon X\times Z \to Z'$ is given by
\begin{subequations}
\begin{equation*}
	\Sigma_{\boldsymbol{f},\boldsymbol{g} }(\bx,\bz) \coloneqq (f_1(\bx)+g_1(\bz),\ldots, f_{D'}(\bx)+g_{D'}(\bz))^T,
\end{equation*} 
and $f_d: X\to \R$, $g_d:Z \to \R$ are generic real-valued functions for $d=1,\ldots,D'$. In particular, $\Phi_{\Theta}^{uc}$ is  a  neural network $\Phi^{uc}_\Theta$ of depth $L+1$, whose last layer $\sigma^{L+1}(\hat{\bb}^{L+1}+W^{L+1}\hat{\bz}^{L+1})$ is given by
$$
\begin{cases}
\hat{\bz}^{L+1}=C(\bz^1,\bz^{L+1}),\quad	W^{L+1}=I, \quad \bb^{L+1}=A^{L+1,j_{L+1}}=\boldsymbol{0},\\
\sigma^{L+1}=\Sigma_{\boldsymbol{f},\boldsymbol{g}},
\end{cases}
$$
$$
W^{L+1}, \bb^{L+1}, A^{L+1,j_{L+1}} \in \Theta_{frozen}.
$$
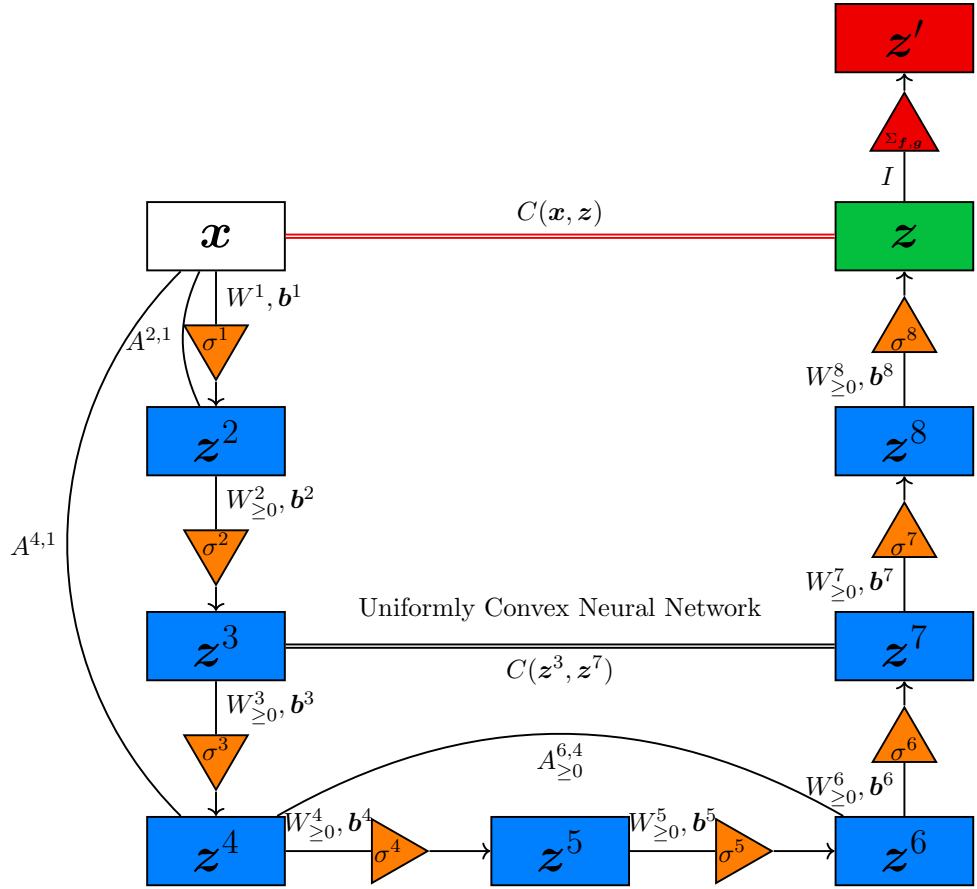
\begin{figure}[!b]
	\begin{adjustbox}{max totalsize={.8\textwidth}{.8\textheight},center}%
		\begin{tikzpicture}[auto,node distance =3 cm ,on grid ,
			thick ,
			square1/.style={
				draw,
				isosceles triangle,
				isosceles triangle apex angle=60,
				inner sep=1pt,
				top color = amber, 
				bottom color = amber
			},
			square1uc/.style={
				draw,
				isosceles triangle,
				isosceles triangle apex angle=60,
				inner sep=1pt,
				top color = Red2, 
				bottom color = Red2
			},
			square2/.style={
				draw,
				minimum width=2cm,
				minimum height=1cm,
				font=\huge,
				top color = azure, 
				bottom color = azure
			},
			square2uc/.style={
				draw,
				minimum width=2cm,
				minimum height=1cm,
				font=\huge,
				top color = Red2, 
				bottom color = Red2
			},
			input/.style ={draw, circle, top color = darkpastelgreen, bottom color =darkpastelgreen, black, text=black,inner sep=0.5pt},
			ghost/.style ={ circle ,top color =white , bottom color = white ,
				draw,white , text=white , minimum width =1 cm},
			bias/.style ={ circle ,top color =white , bottom color = white ,
				draw, white , text=black , inner sep=0.5pt},
			output/.style ={draw, circle, top color = amethyst, bottom color =amethyst, black, text=black,inner sep=0.5pt},
			state/.style={draw,
				circle,
				inner sep=2pt,
				font=\huge,
				top color = azure, 
				bottom color = azure, 		
				draw,
				black,
				text=black}]

			\node[square2, top color=white, bottom color=white] (0) [] {$\boldsymbol{x}$};
			\node[square1, shape border rotate=270] (1/2) [below =1.5cm of 0] {$\sigma^1$};
			\node[square2] (1) [below =1.5cm of 1/2] {$\boldsymbol{z}^2$};
			\node[square1, shape border rotate=270] (3/2) [below =1.5cm of 1] {$\sigma^2$};
			\node[square2] (2) [below =1.5cm of 3/2] {$\boldsymbol{z}^3$};
			\node[square1, shape border rotate=270] (5/2) [below =1.5cm of 2] {$\sigma^3$};
			\node[square2] (3) [below =1.5cm of 5/2] {$\boldsymbol{z}^4$};
			\node[square1] (7/2) [right =2.5cm of 3] {$\sigma^4$};
			\node[square2] (4) [right =2.5cm of 7/2] {$\boldsymbol{z}^5$};
			\node[square1] (9/2) [right =2.5cm of 4] {$\sigma^5$};
			\node[square2] (5) [right =2.5cm of 9/2] {$\boldsymbol{z}^6$};
			\node[square1, shape border rotate=90] (11/2) [above =1.5cm of 5] {$\sigma^6$};
			\node[square2] (6) [above =1.5cm of 11/2] {$\boldsymbol{z}^7$};
			\node[square1, shape border rotate=90] (13/2) [above =1.5cm of 6] {$\sigma^7$};
			\node[square2] (7) [above =1.5cm of 13/2] {$\boldsymbol{z}^8$};
			\node[square1, shape border rotate=90] (15/2) [above =1.5cm of 7] {$\sigma^8$};
			\node[square2, top color=darkpastelgreen, bottom color=darkpastelgreen] (8) [above =1.5cm of 15/2] {$\boldsymbol{z}$};
			\node[square1uc, shape border rotate=90] (17/2) [above =1.4cm of 8] {\tiny{$\Sigma_{\boldsymbol{f},\boldsymbol{g}}$}};
			\node[square2uc] (9) [above =1.5cm of 17/2] {$\boldsymbol{z}'$};
			
			\path (0) edge [black,-,bend right=25] node[left] {$A^{2,1}$} (1);
			\path (0) edge [black,-,bend right=45] node[left] {$A^{4,1}$} (3);
			\path (0) edge [Red2,double,-] node[above] {\textcolor{black}{$C(\bx,\bz)$}} (8);

			\path (0) edge [black,-] node[right] {$W^1, \boldsymbol{b}^1$} (1/2);
			\path (1/2) edge [black,->] node[right] {} (1);
			\path (1) edge [black,-] node[right] {$W^2_{\geq0}, \boldsymbol{b}^2$} (3/2);
			\path (3/2) edge [black,->] node[right] {} (2);
			\path (2) edge [black,-] node[right] {$W^3_{\geq0}, \boldsymbol{b}^3$} (5/2);
			\path (5/2) edge [black,->] node[right] {} (3);

			\path (3) edge [black,-] node[above] {$W^4_{\geq0}, \boldsymbol{b}^4$} (7/2);
			\path (7/2) edge [black,->] node[below] {} (4);
			\path (4) edge [black,-] node[above] {$W^5_{\geq0}, \boldsymbol{b}^5$} (9/2);
			\path (9/2) edge [black,->] node[below] {} (5);
			
			\path (5) edge [black,-] node[left] {$W^6_{\geq0}, \boldsymbol{b}^6$} (11/2);
			\path (11/2) edge [black,->] node[left] {} (6);
			\path (6) edge [black,-] node[left] {$W^7_{\geq0}, \boldsymbol{b}^7$} (13/2);
			\path (13/2) edge [black,->] node[left] {} (7);
			\path (7) edge [black,-] node[left] {$W^8_{\geq0}, \boldsymbol{b}^8$} (15/2);
			\path (15/2) edge [black,->] node[left] {} (8);
			\path (8) edge [black,-] node[left] {$I$} (17/2);
			\path (17/2) edge [black,->] node[left] {} (9);
			
			\path (3) edge [black,-, bend left= 30] node[below] {$A^{6,4}_{\geq0}$} (5);
			
			\path (2) edge [black,double,-] node[below] {$C(\boldsymbol{z}^3,\bz^7)$} (6);
			
			\node at (5,-1) [] (layer) {};

			\node[below=4cm] (S) at (layer.south) {Uniformly Convex Neural Network};
		\end{tikzpicture}
	\end{adjustbox}%
	\caption{Continuation of the examples in Figures \ref{fig:DNN} and \ref{fig:DNN_connected}. This architecture depicts a component-wise uniformly convex neural network $\Phi^{uc}_\Theta$ obtained from the component-wise convex neural network $\Phi^{c}_\Theta$ in Figure \ref{fig:DNN_connected}. The differences are depicted in red. Following equation \eqref{eq:ucNN}, it has been added a  concatenation operator $C(\bx,\bz)$ and a new layer after $\bz=\Phi^{c}_\Theta(\bx)$.}\label{fig:UCneural network}
\end{figure}
Consider the following assumptions, for every $d=1,\ldots,D'$:
\begin{align}
 &f_d \quad  \mbox{continuous and uniformly convex};\tag{C3}\label{hypothesis4}\\
 &g_d\quad  \mbox{continuous, convex and monotone nondecreasing on the image } \Phi^c_{\Theta}(X)\subseteq Z.\tag{C4}\label{hypothesis5}
\end{align}
\end{subequations}
\begin{theorem}\label{lem:DNN_uconvex}
Under the constraints \eqref{hypothesis4} and \eqref{hypothesis5}, the neural network $\Phi^{uc}_\Theta$ defined in \eqref{eq:ucNN} is component-wise uniformly convex. 
\end{theorem}
\begin{proof}
Given a component-wise convex map $\Phi^c_{\Theta} : X \to Z$, by direct computation it is immediate to check that $f_d + g_d\circ\Phi^c_{\Theta}$ is uniformly convex for every $d=1,\ldots, D'$, under the constraints \eqref{hypothesis4} and \eqref{hypothesis5}. 
\end{proof}

Combining Theorem \ref{prop:DNN_convex} and Theorem \ref{lem:DNN_uconvex}, given a neural network $\Phi_{\Theta} : X \to Z$, it is now easy to build a continuous and component-wise uniformly convex neural network $\Phi^{uc}_{\Theta} : X \to Z'$ by employing the following steps:

\begin{enumerate}[label={\upshape(\bfseries s\arabic*)},wide = 0pt, leftmargin = 3em]
	\item\label{step1} Fix a neural network $\Phi_{\Theta} : X \to Z$ whose architecture is given by \eqref{DNN}.
	\item\label{step2} Impose the constraints \eqref{hypothesis2}-\eqref{hypothesis3} to get  $\Phi^c_{\Theta} : X \to Z$.
	\item\label{step3} Define $\Phi^{uc}_{\Theta} : X \to Z'$ by equation \eqref{eq:ucNN} and impose the constraints \eqref{hypothesis4}-\eqref{hypothesis5}.
\end{enumerate}
See Figure \ref{fig:UCneural network} for a visual representation of a component-wise uniformly convex neural network.

\section{Example: Convex and uniformly convex U-net with batch normalization}\label{sec:example}
Here we provide particular examples of convex and uniformly convex U-net, which we will use in Subsection \ref{subsec:CovRegul} and Section \ref{sec:numerical_examples}. U-net is a convolutional neural network originally designed for biomedical image segmentation \cite{ronneberger2015u}. The network consists of a contracting path and a symmetric expanding path so that it has a U-shaped architecture. Several variants of the U-net architecture have been efficiently implemented in the solutions of inverse problems, e.g. \cite{ao2021data,heinrich2018residual,wei2018deep,antholzer2019deep,obmann2020deep}.   

We propose a modified U-net architecture as shown in Figure \ref{fig:U-net}, and we will explain the approaches to obtain convex and uniformly convex U-net architectures according to the general theories proposed in section \ref{sec:ucNN}.
In Subsection \ref{ssec:BN}, we give a brief description of the batch normalization which is one of the extra features we add to the modified U-net. In Subsection \ref{ssub:u-net}, we provide a mathematical modeling of the modified U-net architecture using the notations of general neural networks introduced in  Subsection \ref{sec:DNN}. Then in Subsections \ref{ssec:convex_U-net} and \ref{ssec:u_convex_U-net}, we explain the approaches to achieve convex and uniformly convex U-net architectures.


\begin{figure}[!bth]
	\begin{center}
		\includegraphics[width=\textwidth]{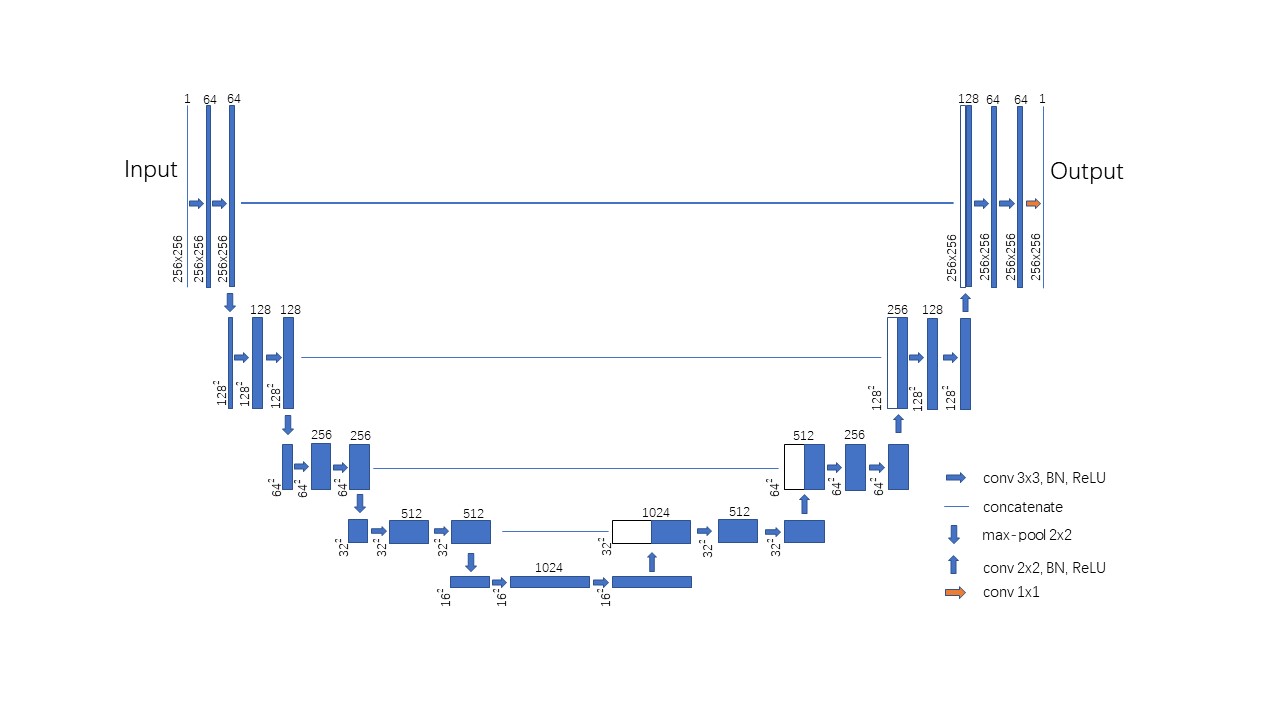}
	\end{center}\caption{A modified U-net architecture.}\label{fig:U-net}
\end{figure}

\subsection{Batch Normalization (BN)}\label{ssec:BN}
Batch normalization (BN) is a state-of-the-art technique for accelerating the training speed and improving the performance of optimization \cite{ioffe2015batch}. The batch normalization (BN) works differently in the training process and in the inference process after training. Hereby we give a brief description of the BN operation according to our notations for general neural networks in equation \eqref{DNN}.

With a given batch size $n_0$ and a mini-batch of the initial input, $$\mathcal{B}=\{\bx_1,\ldots,\bx_{n_0}\},$$  consider the output after convolution at the $k$-th layer of the neural network $\Phi_\Theta$,
$$
\mathcal{B}^k= \{ \bu^k_1,\ldots, \bu^k_{n_0}\}, \quad \mathrm{where}\quad \bu^k_j\coloneqq \hat{\bb}^k+ W^k\hat{\bz}^k_j,\quad j=1,\cdots,n_0\,.
$$
The operation of $\operatorname{BN}$ reads as follows,
\begin{equation*}
	\operatorname{BN}(\bu)\coloneqq \boldsymbol{\beta}^k + \operatorname{diag}(\boldsymbol{\gamma}^k)  \frac{\bu - \mathds{E}(\mathcal{B}^k)}{\sqrt{\operatorname{Var}(\mathcal{B}^k)+\epsilon}},\quad \forall\, \bu\in\mathcal{B}^k
\end{equation*}  
where $\boldsymbol{\gamma}^k$ and $\boldsymbol{\beta}^k$ are column vectors and the free parameters to be learned, $\mathds{E}(\mathcal{B}^k)$ and $\operatorname{Var}(\mathcal{B}^k)$ denote mean value and variance of the elements in $\mathcal{B}^k$, respectively, and  $\epsilon>0$ is a small constant introduced for numerical stability. 

In the training process, since $\mathds{E}(\mathcal{B}^k)$ and $\operatorname{Var}(\mathcal{B}^k)$ will involve the variable $\bu\in\mathcal{B}^k$, 
BN is not necessarily component-wise convex nor monotone nondecreasing. After the training process is completed, $\mathds{E}(\mathcal{B}^k)$ and $\operatorname{Var}(\mathcal{B}^k)$ will be fixed using the full population rather than the mini-batch, and so BN reduces to an affine transformation.

\subsection{The modified U-net architecture}\label{ssub:u-net}
We provide a mathematical modeling of the modified U-net architecture $\Phi_{\Theta}$ as shown in Figure \ref{fig:U-net}. Here we will make use of the notations of general neural networks introduced in Subsection \ref{sec:DNN}. Before proceeding further, it is necessary to briefly introduce the following operators.

\begin{itemize}
	 \item[$\blacksquare$] Convolution: It is a linear operator which can be represented by a tensor $K$ of dimension \newline $f_1\times f_2\times c_1\times c_2$, where $f_1\times f_2$ is the filter size of the convolution, $c_1$ and $c_2$ denote the number of channels of the input and output images, respectively. More details about the convolution operator can be found in \cite[Section 7]{higham2019deep} and \cite{dumoulin2016guide}. In the U-net architecture, a bias term is added after every convolution.
	 \item[$\blacksquare$] Zero padding: It is a linear operator which can be achieved by embedding a given matrix of dimension $n\times m$ into a zero matrix of dimension $n'\times m'$, with $n'\geq n$ and $m'\geq m$. By reshaping the matrices before and after zero padding to be column vectors, zero padding can be expressed as a matrix product, $\mathbf{M}=P_0\mathbf{N}$, where $P_0$ is a binary matrix with entries from $\{0,1\}$.
	 	\item[$\blacksquare$] Rectified linear unit (ReLU): $\bx\mapsto \max\{0,\bx\}$, where the maximum is applied component-wise.
	 \item[$\blacksquare$]$\operatorname{max-pool}$: Given a matrix of real numbers, partition it into $i\times j$ sub-matrices; the $\operatorname{max-pool}_{i\times j}(\cdot)$ returns a matrix that calculates the maximum value for each sub-matrix in the $i\times j$ partition.
	 \item[$\blacksquare$] Batch normalization (BN): Details are given in Subsection \ref{ssec:BN}. We emphasis that BN reduces to an affine transformation after training.
\end{itemize}

In the modified U-net architecture $\Phi_{\Theta}$, every convolution operator is combined with a zero padding, such that the input and output images have the same size. In addition, we add a Batch normalization (BN) between each convolution and ReLU function. As shown in Figure \ref{fig:U-net}, the modified U-net architecture can be constructed by several unit blocks: right arrow (involving concatenation), down arrow, and up arrow. Accordingly, we provide mathematical modeling for these unit blocks. 

\begin{itemize}
	\item Right arrow \tikz[baseline={(a.base)}]{\node (a) at (0pt,0pt) [] {};
		\draw[-{Triangle[width=8pt,length=4pt]}, color=arrow, line width=5pt]($(a)+(0pt,4pt)$) -- ($(a)+(12pt,4pt)$);} : A $k$-th right arrow reads as 
	$$
	\operatorname{ReLU}(\operatorname{BN}(\bb^k+K^kP_0\hat{\bz}^k)).
	$$ 
	According to the notations in equation \eqref{DNN}, the block of right arrow can be modeled as the composition of two layers, 
	\begin{equation} \label{eqn_RA}
	\begin{cases}
	\bz^{k+1}=\sigma^{k}(\hat{\bb}^{k} + W^{k}\hat{\bz}^{k}),\ \hat{\bz}^{k+1}=\bz^{k+1} \\
	\bz^{k+2}=\sigma^{k+1}(\hat{\bb}^{k+1} + W^{k+1}\hat{\bz}^{k+1})
	\end{cases}\,,
	\end{equation}
	where
	\[
        \begin{cases}W^k=K^kP_0,\quad \hat{\bb}^k=\bb^k \\
                             \sigma^k(\bu)= \frac{\bu - \mathds{E}(\mathcal{B}^k)}{\sqrt{\operatorname{Var}(\mathcal{B}^k)+\epsilon}} \\
		              K^k,\,\bb^k\in \Theta_{free}, \quad P_0\in \Theta_{frozen}
        \end{cases}\,, \quad \mathrm{and} \quad
	\begin{cases}W^{k+1}=\operatorname{diag}(\boldsymbol{\gamma}^{k}),\quad \hat{\bb}^{k+1}=\boldsymbol{\beta}^{k} \\
			\sigma^{k+1}(\bu)= \operatorname{ReLU}(\bu) \\
			\boldsymbol{\gamma}^{k},\,\boldsymbol{\beta}^{k}\in \Theta_{free}
	\end{cases}\,.
	\]
$P_0$ denotes the matrix of zero padding, which is a binary matrix with entries from $\{0, 1\}$. The filter size of every convolution $K^k$ is $3\times 3$ and the stride is $1$. In the ascending part of the U-net architecture, the concatenation operation is involved before the first convolution of each layer, and it can be modeled by the formulation of general neural networks as shown in Subsection \ref{sec:DNN}, i.e.  $\hat{\bz}^{k}=C(\bz^{i_k},\bz^k)$.

\item Down arrow \tikz[baseline={(a.base)}]{\node (a) at (0pt,0pt) [] {};
		\draw[-{Triangle[width=8pt,length=4pt]}, color=arrow, line width=5pt]($(a)+(0pt,12pt)$) -- ($(a)+(0pt,0pt)$);} : An $i$-th down arrow reads as $$\operatorname{max-pool}_{2\times2}(\hat{\bz}^i).$$
It can be simply modeled by the formula of general neural networks as shown in equation \eqref{DNN},
\begin{equation}\label{eqn_DA}
\sigma^i(\hat{\bb}^i+W^i\hat{\bz}^i)\,,
\end{equation}
where $ \hat{\bb}^i=\boldsymbol{0}$, $W^i=I$, and $\sigma^i(\bu)= \operatorname{max-pool}_{2\times2}(\bu)$; $\sigma^i$ is continuous, component-wise convex and monotone nondecreasing.

\item 	Up arrow \tikz[baseline={(a.base)}]{\node (a) at (0pt,0pt) [] {};
	\draw[-{Triangle[width=8pt,length=4pt]}, color=arrow, line width=5pt]($(a)+(0pt,0pt)$) -- ($(a)+(0pt,12pt)$);} : The up arrow has the same structure of the right arrow, and it can be modeled in the same way by equation (\ref{eqn_RA}). There are some minor differences in the convolution $K^k$: (i) the filter size of $K^k$ is $2\times 2$; (ii) the way of zero padding is different from that in the right arrow, but since it is still a zero padding, the matrix $P_0$ is still a binary matrix with entries from $\{0,1\}$.

\item Last right arrow \tikz[baseline={(a.base)}]{\node (a) at (0pt,0pt) [] {};
	\draw[-{Triangle[width=8pt,length=4pt]}, color=last-arrow, line width=5pt]($(a)+(0pt,4pt)$) -- ($(a)+(12pt,4pt)$);} :  The last right arrow is just a convolution with filter size $1\times 1$. It can be modeled as part of equation (\ref{eqn_RA}):
	\begin{equation} \label{eqn_last_RA}
	\bz^{L+1}=\sigma^{L}(\hat{\bb}^{L} + W^{L}\hat{\bz}^{L})\,,
	\end{equation}
where $W^L=K^L P_0$, $\hat{\bb}^L=\bb^L$, and  $\sigma^L(\bu)=\bu$. The output $\bz^{L+1}$ is of size $256\times256\times1$, where the channel number is $1$. As a result, the bias term $\hat{\bb}^L=\bb^L$ is essentially a scalar, i.e. the components of $\bb^L$ are all equal to $b^L_0\in\R$.
\end{itemize}


\subsection{Convex U-net}\label{ssec:convex_U-net}
In Subsection \ref{ssub:u-net}, we showed that the modified U-net architecture $\Phi_{\Theta}$ can be modeled by the set of equations \eqref{DNN}. Then according to Theorem \ref{prop:DNN_convex}, we can achieve a convex U-net by imposing the constraints \eqref{hypothesis2} and \eqref{hypothesis3}. Since the skip connection described by $A^{k,j_k}$ is not involved in $\Phi_{\Theta}$, we only need to consider the properties of $\sigma^k$ and $W^k$.

The architecture modeled by equation (\ref{eqn_DA}) automatically satisfies the constraints \eqref{hypothesis2} and \eqref{hypothesis3}. For the architecture modeled by equation (\ref{eqn_RA}) (which includes the formulation in equation (\ref{eqn_last_RA}) as well), $\sigma^k$ and $\sigma^{k+1}$ are continuous, convex, and monotone nondecreasing after training, so that the constraint \eqref{hypothesis2} is satisfied. To impose the constraint \eqref{hypothesis3} on $W^k$ and $W^{k+1}$, we have to impose nonnegativity constraint on the entries of $K^k$ and $\boldsymbol{\gamma}^{k}$. Above all, we have the following conclusion.
\begin{proposition} \label{prop_CovexUnet}
By imposing nonnegativity constraint on $\{K^k\mid\forall k\ge2\}$ and $\{\boldsymbol{\gamma}^{k}\mid\forall k\ge1\}$, the modified U-net $\Phi_{\Theta}$ turns out to be a convex U-net in the inference process after training.
\end{proposition}

In next subsection, we will build a uniformly convex U-net by employing the architecture of convex U-net. We will ask the output of the convex U-net to be nonnegative, which is helpful to construct the uniformly convex U-net in a very simple way. Considering the convex U-net obtained in Proposition \ref{prop_CovexUnet}, thanks to the ReLU function applied just before the last convolution operator, and the nonnegativity of the entries of the last convolution operator itself, the output $\bz^{L+1}$ must be nonnegative if we further impose nonnegativity constraint on the last bias term $\bb^L$.

\begin{proposition} \label{prop_CovexUnet_nonnegative}
By imposing nonnegativity constraint on $\{K^k\mid\forall k\ge2\}$, $\{\boldsymbol{\gamma}^{k}\mid\forall k\ge1\}$ and $\bb^L$, the modified U-net $\Phi_{\Theta}$ turns out to be a convex U-net $\Phi_{\Theta}^c$ having nonnegative output, i.e. the components of $\bz^{L+1}=\Phi_{\Theta}^c(\bx)$ are all nonnegative.
\end{proposition}

\subsection{Uniformly convex U-net}\label{ssec:u_convex_U-net}
With the trained convex U-net $\Phi_{\Theta}^c$ having nonnegative output, one can construct a uniformly convex U-net according to the following formula,
\begin{equation}\label{eq:Reg}
 \Phi_{\Theta}^{uc}(\bx)\coloneqq a\|\bx\|_p^p+\| \Phi^{c}_\Theta(\bx) \|_q^q\,,\qquad p\ge2,\ q\ge1\,,
\end{equation}
where $\|\cdot\|_p$ and $\|\cdot\|_q$ denote the standard $l^p$-norm and $l^q$-norm, respectively, and $a>0$ is a small positive constant.

Indeed, formula (\ref{eq:Reg}) is in the form of equation \eqref{eq:ucNN}, with
\[
D'=1,\quad f_1(\bx)=a\|\bx\|_p^p,\quad\mathrm{and}\quad g_1(\bz)=\|\bz\|_q^q\,.
\]
The function $f_1(\bx)=a\|\bx\|_p^p$ is uniformly convex for $p\geq 2$; e.g., see \cite[Theorem 2.3]{borwein2009uniformly} and \cite[Theorem 1.f.1]{lindenstrauss2013classical}. The function $g_1(\bz)=\|\bz\|_q^q\ (q\ge1)$ is continuous, convex and monotone nondecreasing in the domain $\R^N_{\geq0}:=\{\bz\in\R^N\mid z_i\ge0,\ i=1,\cdots,N\}$. Moreover, as shown in Proposition \ref{prop_CovexUnet_nonnegative}, the output of the convex U-net $\Phi_{\Theta}^c$ satisfies $\Phi_{\Theta}^c\in\R^N_{\geq0}$. As a result, $\Phi_{\Theta}^{uc}$ in equation (\ref{eq:Reg}) is uniformly convex according to Theorem \ref{lem:DNN_uconvex}. Equation (\ref{eq:Reg}) gives a very simple way to construct the uniformly convex U-net.

\section{Iterated network Tikhonov (iNETT) method}\label{sec:iNETT}
In this section we propose the iterated network Tikhonov (iNETT) method. First, fix a continuous, uniformly  convex and well trained neural  network
\begin{equation*}
	\Phi^{uc}_\Theta : X \to (\R, |\cdot|),
\end{equation*}
for example by applying the steps \ref{step1}, \ref{step2} and \ref{step3} at the end of Section \ref{sec:ucNN}, and define 
$$\Reg\coloneqq \Phi_{\Theta}^{uc}.
$$ 
For the numerical experiments in Section \ref{sec:numerical_examples}, we will adopt the modified U-net architecture presented in Section \ref{sec:example} as equation \eqref{eq:Reg}.

Fix $r\in(1,\infty)$, $\bx_0\in X$, $\xi_0 \in \partial \Reg(\bx_0)$  and a sequence of positive real numbers $\{\alpha_n\}_n$ such that 
$$
\sum_{n=1}^\infty \alpha_n^{-1}=\infty \quad \mbox{and} \quad \alpha_n\leq c\alpha_{n+1}, \quad \mbox{where } c>0 \mbox{ is a constant}.
$$
The iNETT method is then given by
\begin{equation}\tag*{iNETT}\label{iterNETT2}
	\begin{cases}
		\bx^\delta_n:= \underset{\bx \in X}{\argmin} \frac{1}{r}\|F \bx - \by^\delta\|_Y^r + \alpha_n\Breg_{\boldsymbol{\xi}^\delta_{n-1}}(\bx,\bx^\delta_{n-1}),\\
		\boldsymbol{\xi}^\delta_{n}:= \boldsymbol{\xi}^\delta_{n-1} - \frac{1}{\alpha_{n}} F^TJ_r\left(F\bx_n^\delta - \by^\delta\right),\\
		\bx_0 \in X, \,	\boldsymbol{\xi}_0 \in \partial \Reg(\bx_0).
	\end{cases} 
\end{equation}
The stopping rule associated to \ref{iterNETT2} is the standard \emph{discrepancy principle}. Specifically, we stop the iterations at the first step $n_\delta=n$ such that
\begin{equation}\label{dicrepancy_principle}
	\left\|F\bx_{n}^\delta  - \by^\delta\right\|_Y\leq \tau\delta< \left\|F\bx_{n-1}^\delta  - \by^\delta\right\|_Y, 
\end{equation}
where $\tau>1$ is a fixed constant, and $\delta$ is the parameter which estimates the level of noise $\boldsymbol{\eta}$, viz. $\|\boldsymbol{\eta}\|_Y=\|\by^\delta-\by\|_Y\leq \delta$. Following we will provide convergence analysis and implementation details of the iNETT method.

\subsection{Convergence analysis}\label{sec:conv_analysis}
We present the proof of well-posedness, convergence and stability  of the \ref{iterNETT2} method. The proof is mainly a straightforward adaptation, in this finite dimensional setting, of  standard results of general convex regularization theory (e.g. \cite{scherzer2009variational}) and \cite[Theorem 3.2]{jin2014nonstationary}. For the convenience of the reader and to make the theoretical treatment self-contained, we develop here the main points.

\begin{theorem}\label{thm:convergence}
Let $F : X \to Y$ be a linear operator between finite dimensional normed spaces and suppose that  \eqref{hypothesis0} and \eqref{hypothesis1} are valid.  Let $\Reg\coloneqq \Phi_\Theta^{uc}$ be a continuous and uniformly convex neural network. Then the method \textnormal{\ref{iterNETT2}} is well-posed and there exists a unique $\Breg_{\bxi_0}$-\emph{minimizing} solution $\bx^\dagger$. The method \textnormal{\ref{iterNETT2}}, coupled with the stopping rule \eqref{dicrepancy_principle}, stops in finite steps $n_\delta < \infty$ and it converges to  $\bx^\dagger$ as $\delta\to 0$. In particular, 
\begin{equation*}
\bx_{n_\delta}^\delta \to \bx^\dagger, \quad \Reg(\bx_{n_\delta}^\delta)\to \Reg(\bx^\dagger), \quad \Breg_{\boldsymbol{\xi}^\delta_{n_\delta}}(\bx^\dagger,\bx_{n_\delta}^\delta)\to 0
\end{equation*}
as $\delta \to 0$.
\end{theorem}
\begin{proof}

About the existence of $\bx^\dagger$ in Definition \ref{def:R-minimizing_sol}, observe that
\begin{equation*}
	0\leq c\coloneqq \inf \left\{ \Breg_{\bxi_0}(\bx,\bx_0) \mid x\in X,\; Fx = y  \right\}
\end{equation*}
is well-defined because of \eqref{hypothesis0}.  Since $\Breg_{\boldsymbol{\xi}_0}(\bx,\bx_0)$ is continuous and coercive (see Remark \ref{rem:breg_coercive}), the level sets $\{\bx \mid \Breg_{\bxi_0}(\bx,\bx_0)\leq M\}$ are compact for any $M>0$, and therefore there exists a minimizer $\bx^\dagger$, by standard topological arguments, which is unique because of the linearity of $F$. In the same fashion, and because of the convexity of $\Breg_{\boldsymbol{\xi}^\delta_{n-1}}$, $\bx_n^\delta$ is well-defined and unique at each step.

 Observe now that $J_r$ is a single-valued map thanks to \eqref{hypothesis1}, and that $\bxi_n^\delta \in \partial \Reg(\bx^\delta_n)$ in virtue of the minimality of $\bx_n^\delta$. Indeed, if we define 
 $$
 \mathcal{G}(\bx)\coloneqq  \frac{1}{r}\|F \bx - \by^\delta\|_Y^r + \alpha_n\Breg_{\boldsymbol{\xi}^\delta_{n-1}}(\bx,\bx^\delta_{n-1}),
 $$
by subgradient calculus (see for example \cite[Theorem 10.6]{rockafellar2009variational}), and because $J_r$ is the subgradient of the map $\by \mapsto \frac{1}{r}\|\by\|_Y^r$, it holds that
\begin{equation*}
0\in\partial \mathcal{G}(\boldsymbol{\bx}^\delta_{n}) = F^TJ_r(F\bx^\delta_{n}-\by^\delta) + \alpha_n\partial\Reg(\bx^\delta_{n}) -\alpha_n\bxi^\delta_{n-1}.
\end{equation*}	
The rest of the proof is a straightforward application of \cite[Theorem 3.2]{jin2014nonstationary}.
\end{proof}

\subsection{Implementation}\label{sec:implementation}
\correct{Since our applications in Section \ref{sec:numerical_examples} will be based on computerized tomography, from now on we will consider $X$ and $Y$ as vector spaces of digital images.} We will use the standard $\ell^2$-norm for $X$, that is, $X=(\R^N,\|\cdot\|_2)$ and a normalized $\ell^2$-norm for $Y= \left(\R^M,\|\cdot\|_Y\right)$ as $\|\by\|_Y\coloneqq\frac{1}{\sqrt{M}}\|\by\|_2$, and we will fix $r=2$. Therefore, the  \ref{iterNETT2} iteration  will assume the following form (see also Remark \ref{rem:J_r}),

\begin{equation}\tag*{iNETT}\label{iterNETT3}
	\begin{cases}
		\bx^\delta_n= \underset{\bx \in X}{\argmin}\frac{1}{2M}\|F \bx - \by^\delta\|_2^2+ \alpha_n\Breg_{\boldsymbol{\xi}^\delta_{n-1}}(\bx,\bx^\delta_{n-1})\,,\\
	\boldsymbol{\xi}^\delta_{n}= \boldsymbol{\xi}^\delta_{n-1} - \frac{1}{\alpha_{n}}\frac{1}{M}  F^T\left(F\bx_n^\delta - \by^\delta\right)\,, \\ 
	\bx_0 \in X, \,	\boldsymbol{\xi}_0 \in \partial \Reg(\bx_0)\,.
	\end{cases} 
\end{equation}
The minimum problem in \ref{iterNETT3} can be solved by a gradient descent approach. Recalling that the Bregman distance of the regularization term is defined as
\[
\Breg_{\boldsymbol{\xi}^\delta_{n-1}}(\bx,\bx^\delta_{n-1})=\Reg(\bx)-\Reg(\bx^\delta_{n-1})-\langle \boldsymbol{\xi}^\delta_{n-1}, \bx-\bx^\delta_{n-1}\rangle\,,
\]
the gradient descent algorithm for $\bx^\delta_n$ has the following form,
\begin{equation*} \label{eq:gradient_descent}
\bx^\delta_{n,k+1}=\bx^\delta_{n,k}-s\cdot\left[ \frac{1}{M}F^T\left(F\bx_{n,k}^\delta - \by^\delta\right)+\alpha_{n}\left(\frac{\partial\Reg(\bx)}{\partial\bx}{\Big|}_{\bx=\bx^\delta_{n,k}}-\boldsymbol{\xi}^\delta_{n-1} \right) \right]\,, \quad k\in\mathbb{N}
\end{equation*}
where $s>0$ denotes the step size of the gradient descent iteration. The initial guess for the iteration can be taken as $\bx^\delta_{n,0}=\bx^\delta_{n-1}$.

\subsection{Building a uniformly convex neural network regularizer} \label{subsec:CovRegul}
We will adopt the uniformly convex neural network \eqref{eq:Reg} presented in Subsection \ref{ssec:u_convex_U-net}, with $p=q=2$. That is,
\begin{equation*}
	\Reg(\bx)\coloneqq \Phi_{\Theta}^{uc}(\bx)=a\|\bx\|_2^2+\| \Phi^{c}_\Theta(\bx) \|_2^2,
\end{equation*}
where $\Phi^{c}_\Theta(\bx)$ is a well trained convex U-net having nonnegative output as shown in Proposition \ref{prop_CovexUnet_nonnegative}, and $a>0$ is a small positive constant.

Since the regularizer is designed to penalize artifacts in the solution, the trained U-net $\Phi^{c}_\Theta(\bx)$ should have small output for artifact-free images and have large output for images with artifacts. We follow the strategy proposed in  \cite{li2020nett} for the preparation of training data. Let $\{\bx_s^*\mid s=1,\cdots,N_1+N_2\}$ denote a set of images in $X$. The training images are constructed in the following way,
\begin{equation}\label{train_input}
	\boldsymbol{z}_s=\left\{
	\begin{array}{lcr}
		F^\dagger (F\bx_s^*+\boldsymbol{\eta}) &,& s=1,\cdots,N_1 \\
		\bx_s^*&,& s=N_1+1,\cdots,N_1+N_2
	\end{array}
	\right.
\end{equation}
where $\boldsymbol{\eta}$ denotes noise perturbations on the data, and $F^{\dagger}$ denotes the pseudo inverse of the forward operator $F$. In the set of training images $\{\boldsymbol{z}_s\mid s=1,\cdots,N_1+N_2\}$, the first $N_1$ $\boldsymbol{z}_s$'s simulate the images with artifacts, while the last $N_2$ $\boldsymbol{z}_s$'s simulate the artifact-free images. The output label is the error between $\boldsymbol{z}_s$ and $\bx_s^*$,
\begin{equation}\label{train_output}
	\boldsymbol{r}_s=|\bx_s^*-\boldsymbol{z}_s|,\quad s=1,\cdots,N_1+N_2\,,
\end{equation}
where the absolute value is applied component-wise, and the training set is composed of the pairs of input images and output labels: $$\{(\boldsymbol{z}_s,\boldsymbol{r}_s)\mid s=1,\cdots,N_1+N_2\}.$$
The loss function of one training sample is defined as follows,
\[
L_s(\Theta)=\|\Phi^c_{\Theta}(\boldsymbol{z}_s)- \boldsymbol{r}_s \|_2^2+\lambda\|\Theta_{free}\|_2^2\,,
\]
where
\[
\Theta_{free}= \{\boldsymbol{\beta}^{k};\boldsymbol{\gamma}^{k};\bb^k;K^k\}_{k}
\]
considering the modified U-net architecture as shown in Section \ref{sec:example}. We propose to use a mini-batch optimization approach in the training process, so that the cost function for updating $\Theta_{free}$ has the following form,
\begin{equation*}\label{eq:cost-function}
	\mathcal{L}_t(\Theta)=\frac{1}{|\mathcal{I}_t|}\sum_{s\in\mathcal{I}_t} L_s(\Theta) = \frac{1}{|\mathcal{I}_t|}\sum_{s\in\mathcal{I}_t} \|\Phi^c_{\Theta}(\boldsymbol{z}_s)- \boldsymbol{r}_s \|_2^2+\lambda\|\Theta_{free}\|_2^2\,,
\end{equation*}
where $\mathcal{I}_t\subset\{1,2,\cdots,N_1+N_2\}$ denotes a mini-batch of the index set, and $|\mathcal{I}_t|$ denotes the batch size. The parameter $\lambda$ controls the amount of regularization applied. We utilize a random shuffling strategy \cite{safsha20,liwanfan22} for the sampling of the mini-batch set $\mathcal{I}_t$. Then the set of training parameters $\Theta_{free}$ is updated according to the Adam algorithm \cite{kinba14}. After each updating, we impose non-negativity constraint on $\boldsymbol{\gamma}^{k}$ and $K^k$ for the convexity of the U-net $\Phi^c_{\Theta}$, with the only exception of the convolution matrix $K^1$ in the first layer. In addition, to make the convex U-net $\Phi^c_{\Theta}$ have nonnegative output, we impose non-negativity constraint on the bias term of the last layer $\bb^L$. After the training process is completed, we fix the elements $\mathds{E}(\mathcal{B}^k)$ and $\operatorname{Var}(\mathcal{B}^k)$ which define the batch normalization operators. Algorithm \ref{algo:Phi^c} summarizes the details of the training process.

\begin{algorithm}[!h]\caption{Training the convex U-net $\Phi^c_{\Theta}$ having nonnegative output} \label{algo:Phi^c}
	\begin{algorithmic}[1]
		\State  Initialize the  free parameters set $\Theta_{free}=\{\boldsymbol{\beta}^{k};\boldsymbol{\gamma}^{k};\bb^k;K^k\}_{k}$.
		\For {epoch $e=1,2,\cdots,N_e$}
		\State Sample a random permutation $\{\pi(1), \pi(2),\ldots, \pi(N_1+N_2)\}$ of the index set $\{1,2,\ldots,N_1+N_2\}$.
		\State With a given batch size $n_0$, partition the shuffled index set into $n=\frac{N_1+N_2}{n_0}$ subsets: $\mathcal{I}_t=\{\pi((t-1)n_0+1),\ldots,\pi(t\,n_0)\}$, $t=1,2,\cdots,n$.
		\For {$t=1$ to $n$}
		\State Update $\Theta_{free}$ according to the Adam algorithm and mini-batch cost function  $\mathcal{L}_t(\Theta)$.
		\State Impose nonnegativity constraint: $K^k=\max\{0,K^k\}$, $\forall k\ge2$;\quad $\boldsymbol{\gamma}^{k}=\max\{0,\boldsymbol{\gamma}^{k}\}$, 
		
		\hspace{9pt} $\forall k\ge1$;\quad $\bb^L=\max\{0,\bb^L\}$, where $\bb^L$ denotes the bias term of the last layer.
		\EndFor
		\EndFor
		\State Fix $\mathds{E}(\mathcal{B}^k)$ and $\operatorname{Var}(\mathcal{B}^k)$ in batch normalization (BN) using the full population rather than the mini-batch.
	\end{algorithmic}
\end{algorithm}

\correct{\subsection{Some remarks}
	As final comments for this section, we provide discussions on two subjects: (1) similarities and differences between \ref{NETT} and \ref{iterNETT2}; (2) alternative training strategies for the \ref{iterNETT2} method. 
	
	\subsubsection{NETT and iNETT}	
	The \ref{NETT} method in \cite{li2020nett} is developed in the setting of infinite dimensional Banach spaces. The \ref{iterNETT2} method can be generalized to the infinite dimensional case as well, where one may require the Banach space X to be reflexive, and re-define the neural network architecture \eqref{DNN} in a proper way. In this paper, we propose \ref{iterNETT2} in a finite dimensional case mainly for personal taste. Since we focus on the inverse problems in imaging science, e.g. computerized tomography, the setup is naturally of finite dimension. The images are defined in a finite dimensional vector space, and the neural networks are trained and deployed in the set of images with fixed finite pixels. When considering the accuracy of discretization, it will be necessary to study \ref{iterNETT2} in the infinite dimensional setup. A main difficulty will be the convergence analysis: It is unclear under what hypotheses a convex neural network, coupled with training and optimization rules, converges to an infinite dimensional operator. 
	
	A major difference between \ref{NETT} and \ref{iterNETT2} lies in the assumption of coercivity or convexity on the data-driven regularizer. For the convergence analysis of \ref{NETT} method, the regularization term $\Reg(\bx)\coloneqq \psi(\Phi_\Theta(\bx))$ is assumed to be coercive \cite{li2020nett}. On the other hand, in our \ref{iterNETT2} method, the neural network regularizer $\Reg(\bx)\coloneqq\Phi_{\Theta}^{uc}(\bx)$ is required to be uniformly convex. In fact, at each iteration of  \ref{iterNETT2}, the penalty term $\mathcal{P}_n(\bx)\coloneqq \Breg_{\boldsymbol{\xi}^\delta_{n-1}}(\bx,\bx^\delta_{n-1})$ is coercive, as a consequence of the uniform convexity of $\Reg$ (see Remark \ref{rem:breg_coercive}). The assumption of convexity on $\Reg$ is from the iterated Tikhonov method. For example, considering the proofs of \cite[Theorems 3.1 and 3.2]{jin2014nonstationary}, a crucial role to prove the (strong) convergence is played by the following inequality, 
	$$
	\Breg_{\boldsymbol{\xi}}(\bx,\hat{\bx})\geq f(\|\bx -\hat{\bx}\|_X),
	$$
	where $f \colon [0,+\infty) \to [0,+\infty)$ is a strictly increasing continuous function such that $f(0)=0$. The above inequality is true if and only if $\Reg$ is uniformly convex (see \cite[Theorem 3.5.10]{zalinescu2002convex}). It would be interesting to understand if the assumption of uniform convexity on $\Reg$ can be relaxed, say for example just asking coercivity.  To the best of our knowledge, it is yet an open problem.  
	
	\subsubsection{Alternative training strategies for \ref{iterNETT2}}
	As described in section \ref{subsec:CovRegul}, the training strategy of \ref{iterNETT2} is adapted from that of \ref{NETT} \cite{li2020nett}. This is a natural choice when we consider the evolution of \ref{NETT} into an iterated Tikhonov method. In this way, we can have a simple comparison between the reconstruction results of \ref{NETT} and \ref{iterNETT2}, which will be provided in the next section.
	
	Nevertheless, it would be useful and interesting to study alternative training strategies for \ref{iterNETT2}. In \cite{lunz2018adversarial}, it proposes an adversarial training strategy for the neural network regularizer. The adversarial regularizer learns to discriminate between the distribution of ground truth images and the distribution of unregularized reconstructions. The training can be performed even if only unsupervised training data are available \cite{lunz2018adversarial}. In \cite{mukherjee2020learned}, the adversarial strategy is adopted to successfully build a learned convex regularizer. 
This strategy deserves elaborate study in the future work, and it provides a direction to further improve the \ref{iterNETT2} algorithm.
	
	Moreover, in \cite{sun2016deep}, it develops a deep ADMM-Net for compressive sensing MRI, which is one of the earliest learned reconstruction methods in imaging science. The ADMM-net maps the ADMM iterative procedure to a data flow graph, and then generalizes the operations of ADMM to have learnable parameters as network layers. This strategy inspires us to generalize \ref{iterNETT2} to a fully learned architecture in the future study, e.g. learn the neural network regularizer $\Reg(\bx)$ in the optimization process, which would provide more flexibility of building the data-driven regularizer, even though the convergence analysis can become much more difficult.

}

\section{Application to computerized tomography}\label{sec:numerical_examples}
We provide applications of the \ref{iterNETT3} algorithm, as presented in Subsection \ref{sec:implementation}, in 2D computerized tomography (CT). The imaging problem can be modeled as an inverse problem of the Radon transform:
\begin{equation}\label{app1}
	y(s,\phi)=\int_{\mathbb R} x(s\cos\phi-t\sin\phi, s\sin\phi+t\cos\phi)\,\mathrm{d}t\,,\quad  (s,\phi)\in\mathbb{R}\times[0,2\pi]\,,
\end{equation}
where $x(\cdot,\cdot)$ denotes the attenuation image to be recovered, and $y$ denotes the projection data. After discretization, equation (\ref{app1}) reduces to a linear system,
\begin{equation}\label{eqn8.2}
	\by=F\bx,
\end{equation}
where $\bx=(x_1,\ldots,x_N)^T$ is the vector of the attenuation image, $\by=(y_1,\ldots,y_M)^T$ simulates the measurement projection data, and $F=(a_{ij})_{M\times N}$ denotes the projection matrix; the component $a_{ij}$ represents the contribution of the $j$-th pixel of the attenuation image to the $i$-th projection datum. The inverse problem is to recover $\bx$ from the measurement data $\by^{\delta}$ perturbed by unknown noises, viz. $\by^{\delta}=\by+\boldsymbol{\eta}$. In our CT application, we consider the attenuation image with $256\times256$ pixels, and the sinogram of projection data is of size $256$ pixels by $60$ views. As a result, we have $N=256^2,\, M=256\times60$.

\subsection{Training convex U-net} \label{AppTrainUnet}
The convex U-net for \ref{iterNETT3} is trained according to the approach described in Subsection \ref{subsec:CovRegul}. The attenuation images are taken as synthetic phantoms consisting of randomly generated piecewise constant ellipses. We prepared 2600 synthetic phantoms, denoted as $\{\bx_s^*\mid s=1,\cdots,2600\}$, and Figure \ref{Fig5} shows six of them.

\begin{figure}[h]
	\centering
	\scalebox{0.25}[0.25]{\includegraphics{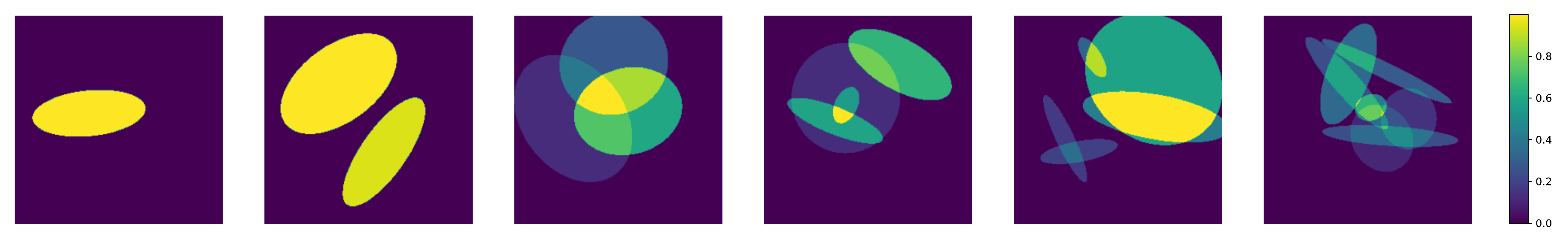}}
	\caption{\correct{six of the 2600 synthetic phantoms.}} \label{Fig5}
\end{figure}

The pairs of input images and output labels $\{(\boldsymbol{z}_s,\boldsymbol{r}_s)\mid s=1,\cdots,2600\}$ are constructed according to equation (\ref{train_input}) and equation (\ref{train_output}), where we take $N_1=N_2=1300$, so that the first 1300 samples simulate the images with artifacts and the second 1300 samples simulate the artifact-free images. The input images of the first 1300 samples are constructed according to
\[
\boldsymbol{z}_s=F^{\dagger}(F\bx_s^*+\boldsymbol{\eta}),
\]
where we add $0-10\%$ Gaussian noise to simulate the real situation of practical measurements. To be specific, 
$$
\boldsymbol{\eta}=\eta\cdot(F\bx_s^*)\cdot \mathcal{N}(0,1),
$$
where $\eta$ denotes a random number between 0 and 0.1, and $\mathcal{N}(0,1)$ denotes the Gaussian noise with zero mean and unit variance. Figure \ref{Fig6}\,(a) shows the simulated measurement data with Gaussian noises for the six synthetic phantoms displayed in Figure \ref{Fig5}. The pseudo inverse $F^\dagger$ for the radon transform is achieved by the algebraic reconstruction technique (ART) \cite{gorbenher70,ao2021data,}, which is a Kaczmarz iterative scheme for the linear system (\ref{eqn8.2}). The ART algorithm has the following formula:
\begin{equation*}\label{eqn3.2}
	\bx^{(n)}=\bx^{(n-1)}+\frac{y^\delta_i-\mathbf{a}_i^T\bx^{(n-1)}}{\|\mathbf{a}_i\|_2^2}\,\mathbf{a}_i,\qquad n\in\mathbb{N}^+,\ i=1,2,\cdots,M\,,
\end{equation*}
where $\mathbf{a}_i$ denotes the $i$-th row of the projection matrix $F$, $\mathbf{a}_i=(a_{i1},\cdots,a_{iN})^T$, and $y^\delta_i$ denotes the $i$-th component of the measurement data. One round of ART iterations implies one complete sweep of the measurement data $y^\delta_i$, with $i$ going from $1$ to $M$. In our application, the pseudo inverse $F^\dagger$ is constructed by 5 rounds of ART iterations, and the initial guess is taken as $\bx^{(0)}=\left(\frac{1}{N}\right)_{N\times1}$. Figure \ref{Fig6}\,(b) shows the input images $\boldsymbol{z}_s=F^{\dagger}(F\bx_s^*+\boldsymbol{\eta})$ corresponding to the six synthetic phantoms displayed in Figure \ref{Fig5}. Since ART and \ref{iterNETT3} are both iterative approaches, the artifacts in the images generated by ART is consistent with the artifacts to be regularized in the \ref{iterNETT3} algorithm.
\begin{figure}[h]
	\centering
	\correct{\subfigure[]{\scalebox{0.25}[0.25]{\includegraphics{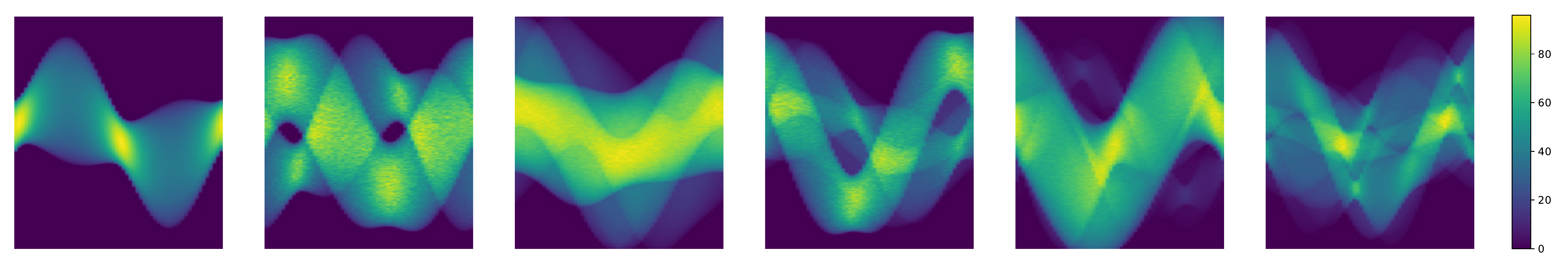}}}}
	\correct{\subfigure[]{\scalebox{0.25}[0.25]{\includegraphics{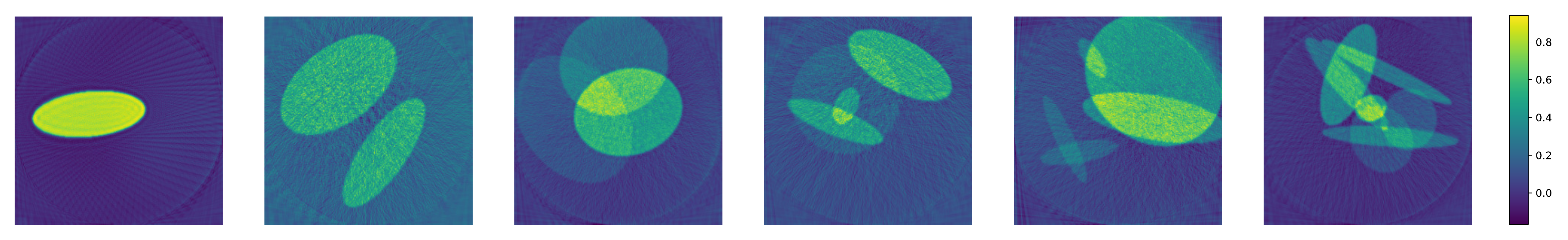}}}}
	\caption{Measurement data and input images with artifacts corresponding to the six synthetic phantoms displayed in Figure \ref{Fig5}. (a) Measurement data with Gaussian noises: $F\bx_s^*+\boldsymbol{\eta}$; (b) input images with artifacts constructed by 5 round of ART iterations on the measurement data.} \label{Fig6}
\end{figure}

We separate the 2600 prepared samples into three sets: the training set of 2000 samples, $\{(\boldsymbol{z}_s,\boldsymbol{r}_s)\mid s\in \{1,\cdots,1000\}\cup \{1301,\cdots2300 \} \}$, the validation set of 400 samples, $\{(\boldsymbol{z}_s,\boldsymbol{r}_s)\mid s\in \{1001,\cdots,1200\}\cup \{2301,\cdots2500 \} \}$, and the test set of 200 samples, $\{(\boldsymbol{z}_s,\boldsymbol{r}_s)\mid s\in \{1201,\cdots,1300\}\cup \{2501,\cdots2600 \} \}$. Each set consists of two types of input images: half of images with artifacts, and half of artifact-free images. The convex U-net $\Phi^c_{\Theta}$ is trained according to Algorithm \ref{algo:Phi^c}, where the regularization parameter $\lambda$ is taken as $5\times10^{-4}$, the batch size is set as $n_0=|\mathcal{I}_t|=10$, and the learning rate is $5\times10^{-4}$. Figure \ref{Fig7} shows the convergence plot in the training process, where the blue curve illustrates the convergence of $\mathcal{L}_t(\Theta)$ on training set, and the red curve shows its performance on validation set. Since we employ a mini-batch optimization approach with random shuffling, and the non-negativity constraint is imposed after every iteration, the cost function $\mathcal{L}_t(\Theta)$ has an oscillating behavior. Finally, we evaluate the trained convex U-net on the test set of 200 samples. The mean squared error (MSE) is used to measure the performance of reconstructions, and it is defined in the following way,
\[
\mathrm{MSE}=\frac{1}{N}\|\Phi^c_{\Theta}(\boldsymbol{z}_s)- \boldsymbol{r}_s \|_2^2\,,\qquad \boldsymbol{z}_s,\, \boldsymbol{r}_s \in\mathbb{R}^N\,.
\]
\begin{figure}[!t]
	\centering
	\scalebox{0.4}[0.4]{\includegraphics{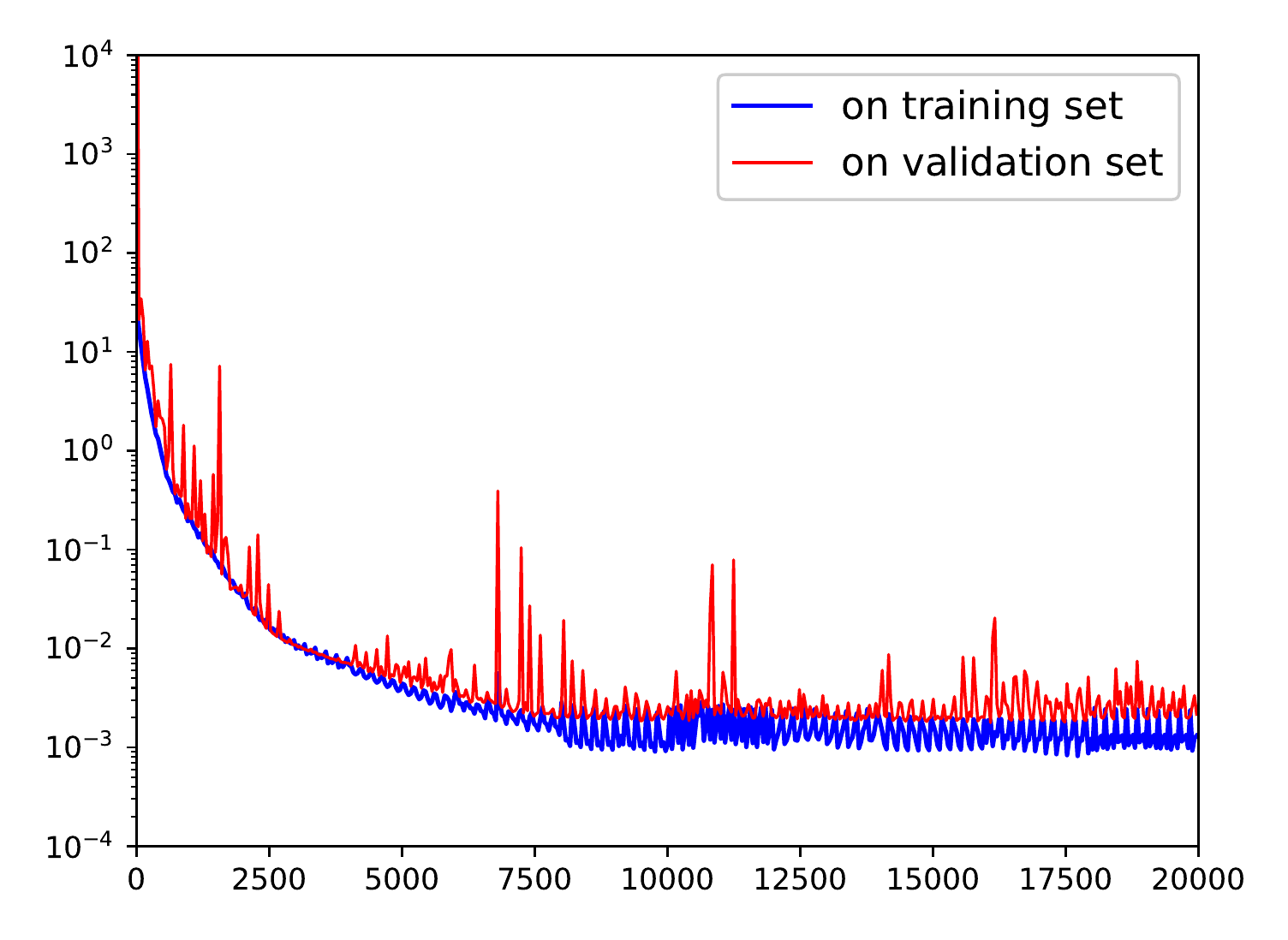}}
	\caption{Convergence plot of $\mathcal{L}_t(\Theta)$ in the training process. The blue curve illustrates the convergence on training set, and the red curve shows the performance on validation set.} \label{Fig7}
\end{figure}
\begin{figure}[!htbp]
	\centering
	\subfigure[]{\scalebox{0.4}[0.4]{\includegraphics{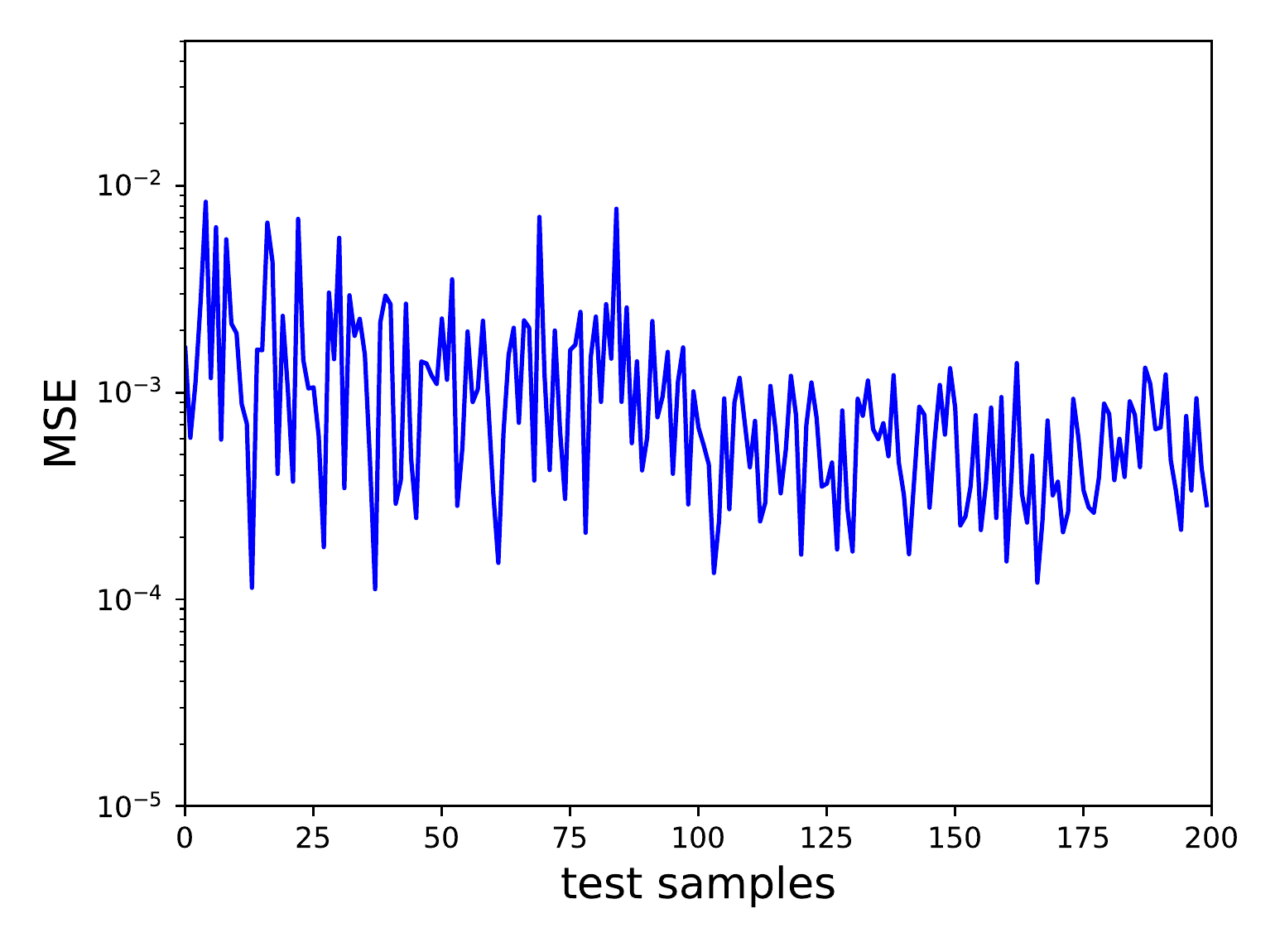}}}
	\subfigure[]{\scalebox{0.4}[0.4]{\includegraphics{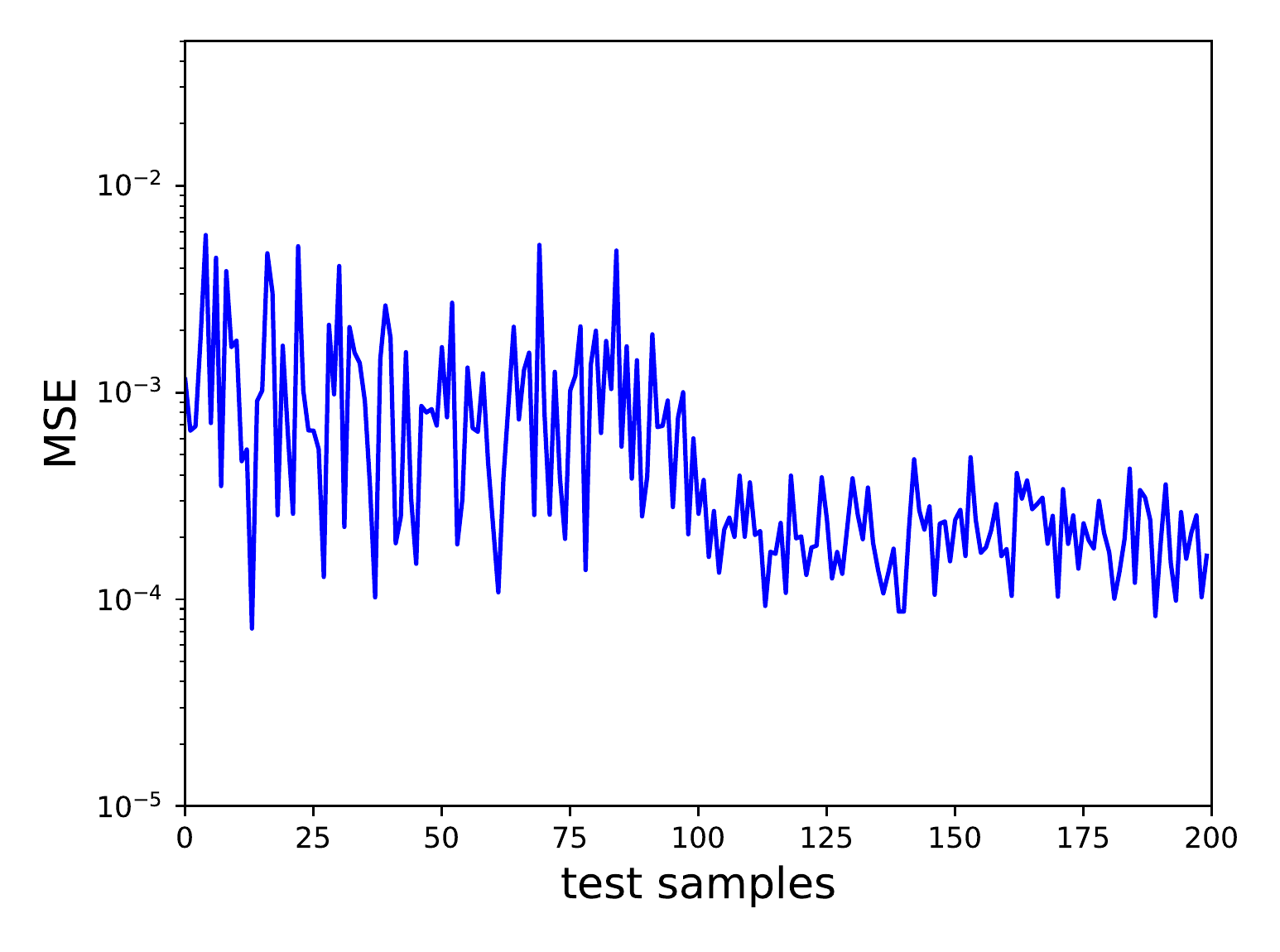}}}
	\caption{Mean squared error (MSE) on the test set of 200 samples. The first 100 samples correspond to input images with artifacts, and the second 100 samples correspond to artifact-free images with $\boldsymbol{r}_s=\boldsymbol{0}$. (a) MSE for the convex U-net $\Phi_{\Theta}^c$; (b) MSE for the general U-net $\Phi_{\Theta}$.} \label{Fig8}
\end{figure}
\FloatBarrier
Figure \ref{Fig8}\,(a) shows the values of MSE for the trained convex U-net on the 200 test samples. In next subsection, we will illustrate the performance of \ref{iterNETT3} and compare it with the standard NETT algorithm \cite{li2020nett}. Since the NETT  algorithm does not require convexity of the neural network, we train a general U-net without the constraint of convexity, denoted as $\Phi_{\Theta}$. The training set, validation set, and test set are taken the same as those of convex U-net $\Phi_{\Theta}^c$. Figure \ref{Fig8}\,(b) shows the values of MSE for the general U-net on the 200 test samples. It concludes that the performance of the convex U-net $\Phi_{\Theta}^c$ is almost comparable with that of the general U-net $\Phi_{\Theta}$ on the test set.

\subsection{Reconstruction results}
The \ref{iterNETT3}  algorithm is implemented according to Subsection \ref{sec:implementation}.  With the convex U-net $\Phi_{\Theta}^c$, we build the uniformly convex regularizer $\Reg(\bx)$ for \ref{iterNETT3} as shown in Subsection \ref{subsec:CovRegul}, where we set $a=10^{-3}$. The sequence of regularization parameters $\alpha_n$ is taken as $2^{-n}$, $n\in\mathbb{N}^+$; the initial guess is $\bx_0=\frac{1}{N}$ with $N=256^2$, and $\bxi_0 \in\partial \Reg(\bx_0)$. The stopping rule is the discrepancy principle as shown in equation (\ref{dicrepancy_principle}), where we take the constant $\tau=1.01$.

\subsubsection{Example 1: \correct{synthetic phantom with non-overlapping ellipses}}
Figure \ref{Fig9}\,(a) shows the true model of attenuation image. It is a synthetic phantom of the same type as the phantom images in the training set and validation set, but it is not contained in those phantom images. 
\begin{figure}[!htbp]
	\centering
	\subfigure[]{\scalebox{0.4}[0.4]{\includegraphics{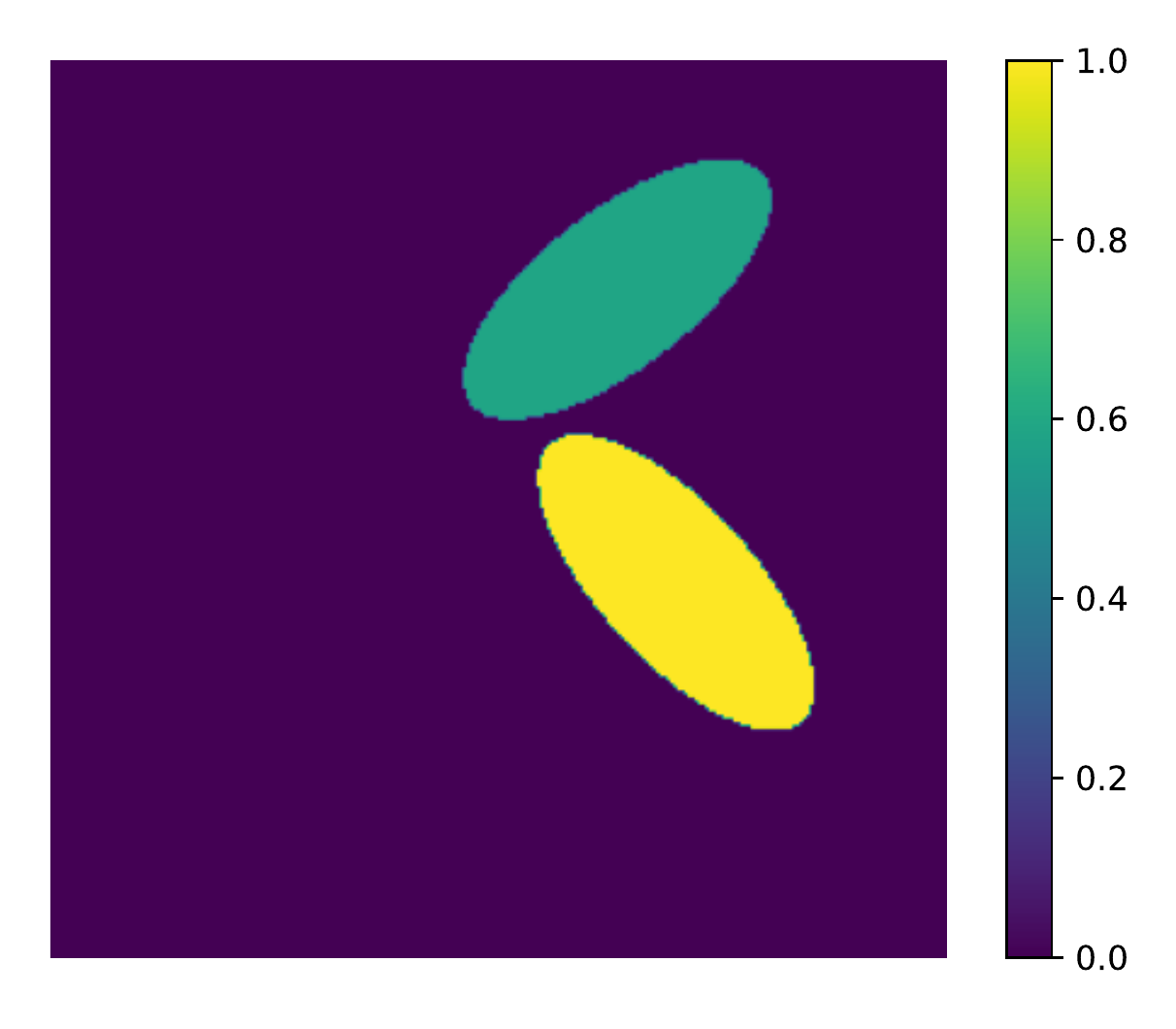}}}
	\subfigure[]{\scalebox{0.4}[0.4]{\includegraphics{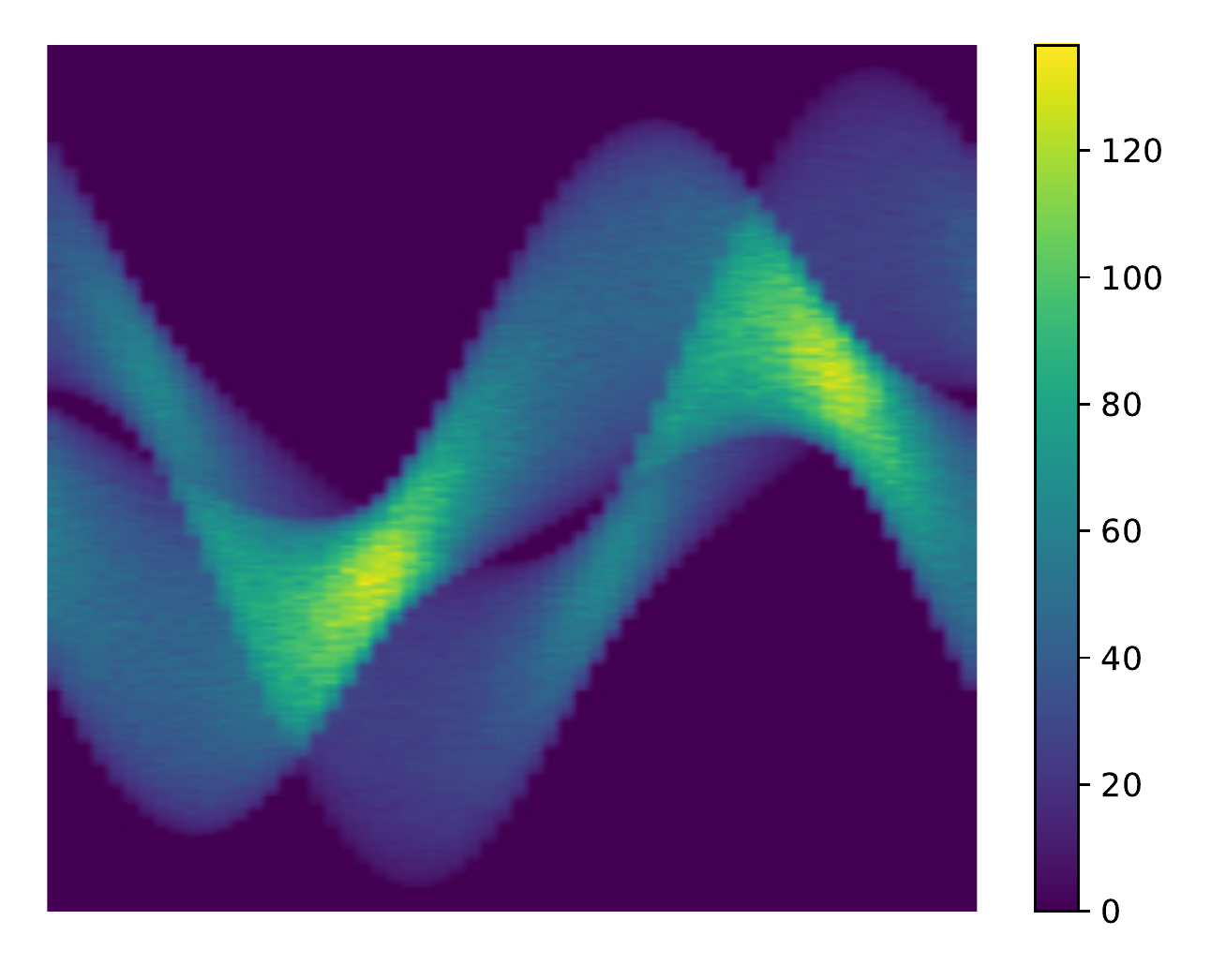}}}\\
	\subfigure[]{\scalebox{0.4}[0.4]{\includegraphics{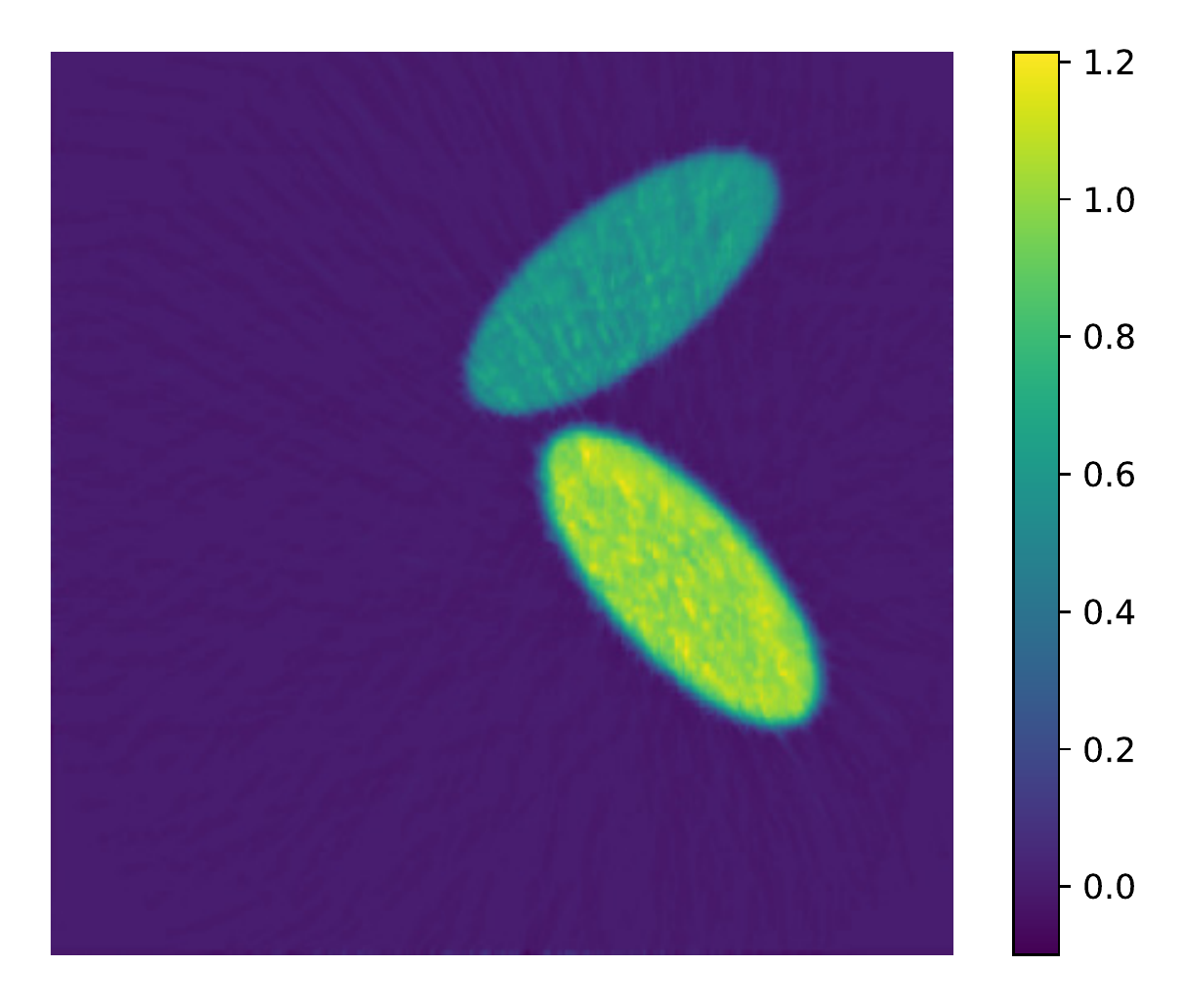}}}
	\subfigure[]{\scalebox{0.4}[0.4]{\includegraphics{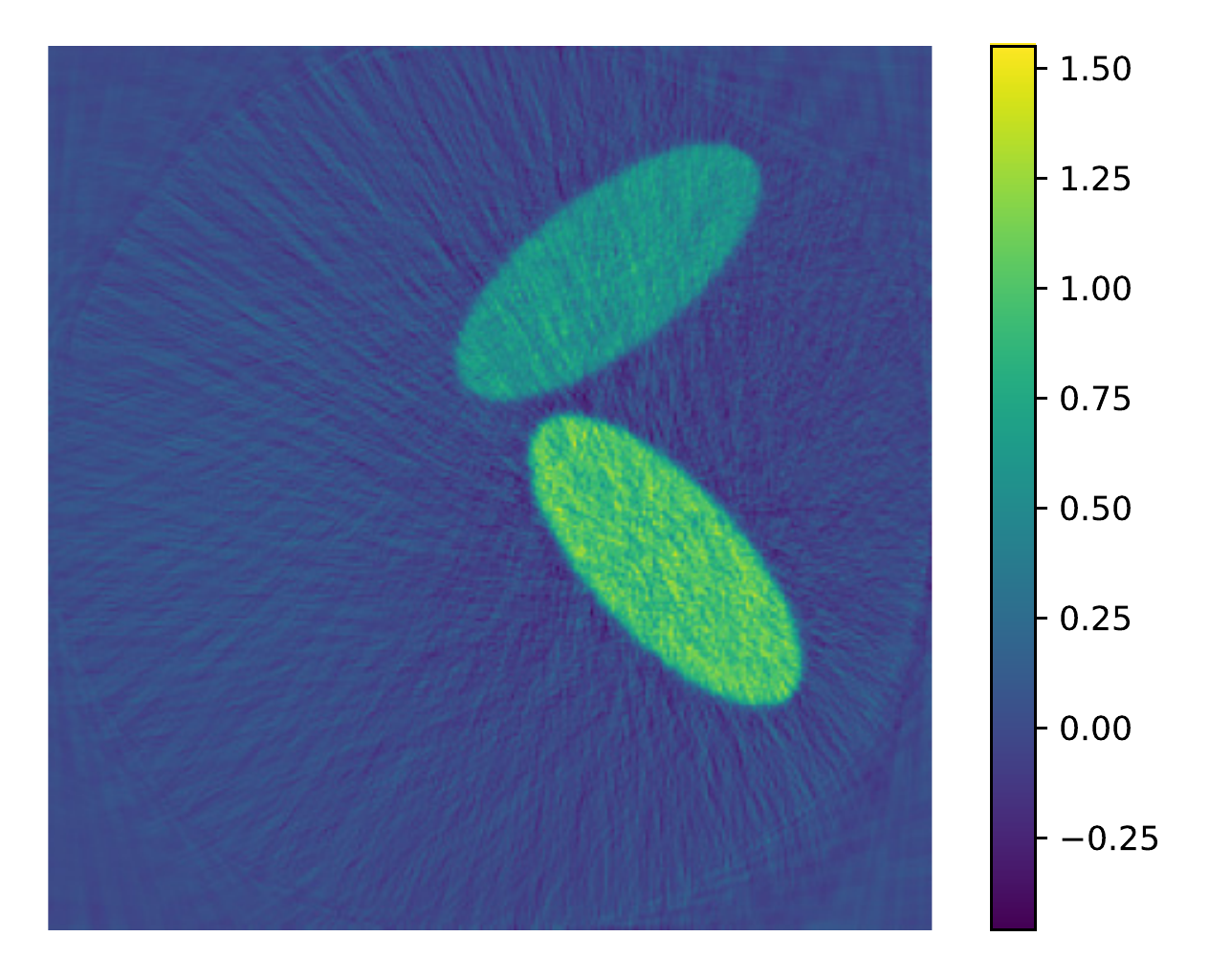}}}
	\caption{Example 1: synthetic phantom. (a) True model of attenuation image; (b) measurement data with $5\%$ Gaussian noises; (c) reconstruction result by \ref{iterNETT3}; (d) reconstruction result by ART.} \label{Fig9}
\end{figure}
\FloatBarrier
Figure \ref{Fig9}\,(b) shows the simulated measurement data, where we add $5\%$ Gaussian noises with zero mean and unit variance. The \ref{iterNETT3} algorithm is performed to reconstruct the attenuation image from the noisy measurement data. Figure \ref{Fig9}\,(c) shows the reconstruction result. Figure \ref{Fig9}\,(d) provides the reconstruction result by ART, where we have performed 5 rounds of ART iterations to get the solution. To quantitatively evaluate the performance of reconstruction, we compute the peak signal-to-noise ratio (PSNR) and the structural similarity index (SSIM), where for both a higher value means a better reconstruction (note that the range of SSIM is between $0$ and $1$); the results are shown in the second and third columns of Table \ref{TabEX1}. The \ref{iterNETT3} algorithm provides much better reconstruction for the synthetic phantom. It is capable of removing artifacts due to under-sampling and data noises while preserving the resolution of ellipses in the reconstructed image. 

As a comparison, \correct{we provide reconstruction results by a standard iterated Tikhonov (SIT) method and by the NETT algorithm \cite{li2020nett}, respectively. In the standard iterated Tikhonov (SIT) method, the regularizer $\Reg$ is taken as $\Reg(\bx)=\|\bx\|_2^2$ without the convex neural network in \ref{iterNETT3}. The iteration is the same as that of \ref{iterNETT3}, while it has a closed form in this standard situation:}
\correct{
\begin{equation}\tag*{SIT}\label{SIT}
\bx_n^\delta=\bx_{n-1}^\delta-\left(F^TF+\hat{\alpha}_n I\right)^{-1}F^T\left(F\bx_{n-1}^\delta-\by^\delta\right),
\end{equation}
}
\correct{where $F^T$ denotes the transpose of $F$, $I$ is the identity matrix, and $\hat{\alpha}_n=2M \alpha_n$; recall that $\alpha_n=2^{-n}$, and $M=256\times60$. The initial guess and the stopping rule are the same as those of \ref{iterNETT3}. Figure \ref{FigSIT} shows the reconstruction result by \ref{SIT}, and Table \ref{TabEX1} lists the values of PSNR and SSIM for this reconstruction. It shows that the \ref{SIT} algorithm has a better performance than ART, but the improvement is inadequate comparing to the performance of \ref{iterNETT3}.}

\correct{In the NETT algorithm,} the attenuation image is recovered by solving the following minimization problem,
\begin{equation}\tag*{NETT}\label{NETT2}
\bx_\alpha^\delta=\underset{\bx\in X}{\argmin}\frac{1}{2M}\|F \bx - \by^\delta\|_2^2+\alpha\|\Phi_{\Theta}(\bx) \|_2^2,
\end{equation}
where $\Phi_\Theta$ denotes the general U-net trained in section \ref{AppTrainUnet}, and $\alpha$ is a fixed parameter controlling the amount of regularization. As stated above, $\Phi_\Theta$ is trained without the constraint of convexity, since the \ref{NETT2} algorithm does not require convexity of neural network in general. We perform the \ref{NETT2} algorithm with a series of different values of $\alpha$, although we are not trying to tune the parameter $\alpha$ exhaustively. Figure \ref{Fig10} shows the reconstruction results with $\alpha=0.001, 0.01, 0.05$ and $0.1$, respectively. Table \ref{TabEX1} lists the values of PSNR and SSIM for those reconstruction results. It shows that the \ref{NETT2} algorithm achieves the best reconstruction among the four solutions as $\alpha=0.05$. Comparing the results in Figure \ref{Fig9}\,(c) and Figure \ref{Fig10}, and considering the values of PSNR and SSIM listed in Table \ref{TabEX1}, we conclude that the \ref{iterNETT3} algorithm can achieve a comparable reconstruction as \ref{NETT2} without the procedure of tuning the regularization parameter manually.

\begin{figure}[h]
	\centering
	{\scalebox{0.4}[0.4]{\includegraphics{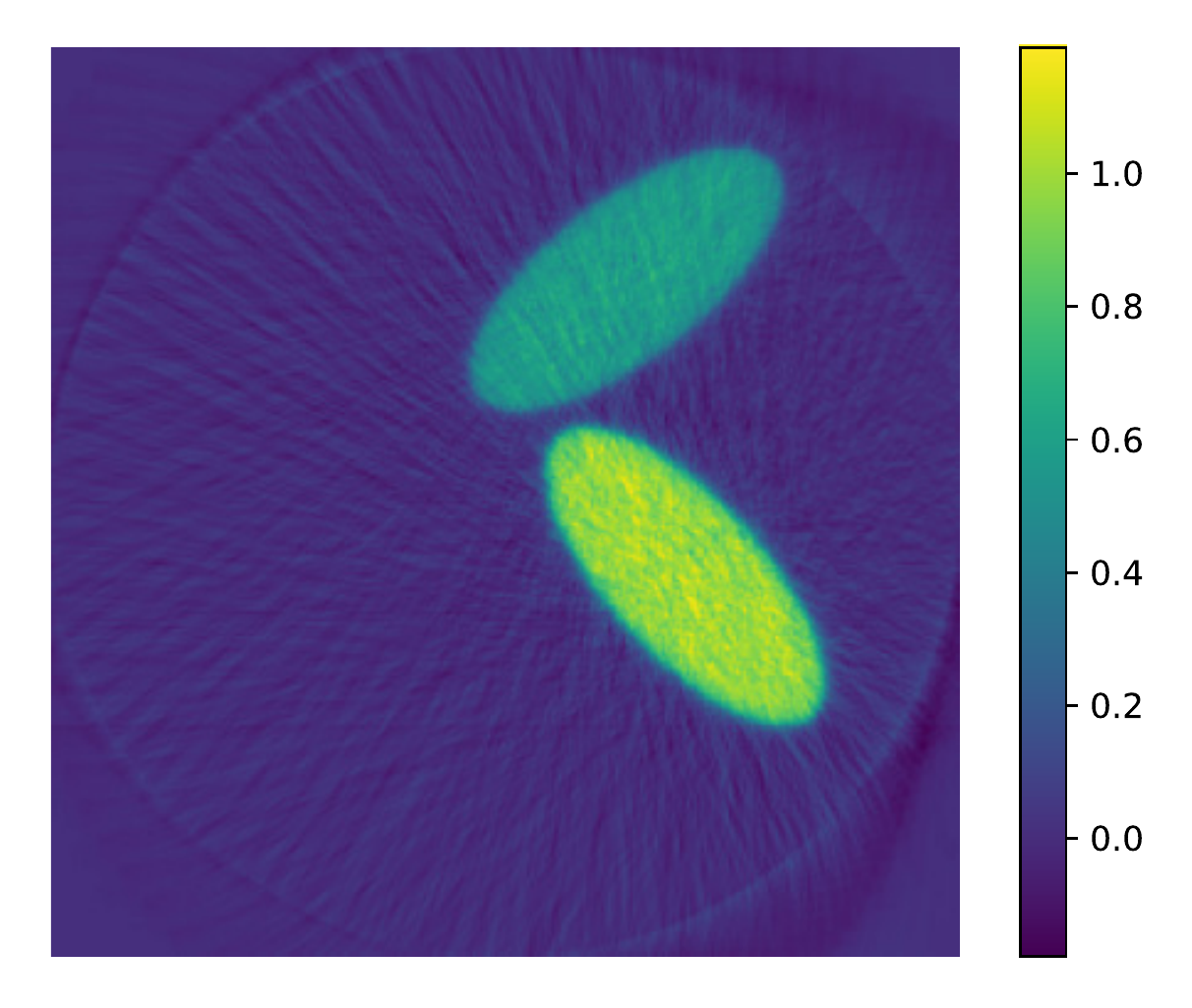}}}
	\caption{\correct{Example 1: synthetic phantom. Reconstruction result by the standard iterated Tikhonov (SIT) method.}} \label{FigSIT}
\end{figure}

\begin{figure}[h]
	\centering
	\subfigure[]{\scalebox{0.4}[0.4]{\includegraphics{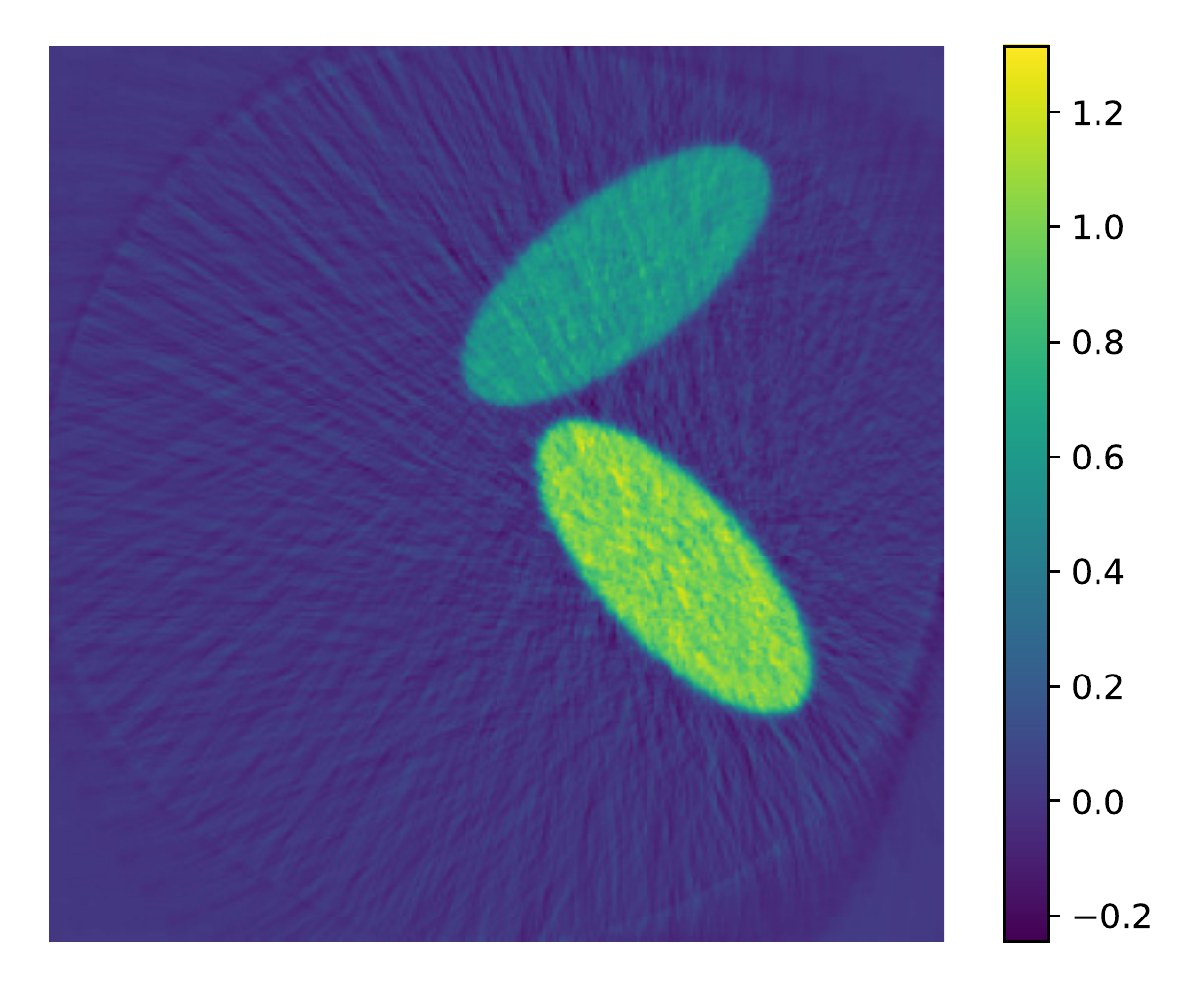}}}
	\subfigure[]{\scalebox{0.4}[0.4]{\includegraphics{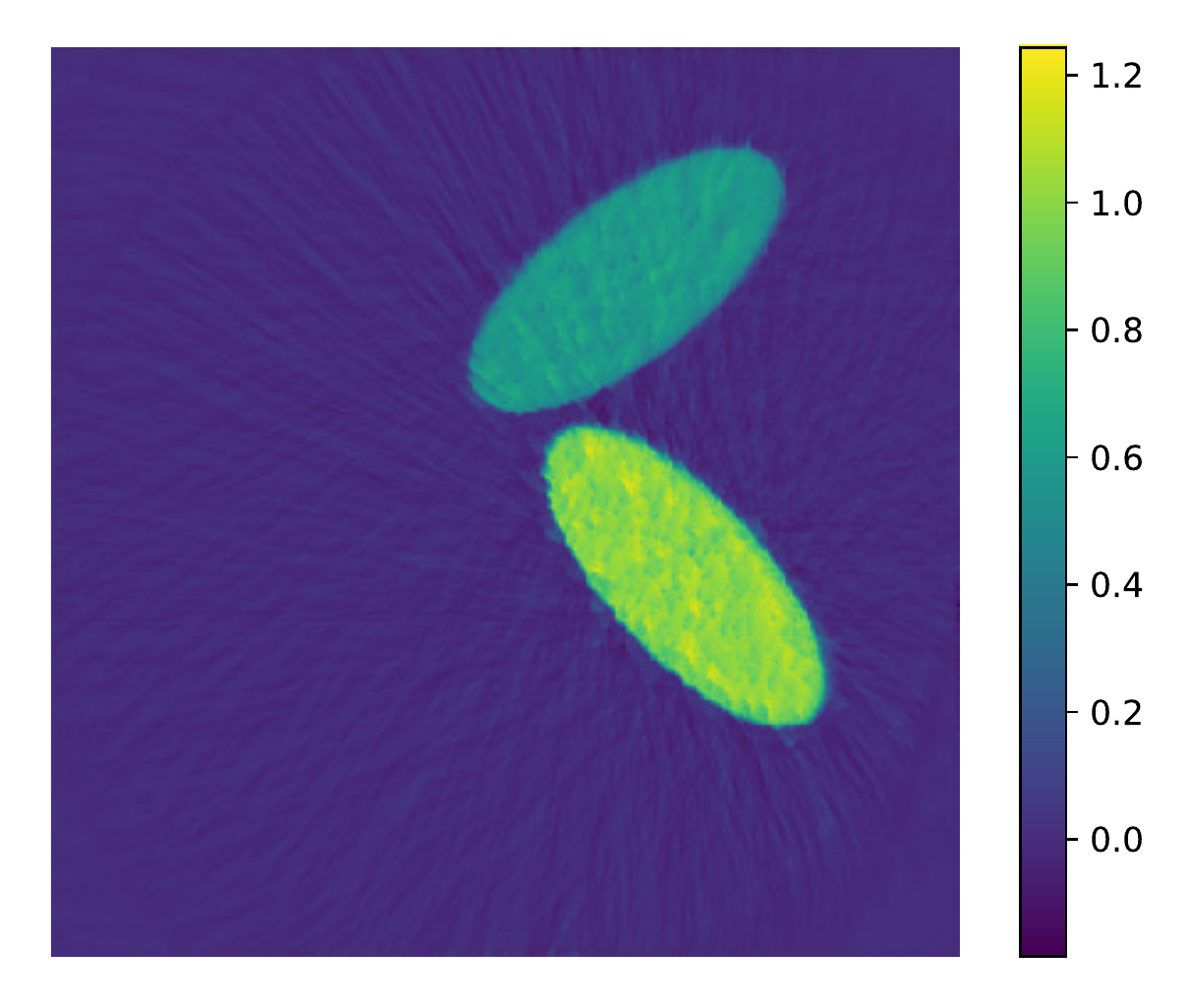}}}\\
	\subfigure[]{\scalebox{0.4}[0.4]{\includegraphics{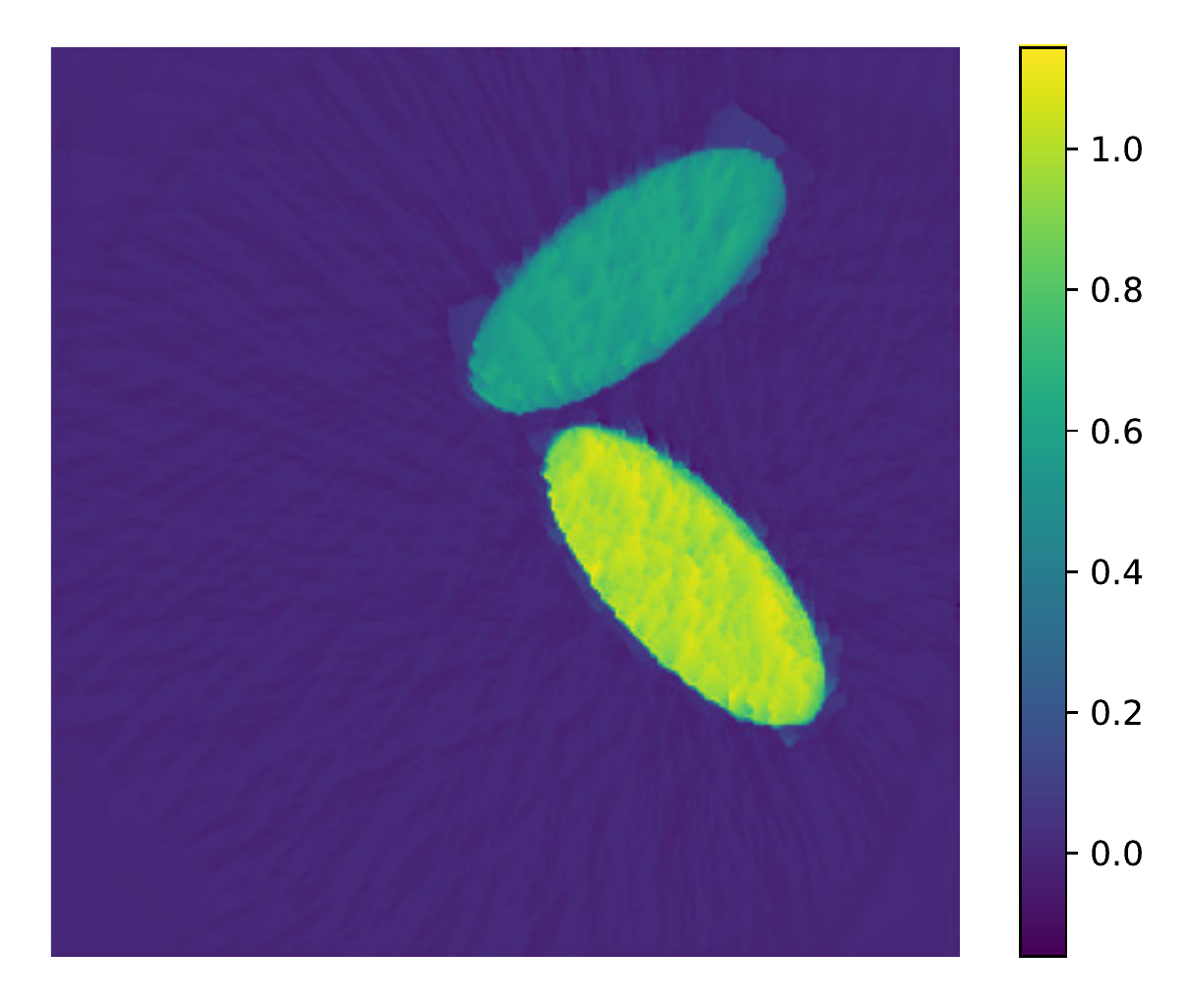}}}
	\subfigure[]{\scalebox{0.4}[0.4]{\includegraphics{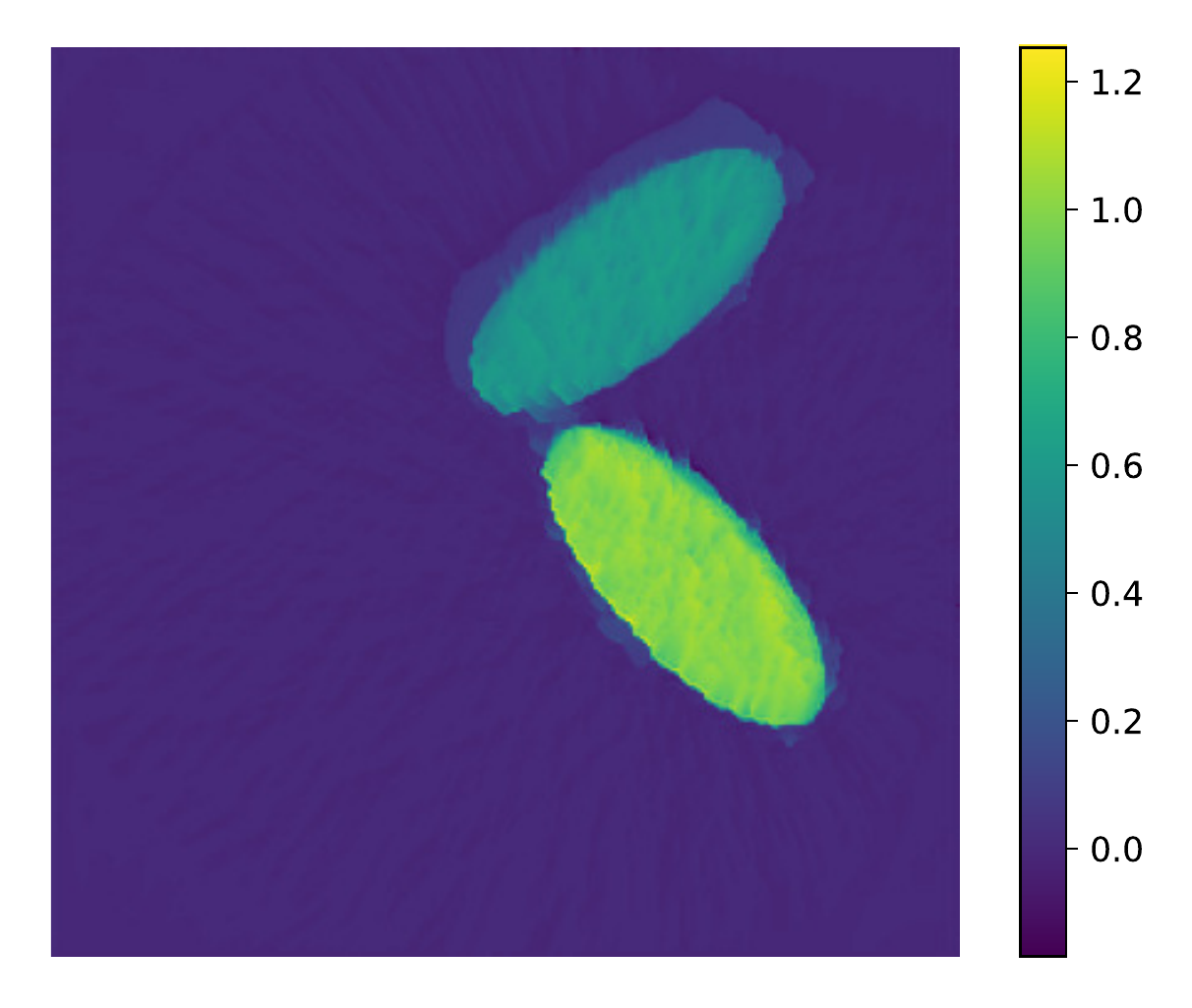}}}
	\caption{Example 1: synthetic phantom. Reconstruction results by \ref{NETT2} with $\alpha=0.001, 0.01, 0.05$ and $0.1$, respectively. (a) $\alpha=0.001$; (b) $\alpha=0.01$; (c) $\alpha=0.05$; (d) $\alpha=0.1$.} \label{Fig10}
\end{figure}

	\begin{table}[H]
		\centering
		\begin{tabular}{cccccccc} 
			\toprule
			\multirow{2}{*}{} \hspace{2.5pt} & \hspace{2.5pt} \multirow{2}{*}{iNETT} \hspace{2.5pt} & \hspace{2.5pt} \multirow{2}{*}{ART} & \hspace{2.5pt} \multirow{2}{*}{SIT} \hspace{2.5pt} & \multicolumn{4}{c}{\begin{tabular}[c]{@{}c@{}}NETT \end{tabular}}  \\ 
			\cdashline{5-8}
			&          \hspace{2.5pt}           &                     &                 & $\alpha=0.001$ & $\alpha=0.01$  & $\alpha=0.05$  & $\alpha=0.1$                     \\ 
			\midrule
			PSNR      \hspace{2.5pt}        & \textbf{21.30}   & 14.48  & 17.97    & 16.85 & 18.23 & 19.34 & 18.67                   \\
			SSIM      \hspace{2.5pt}        & \textbf{0.84}     & 0.54    & 0.66     & 0.62  & 0.73  & 0.79  & 0.79                          \\
			\bottomrule
		\end{tabular}\caption{\correct{(Example 1) PSNR and SSIM for the reconstruction results by \ref{iterNETT3}, ART, \ref{SIT} and \ref{NETT2}, respectively.}}
		\label{TabEX1}
	\end{table}

\subsubsection{\correct{Example 2: synthetic phantom with overlapping ellipses}}
\correct{Figure \ref{FigEX2_1}\,(a) shows the true model of attenuation image. Again, it is a synthetic phantom of the same type as the phantom images in the training and validation sets, but it is not contained in them. }
\begin{figure}[!htbp]
	\centering
	\subfigure[]{\scalebox{0.4}[0.4]{\includegraphics{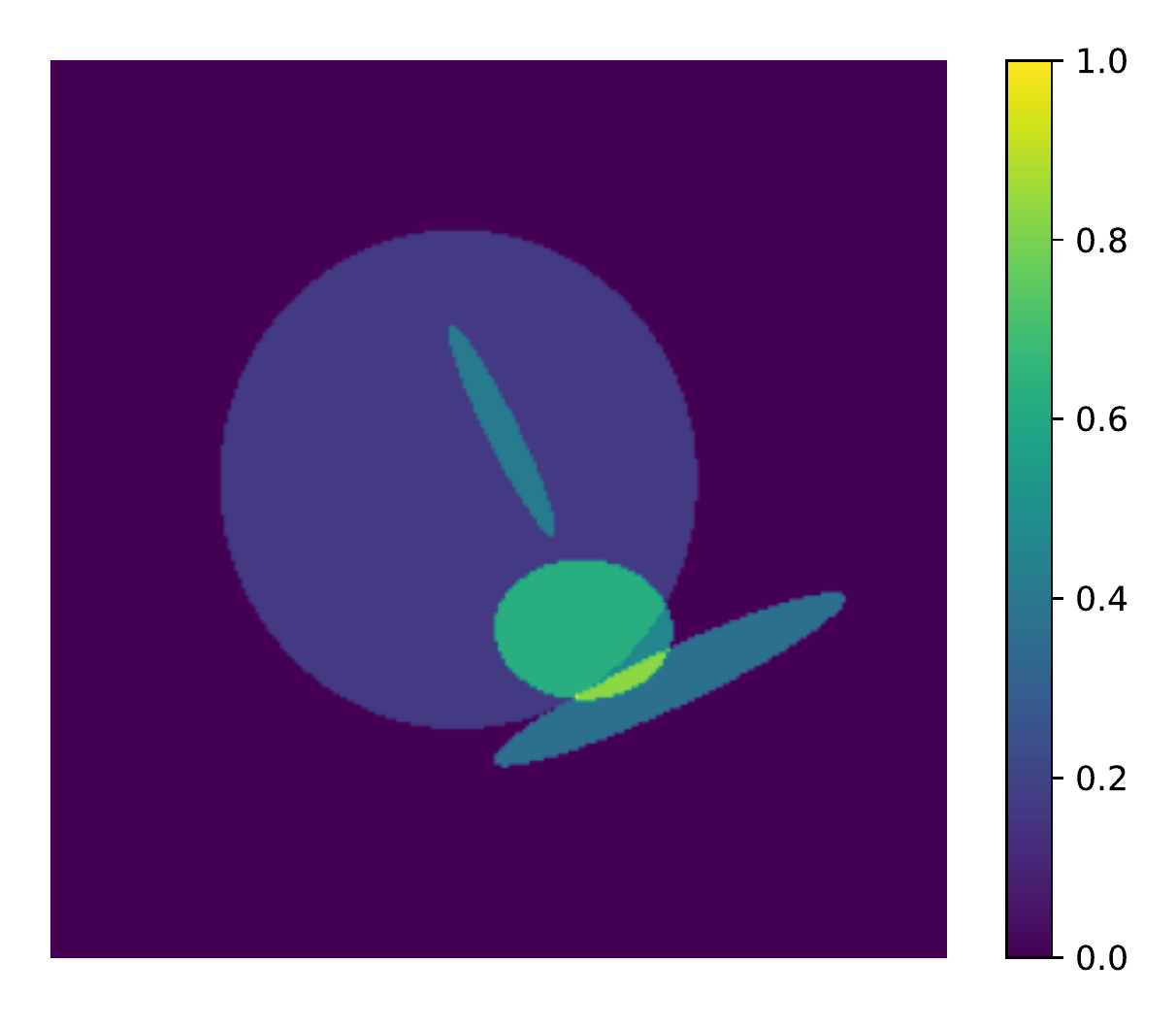}}}
	\subfigure[]{\scalebox{0.4}[0.4]{\includegraphics{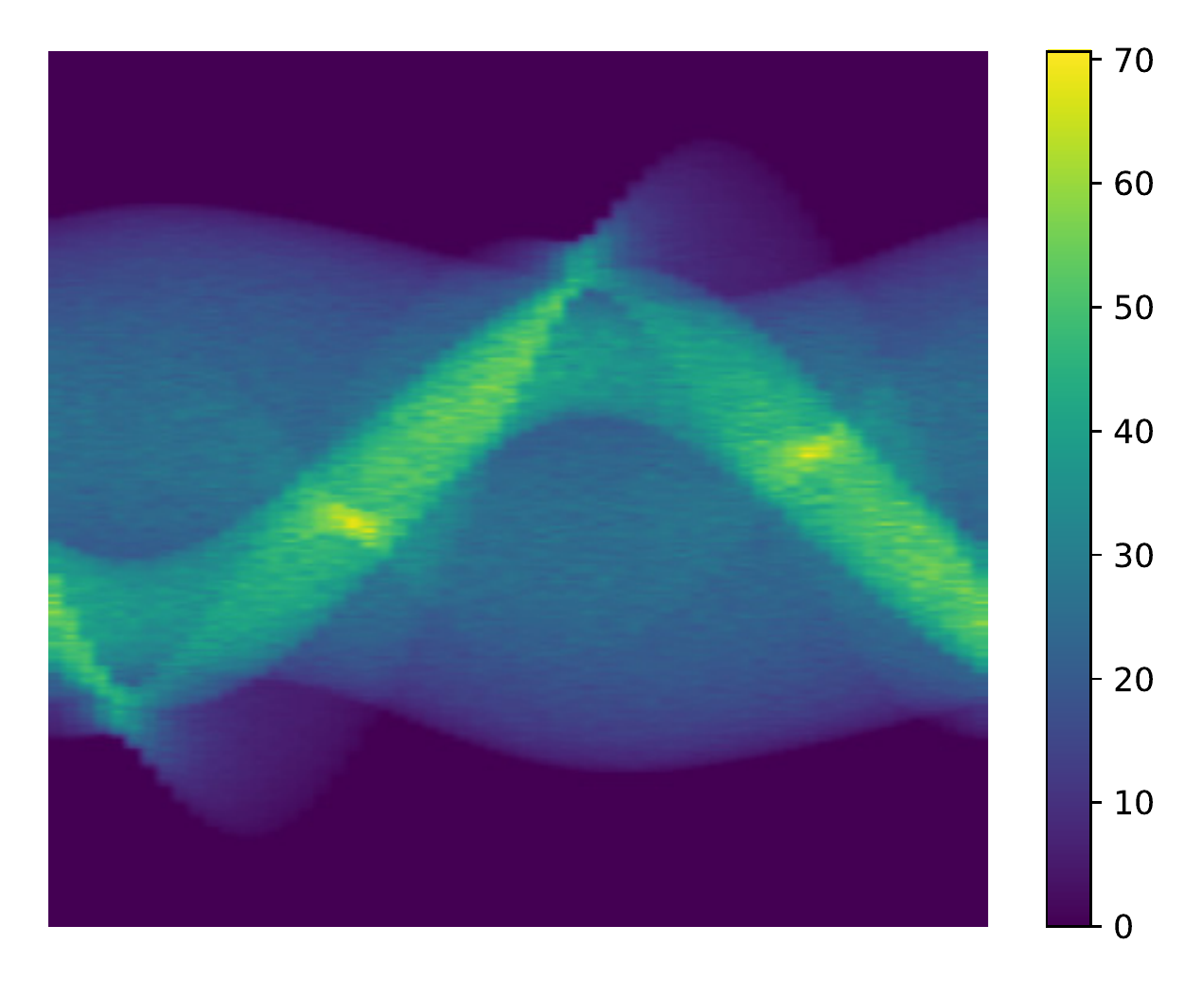}}}\\
	\subfigure[]{\scalebox{0.4}[0.4]{\includegraphics{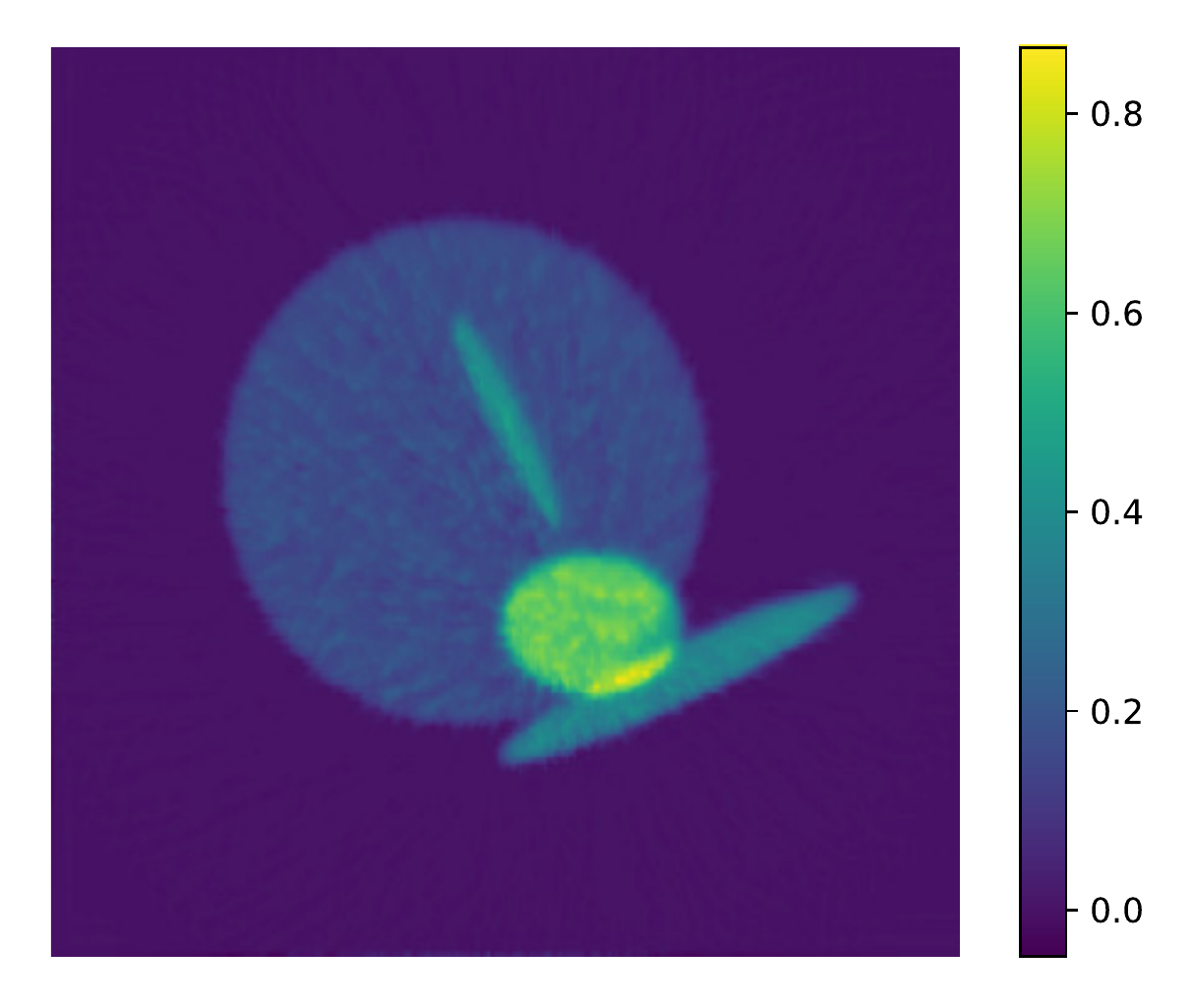}}}
	\subfigure[]{\scalebox{0.4}[0.4]{\includegraphics{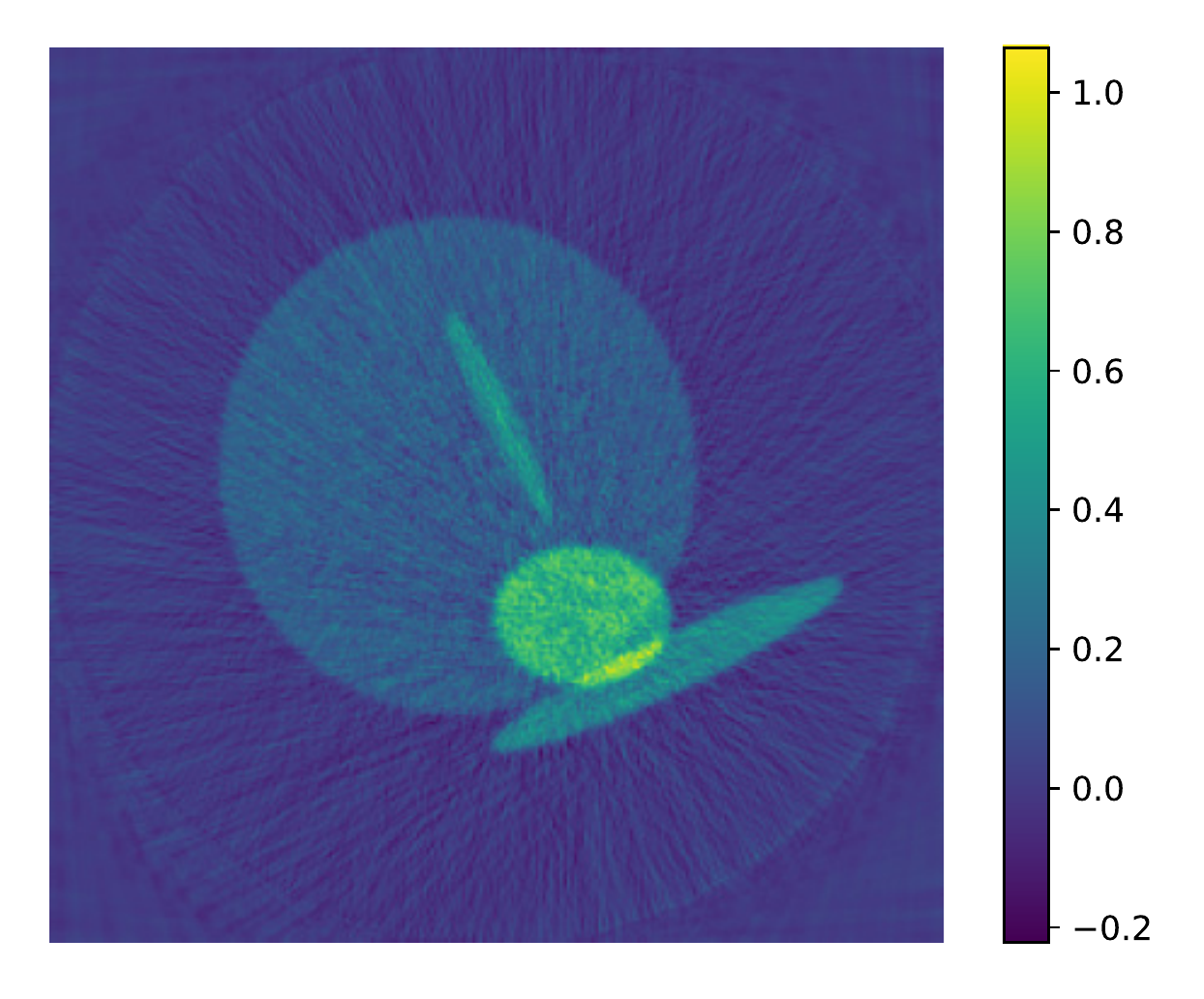}}}
	\caption{\correct{Example 2: synthetic phantom. (a) True model of attenuation image; (b) measurement data with $5\%$ Gaussian noises; (c) reconstruction result by \ref{iterNETT3}; (d) reconstruction result by ART.}} \label{FigEX2_1}
\end{figure}

\correct{Figure \ref{FigEX2_1}\,(b) shows the simulated measurement data, where we add $5\%$ Gaussian noises with zero mean and unit variance. Figure \ref{FigEX2_1}\,(c) shows the reconstruction result by the \ref{iterNETT3} algorithm. Figure \ref{FigEX2_1}\,(d) provides the reconstruction result by ART, where we have performed 5 rounds of ART iterations to get the solution. To quantitatively evaluate the performance of reconstruction, we compute the values of PSNR and SSIM as well; the results are shown in the second and third columns of Table \ref{TabEX2}. It shows that the \ref{iterNETT3} algorithm provides much better reconstruction for the synthetic phantom. It is capable of removing artifacts in the reconstructed image and achieving higher values of PSNR and SSIM.}

\correct{As a comparison, we provide reconstruction results by \ref{SIT} and \ref{NETT2}, respectively. The solution by \ref{SIT} is shown in Figure \ref{FigEX2_2}, and the values of PSNR and SSIM are listed in the fourth column of Table \ref{TabEX2}. The \ref{SIT} algorithm has a better performance than ART, but the improvement is inadequate comparing to the performance of \ref{iterNETT3}. The reconstruction results by \ref{NETT2} are shown in Figure \ref{FigEX2_3}, where we perform the \ref{NETT2} algorithm with $\alpha=0.001, 0.01, 0.05$ and $0.1$, respectively. The PSNR and SSIM of these reconstruction results are listed in Table \ref{TabEX2}. Comparing the results in Figure \ref{FigEX2_1}\,(c) and Figure \ref{FigEX2_3}, and considering the values of PSNR and SSIM listed in Table \ref{TabEX2}, we conclude that the \ref{iterNETT3} algorithm can achieve a comparable (if not better) reconstruction as \ref{NETT2} without tuning the regularization parameter exhaustively.
}

\begin{figure}[H]
	\centering
	{\scalebox{0.4}[0.4]{\includegraphics{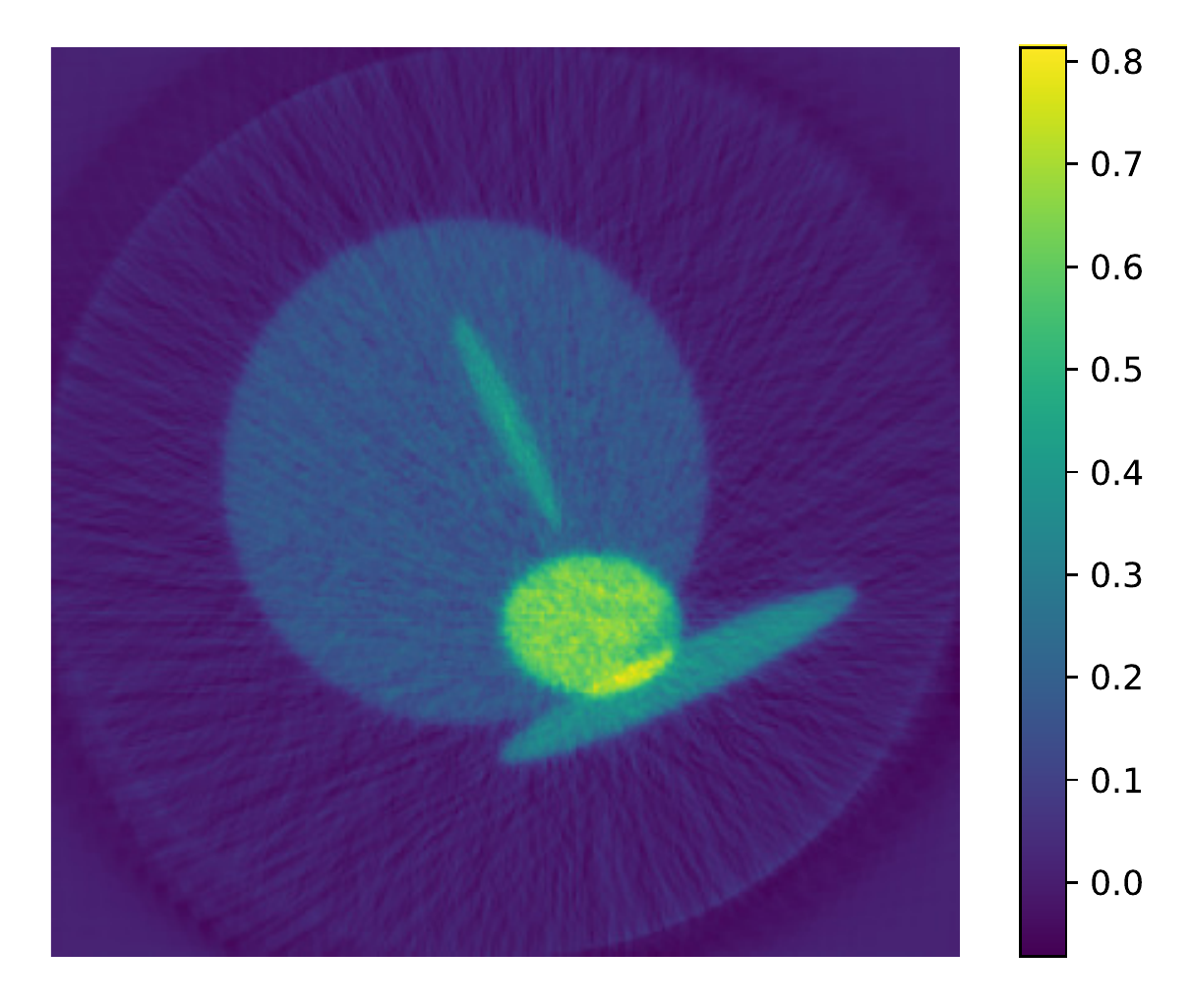}}}
	\caption{\correct{Example 2: synthetic phantom. Reconstruction result by the standard iterated Tikhonov (SIT) method.}} \label{FigEX2_2}
\end{figure}

\begin{figure}[H]
	\centering
	\subfigure[]{\scalebox{0.4}[0.4]{\includegraphics{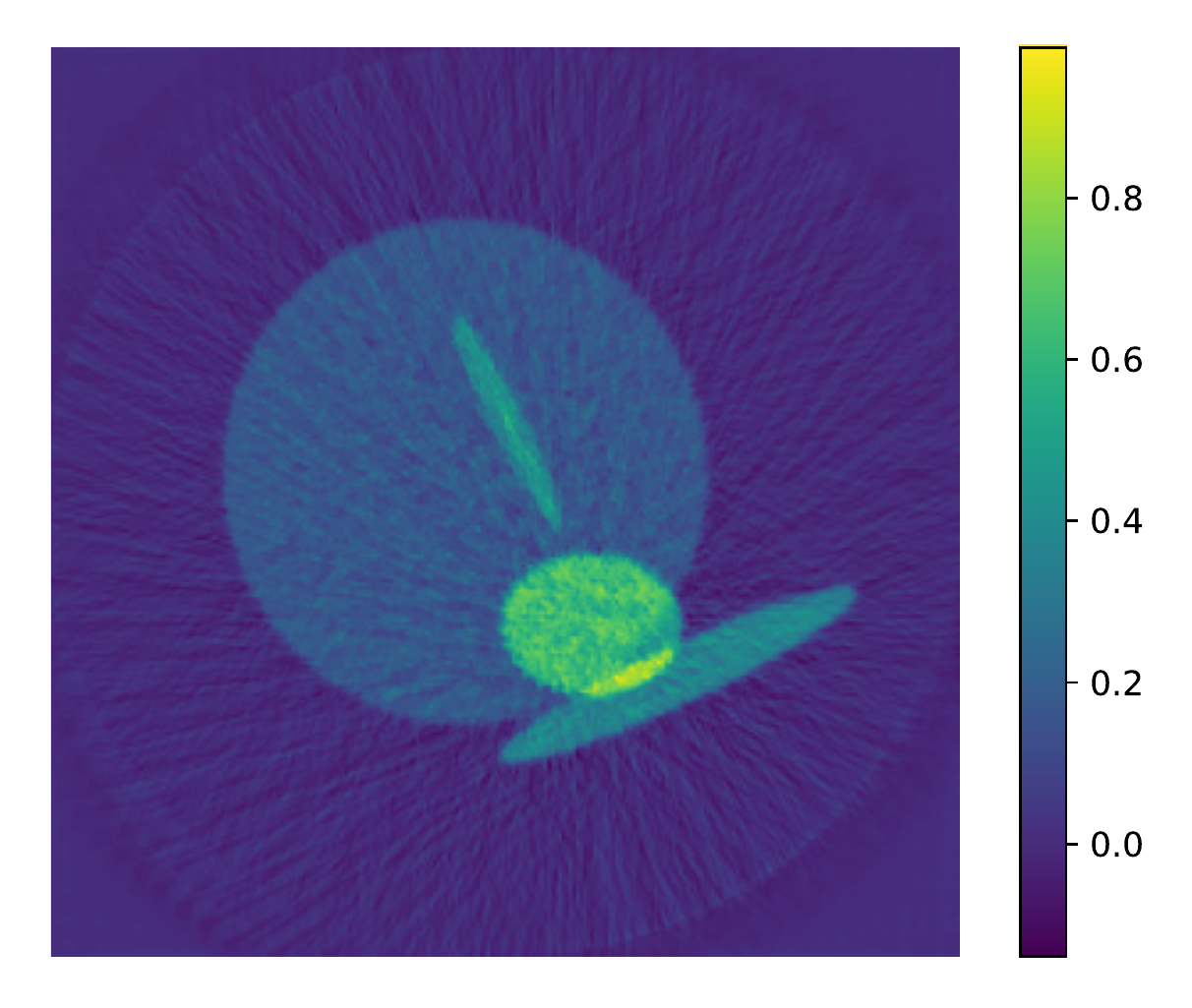}}}
	\subfigure[]{\scalebox{0.4}[0.4]{\includegraphics{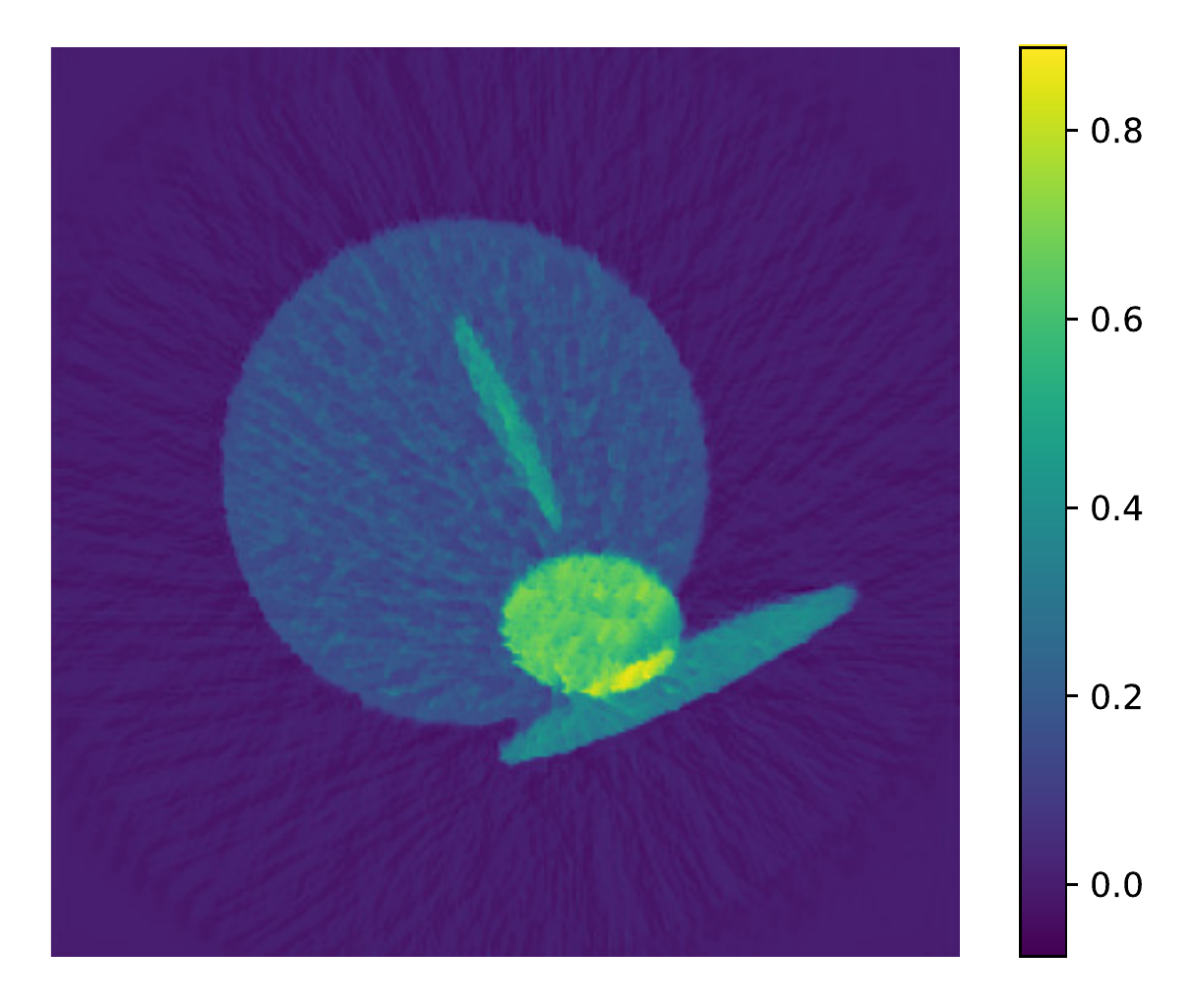}}}\\
	\subfigure[]{\scalebox{0.4}[0.4]{\includegraphics{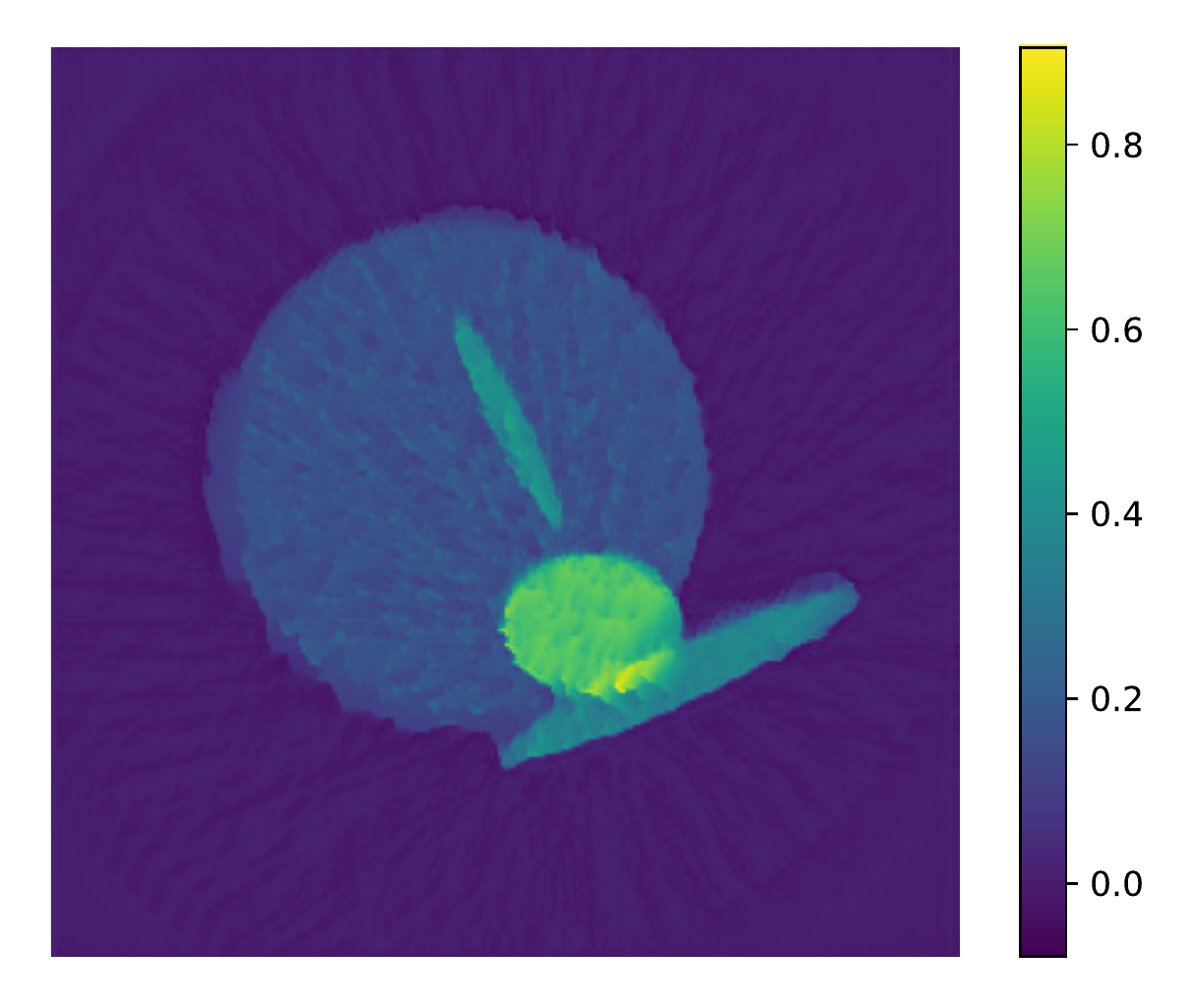}}}
	\subfigure[]{\scalebox{0.4}[0.4]{\includegraphics{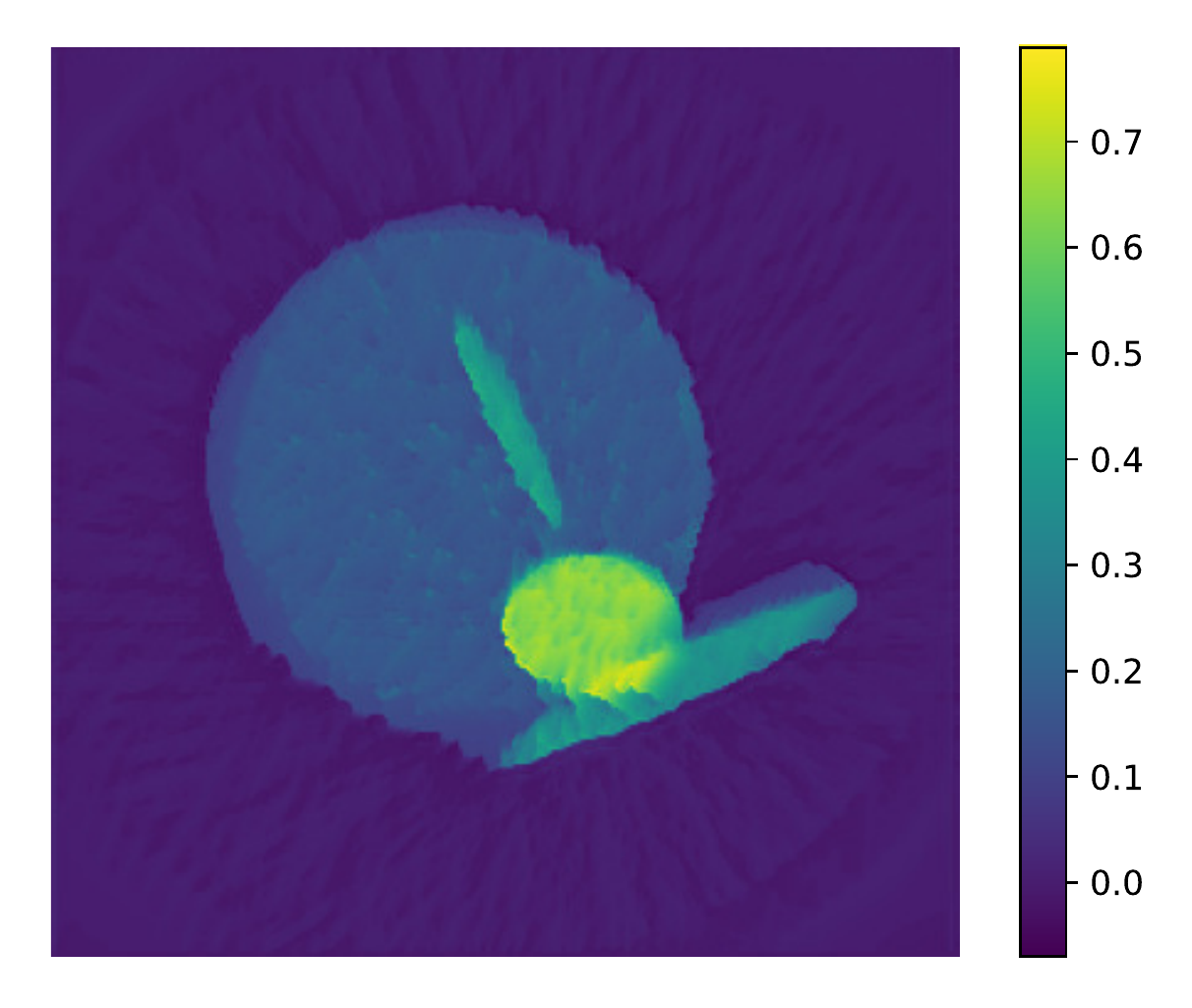}}}
	\caption{\correct{Example 2: synthetic phantom. Reconstruction results by \ref{NETT2} with $\alpha=0.001, 0.01, 0.05$ and $0.1$, respectively. (a) $\alpha=0.001$; (b) $\alpha=0.01$; (c) $\alpha=0.05$; (d) $\alpha=0.1$.}} \label{FigEX2_3}
\end{figure}

\begin{table}[H]
		\centering
		\begin{tabular}{cccccccc} 
			\toprule
			\multirow{2}{*}{} \hspace{2.5pt} & \hspace{2.5pt} \multirow{2}{*}{iNETT} \hspace{2.5pt} & \hspace{2.5pt} \multirow{2}{*}{ART} & \hspace{2.5pt} \multirow{2}{*}{SIT} \hspace{2.5pt} & \multicolumn{4}{c}{\begin{tabular}[c]{@{}c@{}}NETT \end{tabular}}  \\ 
			\cdashline{5-8}
			&          \hspace{2.5pt}           &                     &                 & $\alpha=0.001$ & $\alpha=0.01$  & $\alpha=0.05$  & $\alpha=0.1$                     \\ 
			\midrule
			PSNR      \hspace{2.5pt}        & \textbf{25.35}   & 17.26  & 21.85    & 19.51 & 22.46 & 22.36 & 21.66                   \\
			SSIM      \hspace{2.5pt}        & \textbf{0.91}     & 0.63    & 0.77     & 0.70  & 0.81  & 0.84  & 0.83                          \\
			\bottomrule
		\end{tabular}\caption{\correct{(Example 2) PSNR and SSIM for the reconstruction results by \ref{iterNETT3}, ART, \ref{SIT} and \ref{NETT2}, respectively.}}
		\label{TabEX2}
\end{table}

\subsubsection{Example 3: attenuation images of lung and myocardium}
In this example, we consider reconstructions of attenuation images which are quite different from the synthetic phantoms in the training process. 
\begin{figure}[H]
	\centering
	\subfigure[]{\scalebox{0.4}[0.4]{\includegraphics{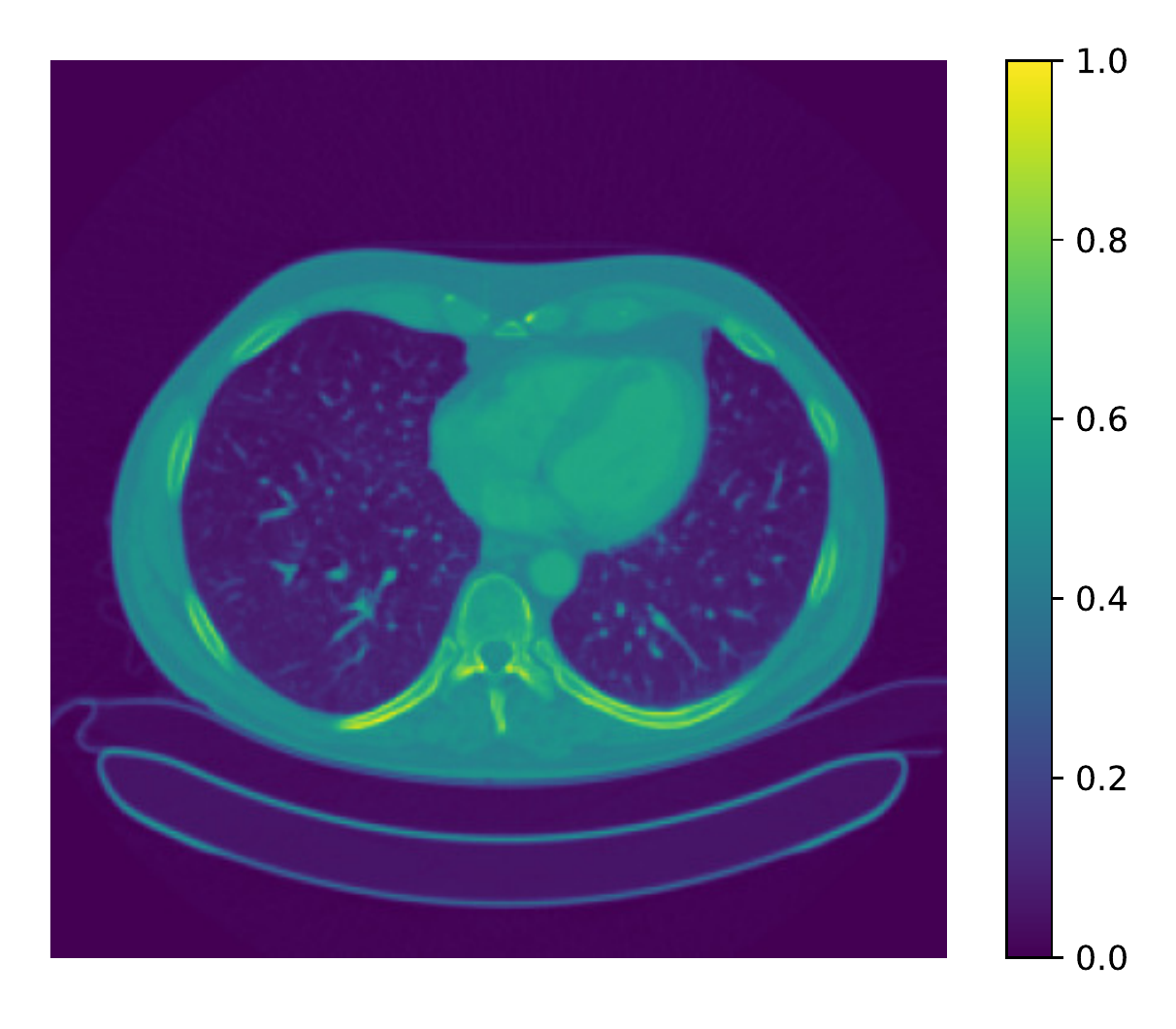}}}
	\subfigure[]{\scalebox{0.4}[0.4]{\includegraphics{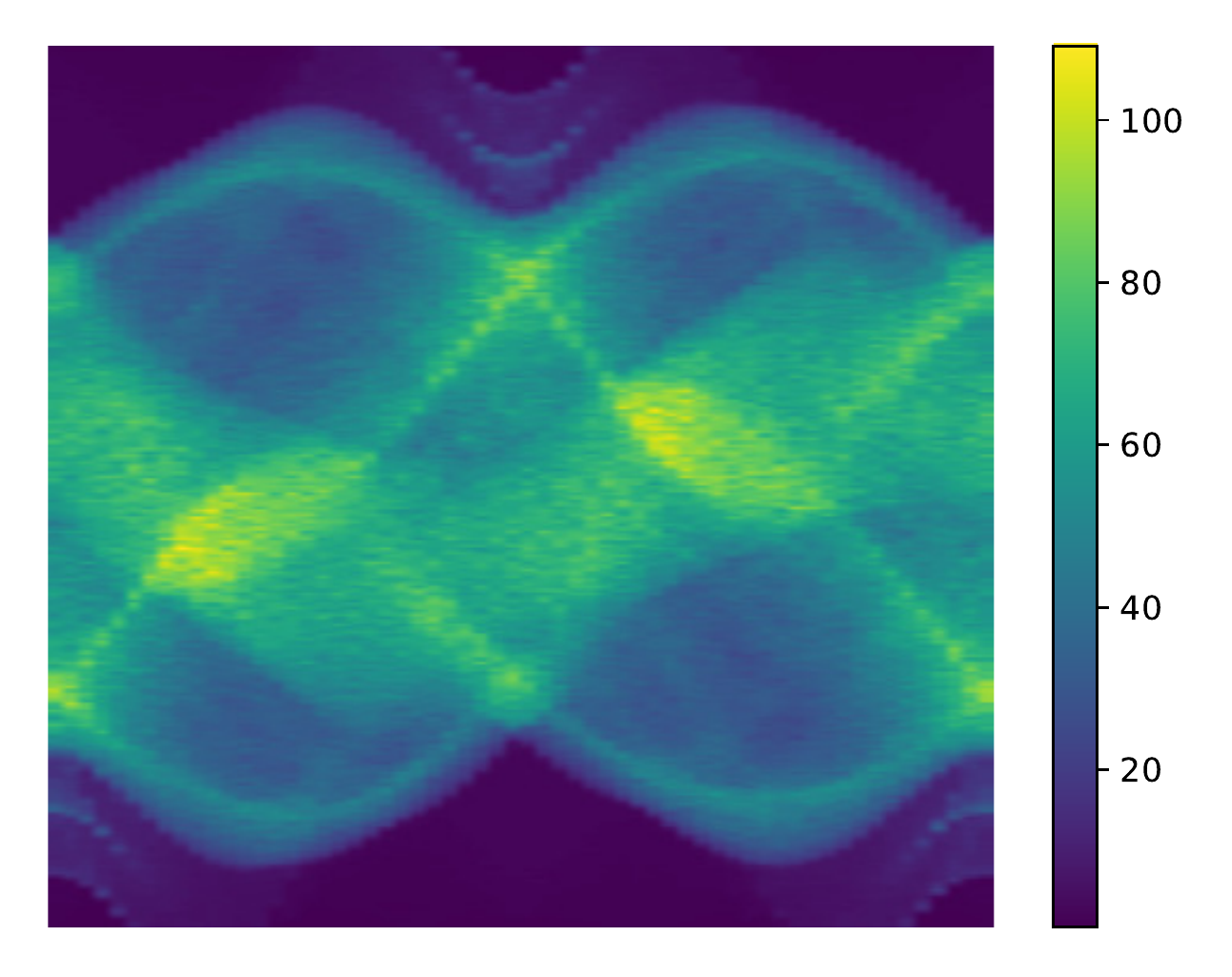}}}\\
	\subfigure[]{\scalebox{0.4}[0.4]{\includegraphics{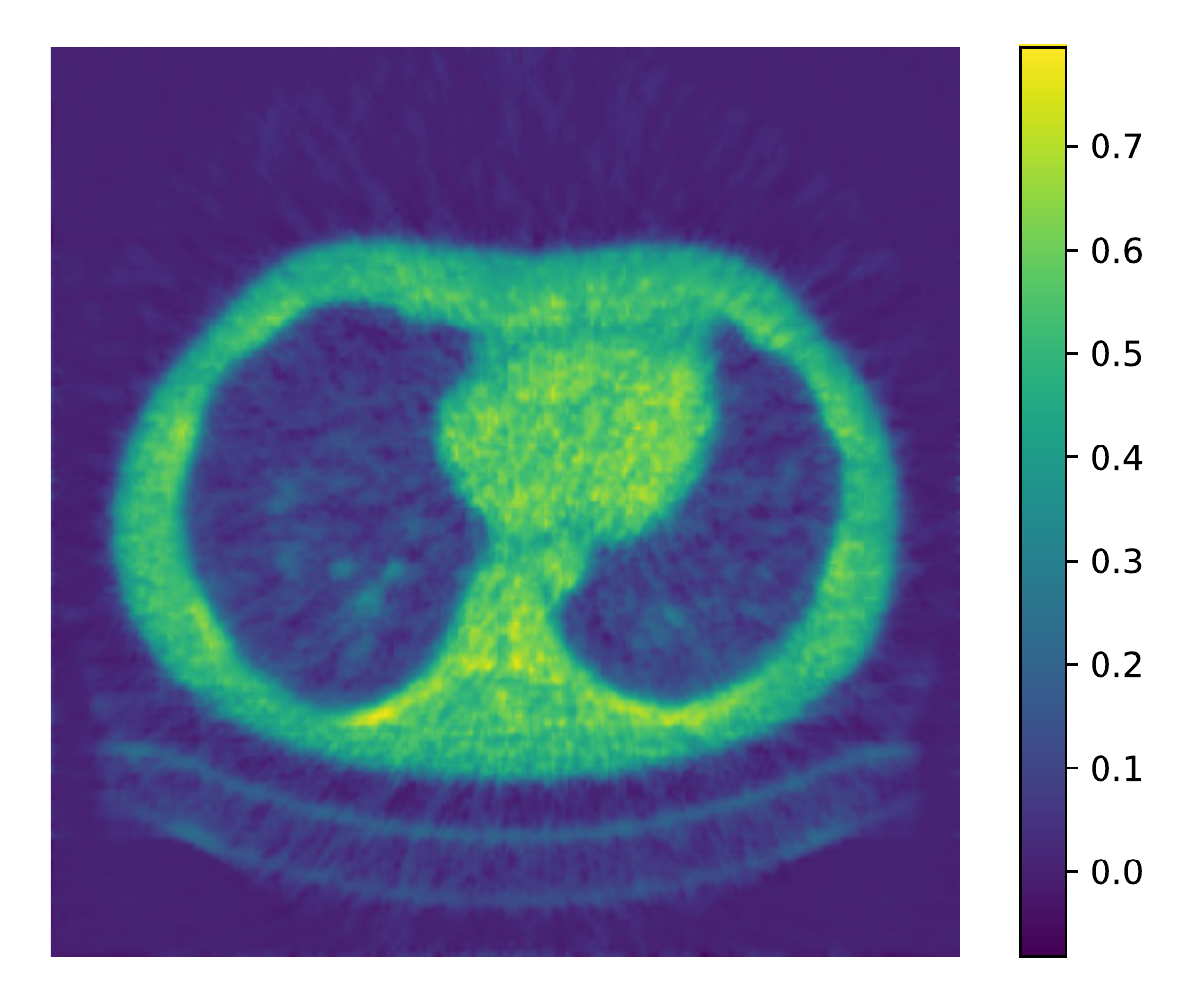}}}
	\subfigure[]{\scalebox{0.4}[0.4]{\includegraphics{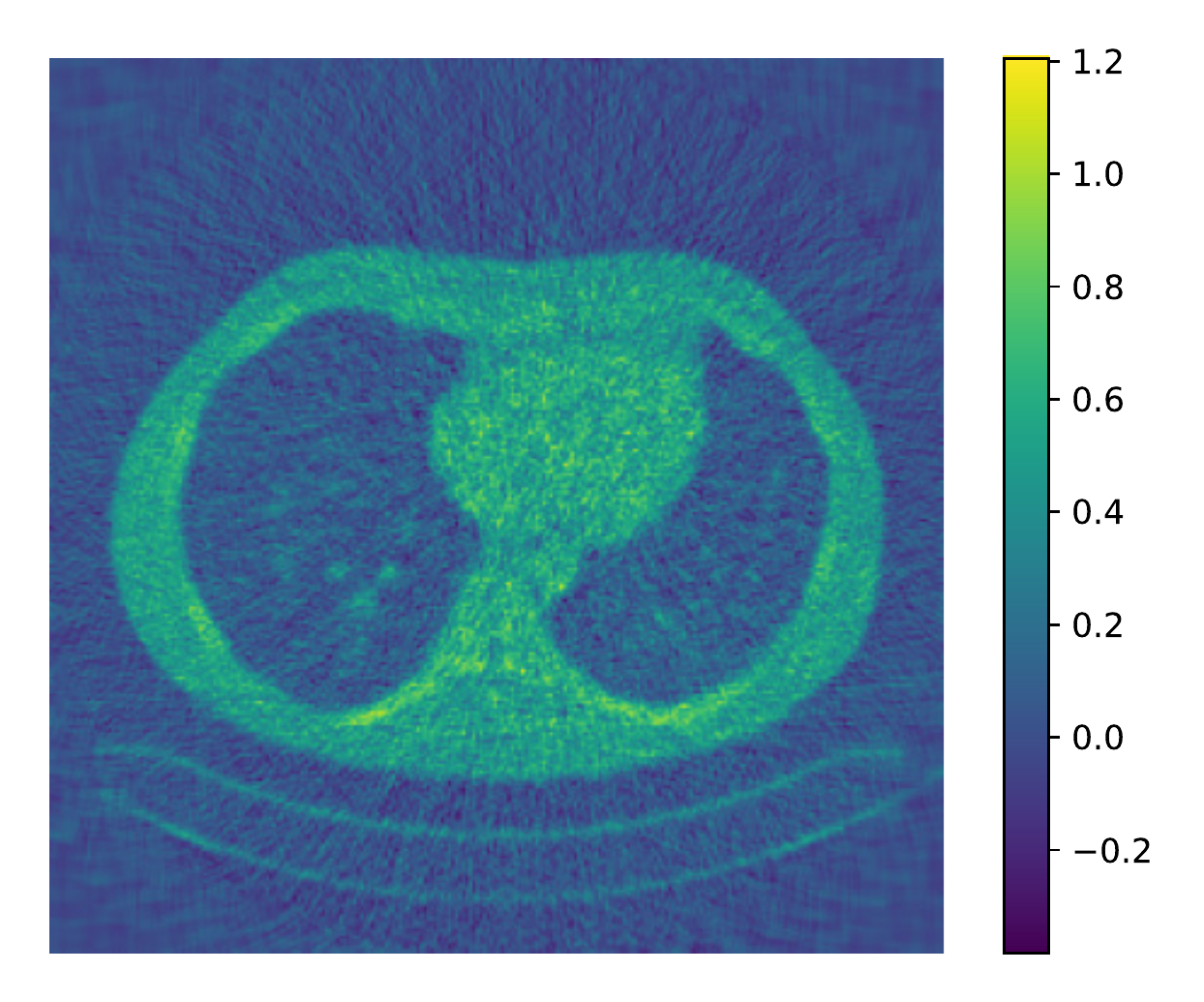}}}
	\caption{Example 3: lung imaging. (a) Attenuation image of lung; (b) measurement data with $5\%$ Gaussian noises; (c) reconstruction result by \ref{iterNETT3}; (d) reconstruction result by ART.} \label{Fig11}
\end{figure}

The true models of attenuation images are shown in Figure \ref{Fig11}\,(a) and Figure \ref{Fig12}\,(a), respectively, where Figure \ref{Fig11}\,(a) shows the attenuation image of lung, and Figure \ref{Fig12}\,(a) shows the attenuation image of myocardium. In the \ref{iterNETT3} algorithm, we employ the same convex U-net $\Phi^c_{\Theta}$ trained on synthetic phantoms as described in section \ref{AppTrainUnet}. The reconstruction results of lung imaging are provided in Figure \ref{Fig11}, where Figure \ref{Fig11}\,(b) shows the measurement data with $5\%$ Gaussian noises, Figure \ref{Fig11}\,(c) shows the recovered solution of \ref{iterNETT3}, and Figure \ref{Fig11}\,(d) shows the recovered solution of ART. Correspondingly, the reconstruction results of myocardium imaging are provided in Figure \ref{Fig12}.

\begin{figure}[!htbp]
	\centering
	\subfigure[]{\scalebox{0.4}[0.4]{\includegraphics{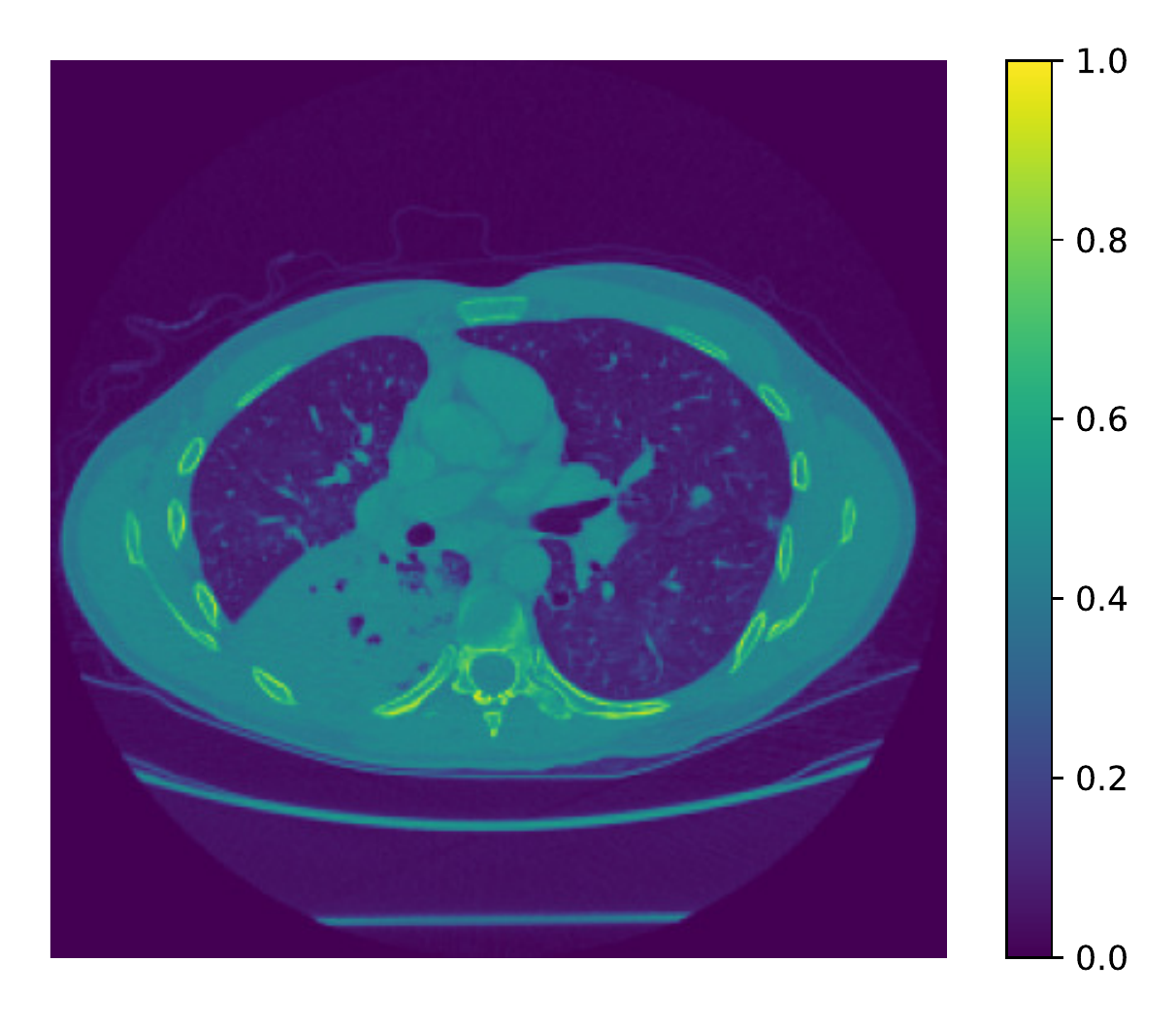}}}
	\subfigure[]{\scalebox{0.4}[0.4]{\includegraphics{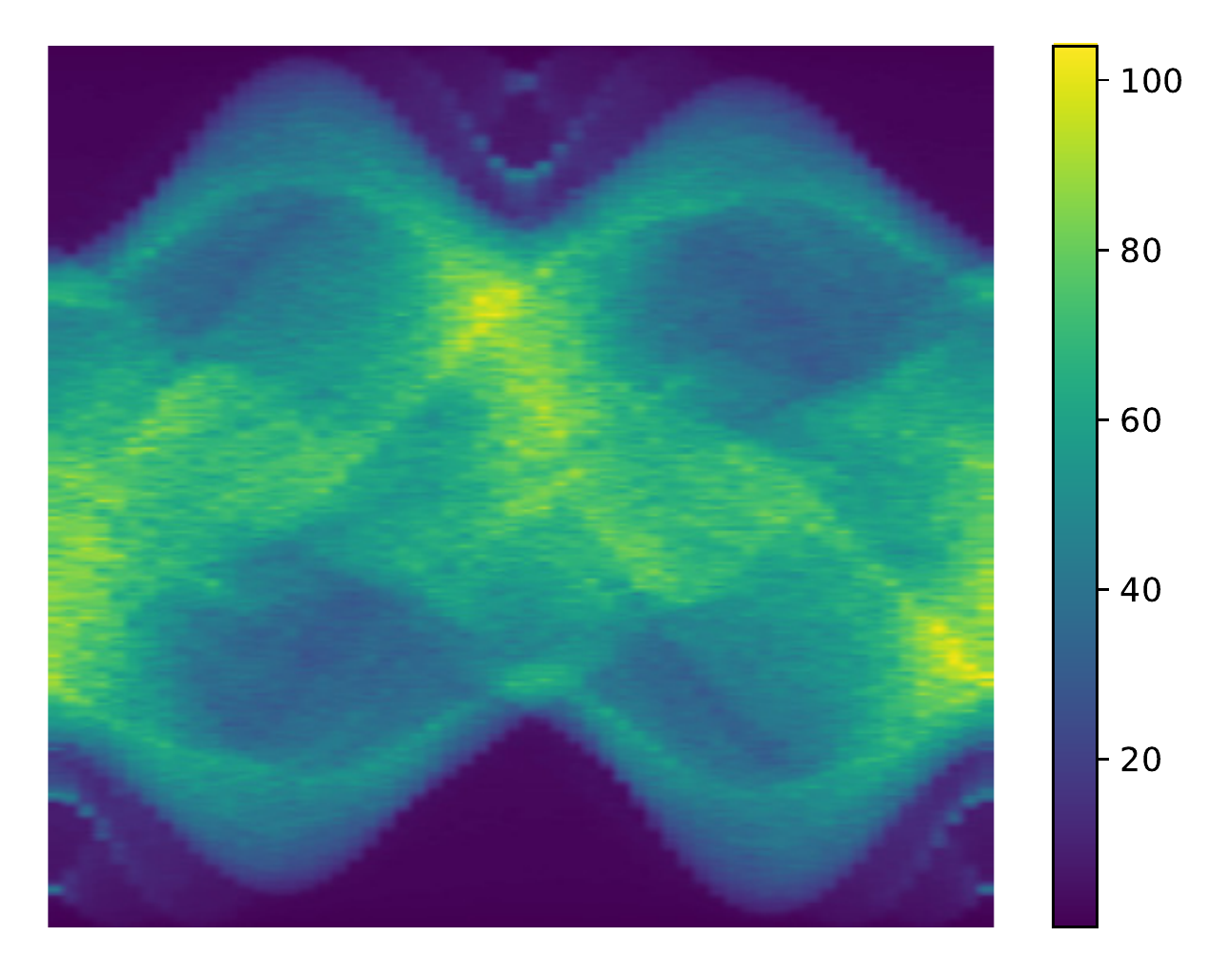}}}\\
	\subfigure[]{\scalebox{0.4}[0.4]{\includegraphics{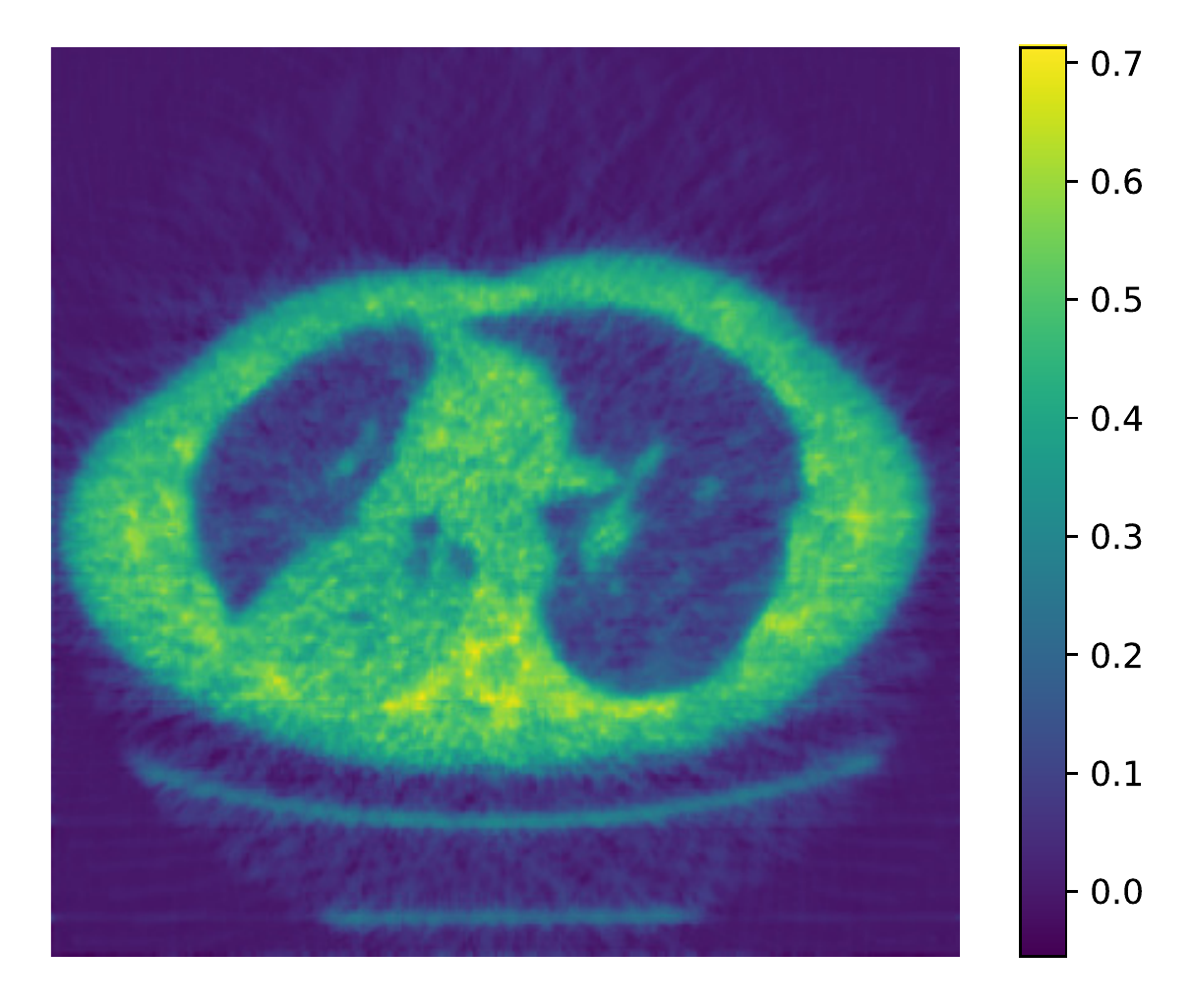}}}
	\subfigure[]{\scalebox{0.4}[0.4]{\includegraphics{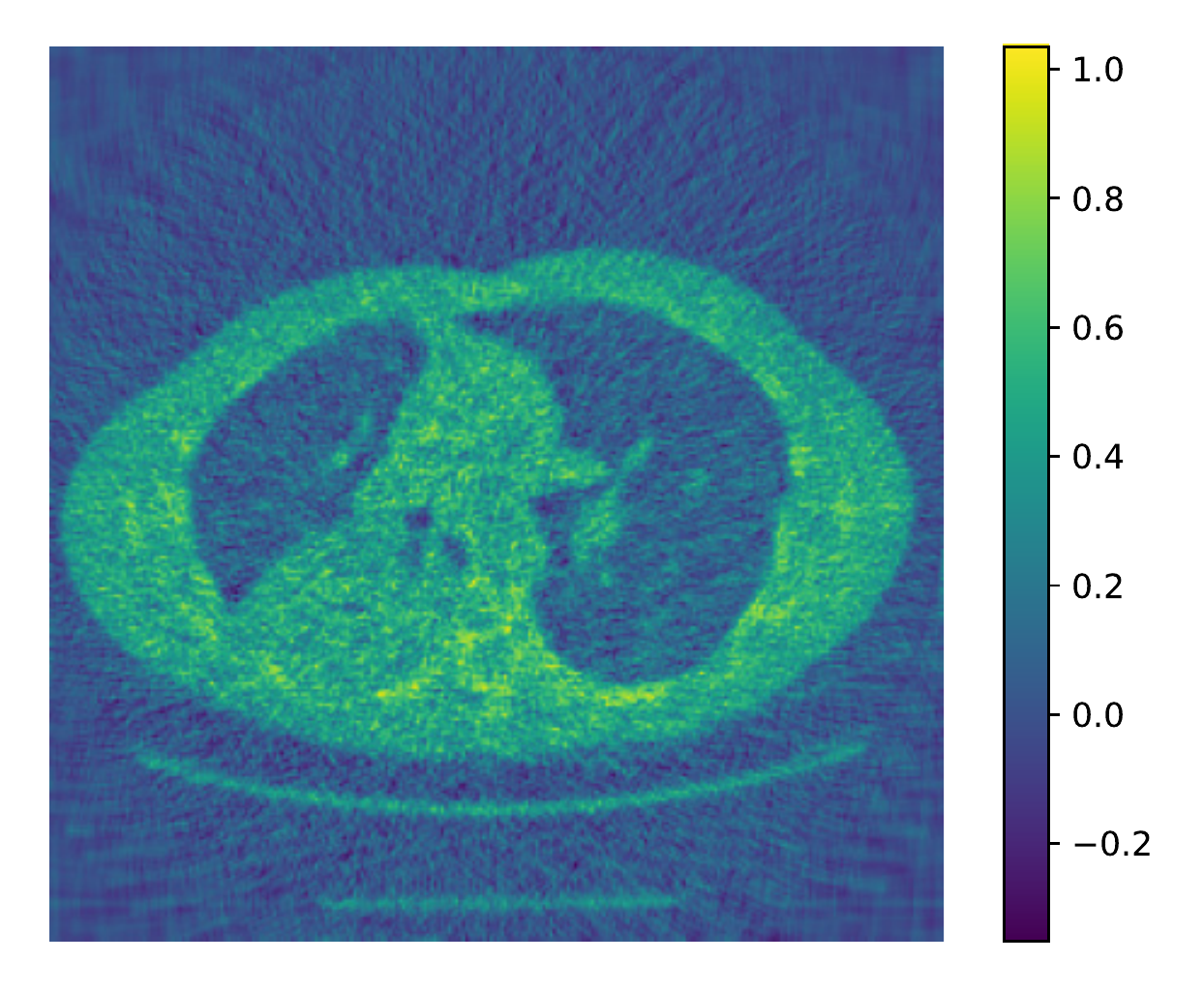}}}
	\caption{Example 3: myocardium imaging. (a) Attenuation image of myocardium; (b) measurement data with $5\%$ Gaussian noises; (c) reconstruction result by \ref{iterNETT3}; (d) reconstruction result by ART.} \label{Fig12}
\end{figure}

In addition, the values of PSNR and SSIM for these reconstruction results are listed in Table \ref{TabEX3}. It shows that the \ref{iterNETT3} algorithm provides much better solutions in both lung imaging and myocardium imaging, removing artifacts in the reconstructed images and achieving higher values of PSNR and SSIM. In practical applications, one can further improve the performance of \ref{iterNETT3} by re-training the convex neural network $\Phi^c_{\Theta}$ according to the types of attenuation images to be reconstructed, e.g. using images of lungs to train $\Phi^c_{\Theta}$ in lung imaging and using images of myocardia to train $\Phi^c_{\Theta}$ in myocardium imaging. In this example, we use the convex U-net trained on synthetic phantoms for the reconstruction of lung and myocardium images, which is partly due to the lack of appropriate training data, and partly to illustrate the generalization ability of the \ref{iterNETT3} algorithm.

	\begin{table}[H]
		\centering
		\begin{tabular}{ccc | cc} 
			\toprule
			\multirow{2}{*}{} & \multicolumn{2}{c}{\begin{tabular}[c]{@{}c@{}} lung imaging \end{tabular} }  & \multicolumn{2}{c}{\begin{tabular}[c]{@{}c@{}} myocardium imaging \end{tabular}}       \\ 
			\cmidrule(lr){2-3}\cmidrule(lr){4-5}
			\hspace{4pt}           &            iNETT  &  ART    \hspace{2.5pt}      &      \hspace{2.5pt}      iNETT  &  ART                                           \\ 
			\midrule
			PSNR      \hspace{4pt}        &            \textbf{20.02}   &  16.46    \hspace{2.5pt}      &      \hspace{2.5pt}        \textbf{19.95} & 16.08                                              \\
			SSIM      \hspace{4pt}        &             \textbf{0.70}    &  0.51       \hspace{2.5pt}      &      \hspace{2.5pt}         \textbf{0.69}  & 0.46                                                  \\
			\bottomrule
		\end{tabular}\caption{(Example 3) List of PSNR and SSIM for the reconstruction results.}
		\label{TabEX3}
	\end{table}

\section{Conclusion}\label{sec:conclusions}
We have proposed a non-stationary iterated network Tikhonov (iNETT) method for the solution of ill-posed inverse problems. Motivated by the network Tikhonov (NETT) method \cite{li2020nett}, we employ a data-driven regularization term including a neural network architecture, where the neural network is trained to penalize artifacts in the recovered solution. The iNETT method is then developed by combining the strategy of neural-network based regularization with the non-stationary iterated Tikhonov method. The main advantage of iNETT is that it avoids the difficult task of estimating the optimal regularization parameter, while keeps the virtue of the data-driven regularizer in NETT. We provide numerical examples to illustrate applications of iNETT in 2D computerized tomography. It shows that the iNETT method can achieve comparable (if not better) results as NETT without tuning the regularization parameter exhaustively.

To achieve the theoretical convergence of iNETT, we require that the neural-network based regularizer in iNETT is uniformly convex. As a result, we develop theories and algorithms for the construction of convex and uniformly convex neural networks. Given a general neural network, we prescribe sufficient conditions to achieve a related neural network which is component-wise convex or uniformly convex. Our formulation can embrace state-of-the-art convolutional neural networks. As a concrete example, we provide rigorous formulas for the U-net architecture, and explain the approaches to obtain convex and uniformly convex U-net architectures, which are successfully used in iNETT for the computations of 2D computerized tomography. The tool of convex and uniformly convex neural networks shall have many interesting applications in the future study.

For the iNETT method itself, there are several directions to further improve the algorithm. (1) Study more efficient neural network architectures with different data-fidelity norms, which can be more stable under random noises with non-Gaussian distributions, e.g. the $\ell^p$-$\ell^q$ norm regularization \cite{buccini2020modulus}. \correct{(2) Test alternative training strategies for the uniformly convex neural network regularizer, such as the adversarial strategy in \cite{lunz2018adversarial,mukherjee2020learned}.} (3) Learn the sequence $\{\alpha_n\}_n$ in iNETT, mimicking  the procedure developed in \cite{bertocchi2020deep}. (4) Devise nonlocal differential-like operators  in the regularization term to enforce some a-priori knowledge on the recovered solution \cite{gilboa2009nonlocal,bianchi2021graph,bianchi2022graph}. \correct{(5) Generalize iNETT to a fully learned network architecture, in the same spirit of \cite{sun2016deep,zhang2018ista}, e.g. learn the neural network regularizer $\Reg(\bx)$ in the optimization process.}

\section*{Acknowledgments}
\correct{Davide Bianchi is supported by NSFC (grant no. 12250410253). Wenbin Li is supported by Natural Science Foundation of Shenzhen (grant no. JCYJ20190806144005645) and NSFC (grant no. 41804096).}

\newpage
\printbibliography

@article{liwanfan22,
  title={A stochastic gradient descent approach with partitioned-truncated singular value decomposition for large-scale inverse problems of magnetic modulus data},
  author={Li, Wenbin and Wang, Kangzhi and Fan, Tingting},
  journal={Inverse Problems},
  volume={38},
  number={7},
  pages={075002},
  year={2022},
  publisher={IOP Publishing}
}

@article{gorbenher70,
  title={Algebraic reconstruction techniques ({ART}) for three-dimensional electron microscopy and {X}-ray photography},
  author={Gordon, Richard and Bender, Robert and Herman, Gabor T},
  fjournal={Journal of theoretical Biology},
  journal={J. Theoret. Biol.},
  volume={29},
  number={3},
  pages={471--481},
  year={1970},
  publisher={Academic Press},
  doi={10.1016/0022-5193(70)90109-8}
}

@InProceedings{safsha20,
	title = 	 {How Good is {SGD} with Random Shuffling?},
	author =       {Safran, Itay and Shamir, Ohad},
	booktitle = 	 {Proceedings of Thirty Third Conference on Learning Theory},
	pages = 	 {3250--3284},
	year = 	 {2020},
	editor = 	 {Abernethy, Jacob and Agarwal, Shivani},
	volume = 	 {125},
	series = 	 {Proceedings of Machine Learning Research},
	publisher =    {PMLR}
}

@misc{kinba14,
  title={Adam: A method for stochastic optimization},
  author={Kingma, Diederik P and Ba, Jimmy},
  eprint={1412.6980},
  year={2014},
  archivePrefix={arXiv},
  primaryClass={cs.LG}
}

@article{li2020nett,
	title={NETT: Solving inverse problems with deep neural networks},
	author={Li, Housen and Schwab, Johannes and Antholzer, Stephan and Haltmeier, Markus},
	fjournal={Inverse Problems},
	journal={Inverse Problems},
	volume={36},
	number={6},
	pages={065005},
	year={2020},
	publisher={IOP Publishing},
	doi = {10.1088/1361-6420/ab6d57},
}

@article{jin2014nonstationary,
	title={Nonstationary iterated Tikhonov regularization in Banach spaces with uniformly convex penalty terms},
	author={Jin, Qinian and Zhong, Min},
    fjournal={Numerische Mathematik},
    journal={Numer. Math.},
	volume={127},
	number={3},
	pages={485--513},
	year={2014},
	publisher={Springer},
	doi = {10.1007/s00211-013-0594-9},
}

@article{aggarwal2018modl,
	title={MoDL: Model-based deep learning architecture for inverse problems},
	author={Aggarwal, Hemant K and Mani, Merry P and Jacob, Mathews},
	fjournal={IEEE transactions on medical imaging},
	journal={IEEE Trans. Med. Imag.},
	volume={38},
	number={2},
	pages={394--405},
	year={2018},
	publisher={IEEE},
	doi = {10.1109/TMI.2018.2865356},
}

@article{sun2016deep,
	title={Deep ADMM-Net for compressive sensing MRI},
	author={Sun, Jian and Li, Huibin and Xu, Zongben and others},
	journal={Advances in neural information processing systems},
	volume={29},
	year={2016}
}

@book{scherzer2009variational,
	title={Variational methods in imaging},
	author={Scherzer, Otmar and Grasmair, Markus and Grossauer, Harald and Haltmeier, Markus and Lenzen, Frank},
	year={2009},
	publisher={Springer},
	doi = {10.1007/978-0-387-69277-7}
}

@article{jin2012nonstationary,
	title={Nonstationary iterated Tikhonov regularization for ill-posed problems in Banach spaces},
	author={Jin, Qinian and Stals, Linda},
	fjournal={Inverse Problems},
	journal={Inverse Problems},
	volume={28},
	number={10},
	pages={104011},
	year={2012},
	publisher={IOP Publishing},
	doi={10.1088/0266-5611/28/10/104011}
}

@book{engl1996regularization,
	title={Regularization of inverse problems},
	author={Engl, Heinz Werner and Hanke, Martin and Neubauer, Andreas},
	volume={375},
	year={1996},
	publisher={Springer Science \& Business Media}
}

@article{hanke1998nonstationary,
	title={Nonstationary iterated Tikhonov regularization},
	author={Hanke, Martin and Groetsch, Charles W},
	fjournal={Journal of Optimization Theory and Applications},
	journal={J. Optim. Theory Appl.},
	volume={98},
	number={1},
	pages={37--53},
	year={1998},
	publisher={Springer},
	doi={10.1023/A:1022680629327}
}

@article{engl1987choice,
	title={On the choice of the regularization parameter for iterated Tikhonov regularization of III-posed problems},
	author={Engl, Heinz W},
	fjournal={Journal of approximation theory},
	journal={J. Approx. Theory},
	volume={49},
	number={1},
	pages={55--63},
	year={1987},
	publisher={Elsevier},
	doi={10.1016/0021-9045(87)90113-4}
}

@article{buccini2017iterated,
	title={Iterated Tikhonov regularization with a general penalty term},
	author={Buccini, Alessandro and Donatelli, Marco and Reichel, Lothar},
	fjournal={Numerical Linear Algebra with Applications},
	journal={Numer. Linear Algebra Appl.},
	volume={24},
	number={4},
	pages={e2089},
	year={2017},
	publisher={Wiley Online Library},
	doi={10.1002/nla.2089}
}

@article{bianchi2017generalized,
	title={On generalized iterated Tikhonov regularization with operator-dependent seminorms},
	author={Bianchi, Davide and Donatelli, Marco},
	year={2017},
	fjournal ={Electronic Transactions on Numerical Analysis},
	journal={ETNA},
	volume={47},
	pages={73--99}
}

@article{bianchi2015iterated,
	title={Iterated fractional Tikhonov regularization},
	author={Bianchi, Davide and Buccini, Alessandro and Donatelli, Marco and Serra-Capizzano, Stefano},
	fjournal={Inverse Problems},
	journal={Inverse Problems},
	volume={31},
	number={5},
	pages={055005},
	year={2015},
	publisher={IOP Publishing},
	doi={10.1088/0266-5611/31/5/055005}
}

@article{hansen1992analysis,
	title={Analysis of discrete ill-posed problems by means of the L-curve},
	author={Hansen, Per Christian},
	fjournal={SIAM review},
	journal={SIAM Rev.},
	volume={34},
	number={4},
	pages={561--580},
	year={1992},
	publisher={SIAM},
	doi={10.1137/1034115}
}

@article{reichel2013old,
	title={Old and new parameter choice rules for discrete ill-posed problems},
	author={Reichel, Lothar and Rodriguez, Giuseppe},
	fjournal={Numerical Algorithms},
	journal={Numer. Algorithms},
	volume={63},
	number={1},
	pages={65--87},
	year={2013},
	publisher={Springer},
	doi={10.1007/s11075-012-9612-8}
}

@article{hanke1996limitations,
	title={Limitations of the L-curve method in ill-posed problems},
	author={Hanke, Martin},
	fjournal={BIT Numerical Mathematics},
	journal={BIT},
	volume={36},
	number={2},
	pages={287--301},
	year={1996},
	publisher={Springer},
	doi={10.1007/BF01731984}
}

@book{rockafellar2009variational,
	title={Variational analysis},
	author={Rockafellar, R Tyrrell and Wets, Roger J-B},
	series={Grundlehren der mathematischen Wissenschaften},
	volume={317},
	year={1998},
	edition={Corrected 3rd printing 2009},
	publisher={Springer, Berlin, Heidelberg},
	doi={10.1007/978-3-642-02431-3}
}

@article{chen2017low,
	title={Low-dose CT via convolutional neural network},
	author={Chen, Hu and Zhang, Yi and Zhang, Weihua and Liao, Peixi and Li, Ke and Zhou, Jiliu and Wang, Ge},
	fjournal={Biomedical optics express},
	journal={Biomed. Opt. Express},
	volume={8},
	issue={2},
	pages={679--694},
	year={2017},
	publisher={Optical Society of America},
	doi={10.1364/BOE.8.000679}
}

@article{donatelli2012nondecreasing,
	title={On nondecreasing sequences of regularization parameters for nonstationary iterated Tikhonov},
	author={Donatelli, Marco},
	fjournal={Numerical Algorithms},
	journal={Numer. Algor.},
	volume={60},
	pages={651--668},
	year={2012},
	doi={10.1007/s11075-012-9593-7},
	publisher={Springer}
}

@article{morotti2021green,
	title={A Green Prospective for Learned Post-Processing in Sparse-View Tomographic Reconstruction},
	author={Morotti, Elena and Evangelista, Davide and Loli Piccolomini, Elena},
	fjournal={Journal of Imaging},
	journal={J. Imaging},
	volume={7},
	issue={8},
	number={139},
	pages={1--14},
	year={2021},
	doi={10.3390/jimaging7080139},
	publisher={Multidisciplinary Digital Publishing Institute}
}

@article{bianchi2022graph,
	title={Graph Laplacian for image deblurring},
	author={Bianchi, Davide and Buccini, Alessandro and Donatelli, Marco and Randazzo, Emma},
	fjournal={Electronic Transaction on Numerical Analysis},
	volume={55},
	pages={169--186},
	journal={ETNA},
	year={2022},
	doi={10.1553/etna_vol55s169}
}

@article{higham2019deep,
	title={Deep Learning: An Introduction for Applied Mathematicians},
	author={Higham, Catherine F and Higham, Desmond J},
	fjournal={SIAM Review},
	journal={SIAM Rev.},
	volume={61},
	number={4},
	pages={860--891},
	year={2019},
	publisher={SIAM},
	doi={10.1137/18M1165748}
}

@inproceedings{he2016deep,
	title={Deep residual learning for image recognition},
	author={He, Kaiming and Zhang, Xiangyu and Ren, Shaoqing and Sun, Jian},
	fbooktitle={Proceedings of the IEEE conference on computer vision and pattern recognition},
	booktitle={Proc. IEEE Comput. Soc. Conf. Comput. Vis. Pattern Recognit. (CVPR)},
	pages={770--778},
	year={2016}
}

@inproceedings{huang2017densely,
	title={Densely connected convolutional networks},
	author={Huang, Gao and Liu, Zhuang and Van Der Maaten, Laurens and Weinberger, Kilian Q},
	fbooktitle={Proceedings of the IEEE conference on computer vision and pattern recognition},
	booktitle={Proc. IEEE Comput. Soc. Conf. Comput. Vis. Pattern Recognit. (CVPR)},
	pages={4700--4708},
	year={2017}
}

@inproceedings{ronneberger2015u,
	title={U-net: Convolutional networks for biomedical image segmentation},
	author={Ronneberger, Olaf and Fischer, Philipp and Brox, Thomas},
	fbooktitle={International Conference on Medical image computing and computer-assisted intervention},
	booktitle={Proc. Med. Image Comput. Comput. Assist. Interv.},
	pages={234--241},
	year={2015},
	organization={Springer},
	doi={10.1007/978-3-319-24574-4_28}
}

@book{zalinescu2002convex,
	title={Convex analysis in general vector spaces},
	author={Z\u{a}linescu, Constantin},
	year={2002},
	publisher={World scientific}
}

@article{zalinescu1983uniformly,
	title={On uniformly convex functions},
	author={Z\u{a}linescu, Constantin},
	fjournal={Journal of Mathematical Analysis and Applications},
	journal={J. Math. Anal. Appl.},
	volume={95},
	number={2},
	pages={344--374},
	year={1983},
	publisher={Elsevier},
	doi={10.1016/0022-247X(83)90112-9}
}

@InProceedings{amos2017input,
	title={Input Convex Neural Networks},
	author={Amos, Brandon and Xu, Lei and Kolter, J Zico},
	booktitle={Proceedings of the 34th International Conference on Machine Learning},
	pages={146--155},
	year={2017},
	editor={Precup, Doina and Teh, Yee Whye},
	volume={70},
	series={Proceedings of Machine Learning Research},
	publisher={PMLR}
}

@article{borwein2009uniformly,
	title={Uniformly convex functions on Banach spaces},
	author={Borwein, Jonathan and Guirao, A and H{\'a}jek, Petr and Vanderwerff, Jon},
	fjournal={Proceedings of the American Mathematical Society},
	journal={Proc. Amer. Math. Soc.},
	volume={137},
	number={3},
	pages={1081--1091},
	year={2009},
	doi={10.1090/S0002-9939-08-09630-5}
	}

@article{arridge2019solving,
	title={Solving inverse problems using data-driven models},
	author={Arridge, Simon and Maass, Peter and {\"O}ktem, Ozan and Sch{\"o}nlieb, Carola-Bibiane},
	fjournal={Acta Numerica},
	journal={Acta Numer.},
	volume={28},
	pages={1--174},
	year={2019},
	publisher={Cambridge University Press},
	doi={10.1017/S0962492919000059}
}

@article{antholzer2021discretization,
	title={Discretization of learned NETT regularization for solving inverse problems},
	author={Antholzer, Stephan and Haltmeier, Markus},
	fjournal={Journal of Imaging},
	journal={J. Imaging},
	volume={7},
	issue={11},
	number={239},
	pages={1--13},
	year={2021},
	publisher={Multidisciplinary Digital Publishing Institute},
	doi={10.3390/jimaging7110239}
}

@inproceedings{antholzer2019nett,
	title={NETT regularization for compressed sensing photoacoustic tomography},
	author={Antholzer, Stephan and Schwab, Johannes and Bauer-Marschallinger, Johnnes and Burgholzer, Peter and Haltmeier, Markus},
	booktitle={Photons Plus Ultrasound: Imaging and Sensing 2019},
	volume={10878},
	pages={272--282},
	year={2019},
	organization={SPIE},
	doi={10.1117/12.2508486}
}

@article{obmann2021augmented,
	title={Augmented NETT regularization of inverse problems},
	author={Obmann, Daniel and Nguyen, Linh and Schwab, Johannes and Haltmeier, Markus},
	fjournal={Journal of Physics Communications},
	journal={J. Phys. Commun.},
	volume={5},
	number={10},
	pages={105002},
	year={2021},
	publisher={IOP Publishing},
	doi={10.1088/2399-6528/ac26aa}
}

@inproceedings{sivaprasad2021curious,
	title={The curious case of convex neural networks},
	author={Sivaprasad, Sarath and Singh, Ankur and Manwani, Naresh and Gandhi, Vineet},
	booktitle={Joint European Conference on Machine Learning and Knowledge Discovery in Databases},
	pages={738--754},
	year={2021},
	organization={Springer},
	doi={10.1007/978-3-030-86486-6_45}
}

@article{ao2021data,
	title={A data and knowledge driven approach for SPECT using convolutional neural networks and iterative algorithms},
	author={Ao, Wenqi and Li, Wenbin and Qian, Jianliang},
	fjournal={Journal of Inverse and Ill-posed Problems},
	journal={J. Inverse Ill-Posed Probl.},
	volume={29},
	number={4},
	pages={543--555},
	year={2021},
	publisher={De Gruyter},
	doi={10.1515/jiip-2020-0056}
}

@book{andrews1977digital,
	title={Digital image restoration},
	author={Andrews, Harry C and Hunt, Boby Ray},
	year={1977},
	publisher={Prentice-Hall}
}

@inproceedings{ioffe2015batch,
	title={Batch normalization: Accelerating deep network training by reducing internal covariate shift},
	author={Ioffe, Sergey and Szegedy, Christian},
	booktitle={International conference on machine learning},
	pages={448--456},
	year={2015},
	organization={PMLR}
}

@book{lindenstrauss2013classical,
	title={Classical Banach spaces II: function spaces},
	author={Lindenstrauss, Joram and Tzafriri, Lior},
	volume={97},
	year={2013},
	publisher={Springer Science \& Business Media}
}

@misc{tan2022data,
	title={Data-Driven Mirror Descent with Input-Convex Neural Networks},
	author={Tan, Hong Ye and Mukherjee, Subhadip and Tang, Junqi and Sch{\"o}nlieb, Carola-Bibiane},
	eprint={2206.06733},
	year={2022},
	archivePrefix={arXiv},
	primaryClass={math.OC}
}

@article{osher2005iterative,
	title={An iterative regularization method for total variation-based image restoration},
	author={Osher, Stanley and Burger, Martin and Goldfarb, Donald and Xu, Jinjun and Yin, Wotao},
	fjournal={Multiscale Modeling \& Simulation},
	journal={Multiscale Model. Simul.},
	volume={4},
	number={2},
	pages={460--489},
	year={2005},
    doi={10.1137/040605412}
}

@article{bachmayr2009iterative,
	title={Iterative total variation schemes for nonlinear inverse problems},
	author={Bachmayr, Markus and Burger, Martin},
	fjournal={Inverse Problems},
	journal={Inverse Problems},
	volume={25},
	pages={105004},
	year={2009},
	doi={10.1088/0266-5611/25/10/105004},
	publisher={IOP Publishing}
}

@article{bertocchi2020deep,
	title={Deep unfolding of a proximal interior point method for image restoration},
	author={Bertocchi, Carla and Chouzenoux, Emilie and Corbineau, Marie-Caroline and Pesquet, Jean-Christophe and Prato, Marco},
	fjournal={Inverse Problems},
	journal={Inverse Problems},
	volume={36},
	number={3},
	pages={034005},
	year={2020},
	doi={10.1088/1361-6420/ab460a},
	publisher={IOP Publishing}
}

@misc{mukherjee2020learned,
	title={Learned convex regularizers for inverse problems},
	author={Mukherjee, Subhadip and Dittmer, S{\"o}ren and Shumaylov, Zakhar and Lunz, Sebastian and {\"O}ktem, Ozan and Sch{\"o}nlieb, Carola-Bibiane},
	eprint={2008.02839},
	year={2020},
	archivePrefix={arXiv},
	primaryClass={cs.LG}
}

@article{heinrich2018residual,
	title={Residual U-net convolutional neural network architecture for low-dose CT denoising},
	author={Heinrich, Mattias P and Stille, Maik and Buzug, Thorsten M},
	fjournal={Current Directions in Biomedical Engineering},
	journal={Curr. Dir. Biomed. Eng.},
	volume={4},
	number={1},
	pages={297--300},
	year={2018},
	publisher={De Gruyter},
	doi={10.1515/cdbme-2018-0072}
}

@article{wei2018deep,
	title={Deep-learning schemes for full-wave nonlinear inverse scattering problems},
	author={Wei, Zhun and Chen, Xudong},
	fjournal={IEEE Transactions on Geoscience and Remote Sensing},
	journal={IEEE Trans. Geosci. Remote Sensing},
	volume={57},
	number={4},
	pages={1849--1860},
	year={2018},
	publisher={IEEE},
	doi={10.1109/TGRS.2018.2869221}
}

@article{antholzer2019deep,
	title={Deep learning for photoacoustic tomography from sparse data},
	author={Antholzer, Stephan and Haltmeier, Markus and Schwab, Johannes},
	fjournal={Inverse problems in science and engineering},
	journal={Inverse Probl. Sci. Eng.},
	volume={27},
	number={7},
	pages={987--1005},
	year={2019},
	publisher={Taylor \& Francis},
	doi={10.1080/17415977.2018.1518444}
}

@article{obmann2020deep,
	title={Deep synthesis network for regularizing inverse problems},
	author={Obmann, Daniel and Schwab, Johannes and Haltmeier, Markus},
	fjournal={Inverse Problems},
	journal={Inverse Problems},
	volume={37},
	number={1},
	pages={015005},
	year={2020},
	publisher={IOP Publishing},
	doi={10.1088/1361-6420/abc7cd}
}

@misc{dumoulin2016guide,
	title={A guide to convolution arithmetic for deep learning},
	author={Dumoulin, Vincent and Visin, Francesco},
	eprint={1603.07285},
	archivePrefix={arXiv},
	primaryClass={stat.ML},
	year={2016}
}

@article{buccini2020modulus,
	title={Modulus-based iterative methods for constrained $\ell^p$--$\ell^q$ minimization},
	author={Buccini, Alessandro and Pasha, Mirjeta and Reichel, Lothar},
	fjournal={Inverse Problems},
	journal={Inverse Problems},
	volume={36},
	number={8},
	pages={084001},
	year={2020},
	publisher={IOP Publishing},
	doi={10.1088/1361-6420/ab9f86}
}

@inproceedings{bianchi2021graph,
	title={Graph approximation and generalized Tikhonov regularization for signal deblurring},
	author={Bianchi, Davide and Donatelli, Marco},
	booktitle={2021 21st International Conference on Computational Science and Its Applications (ICCSA)},
	pages={93--100},
	year={2021},
	organization={IEEE},
	doi={10.1109/ICCSA54496.2021.00023}
}

@article{gilboa2009nonlocal,
	title={Nonlocal operators with applications to image processing},
	author={Gilboa, Guy and Osher, Stanley},
	fjournal={Multiscale Modeling \& Simulation},
	journal={Multiscale Model. Simul.},
	volume={7},
	number={3},
	pages={1005--1028},
	year={2009},
	publisher={SIAM},
	doi={10.1137/070698592}
}

@article{cai2016regularization,
	title={Regularization preconditioners for frame-based image deblurring with reduced boundary artifacts},
	author={Cai, Yuantao and Donatelli, Marco and Bianchi, Davide and Huang, Ting-Zhu},
	fjournal={SIAM Journal on Scientific Computing},
	journal={SIAM J. Sci. Comput.},
	volume={38},
	number={1},
	pages={B164--B189},
	year={2016},
	publisher={SIAM},
	doi={10.1137/140976261}
}

@inproceedings{zhang2018ista,
	title={ISTA-Net: Interpretable optimization-inspired deep network for image compressive sensing},
	author={Zhang, Jian and Ghanem, Bernard},
	booktitle={Proceedings of the IEEE conference on computer vision and pattern recognition},
	pages={1828--1837},
	year={2018}
}

@article{lunz2018adversarial,
	title={Adversarial regularizers in inverse problems},
	author={Lunz, Sebastian and {\"O}ktem, Ozan and Sch{\"o}nlieb, Carola-Bibiane},
	journal={Advances in neural information processing systems},
	volume={31},
	year={2018}
}
\end{document}